\documentclass[a4paper,oneside,reqno]{amsart}

\usepackage{amssymb,amsthm,amsmath}
\usepackage{bbm}
\usepackage{enumitem}
\usepackage{microtype}
\usepackage{geometry}
\usepackage{hyperref}
\usepackage{tikz}
\usepackage{mathtools}
\usepackage{stmaryrd} 
\usepackage{thmtools}
\usepackage{thm-restate}
\usepackage{enumitem}
\usepackage[utf8]{inputenc}
\usepackage[T1]{fontenc}
\usepackage{lmodern}
\usepackage[normalem]{ulem}

\declaretheorem[numberwithin=section]{theorem}
\declaretheorem[sibling=theorem]{lemma}
\declaretheorem[sibling=theorem]{claim}

\declaretheorem[sibling=theorem]{proposition}
\declaretheorem[style=definition,sibling=theorem]{definition}
\declaretheorem[style=remark,sibling=theorem]{remark}
\declaretheorem[style=definition,sibling=theorem]{example}

\renewcommand{\le}{\leqslant}
\renewcommand{\ge}{\geqslant}
\renewcommand{\leq}{\leqslant}
\renewcommand{\geq}{\geqslant}
\renewcommand{\emptyset}{\varnothing}
\renewcommand{\Pr}{\mathbb{P}}

\DeclareMathOperator{\e}{\mathbb{E}} 

\DeclareMathOperator{\Ex}{\mathbb{E}} 
\DeclareMathOperator{\Var}{Var}
\DeclareMathOperator{\Po}{Po}
\DeclareMathOperator{\codim}{codim}
\DeclareMathOperator{\UT}{UT}
\DeclareMathOperator{\Clique}{Clique}
\DeclareMathOperator{\Gmin}{Near-min}
\DeclareMathOperator{\Hub}{Hub}

\DeclareMathOperator{\Aut}{Aut}
\DeclareMathOperator{\Emb}{Emb}
\DeclareMathOperator{\Hom}{Hom}
\newcommand{\1}{\mathbbm{1}}
\newcommand{\Ber}{\mathrm{Ber}}
\newcommand{\kap}{\text{$k$-AP}}

\DeclarePairedDelimiter{\ceil}{\lceil}{\rceil}
\DeclarePairedDelimiter{\floor}{\lfloor}{\rfloor}

\newcommand{\eps}{\varepsilon}

\newcommand{\br}[1]{\llbracket{#1}\rrbracket}

\newcommand{\ff}[2]{{#1}^{\underline{#2}}}

\newcommand{\barX}{\bar X_r}

\newcommand{\one}{\mathbf{1}}

\newcommand{\cA}{\mathcal{A}}
\newcommand{\cB}{\mathcal{B}}
\newcommand{\cC}{\mathcal{C}}
\newcommand{\cD}{\mathcal{D}}
\newcommand{\cE}{\mathcal{E}}
\newcommand{\cF}{\mathcal{F}}

\newcommand{\cI}{\mathcal{I}}

\newcommand{\cS}{\mathcal{S}}
\newcommand{\cX}{\mathcal{X}}
\newcommand{\cH}{\mathcal{H}}

\newcommand{\NN}{\mathbb{N}}

\newcommand{\RR}{\mathbb{R}}
\newcommand{\ZZ}{\mathbb{Z}}

\newcommand{\core}{\cI^*}
\newcommand{\mincore}{\mathcal{J}^*}

\newcommand{\mmin}{m_{\min}}
\newcommand{\mmax}{m_{\max}}

\newcommand{\Gcore}{{G^*}}
\newcommand{\Gcoreb}{{G^*_{\mathrm{exc}}}}
\newcommand{\Gcoree}{{G^*_{\mathrm{bad}}}}
\newcommand{\Gcoreh}{{G^*_{\mathrm{high}}}}
\newcommand{\mb}{m_{\textrm{exc}}}
\newcommand{\me}{m_{\textrm{bad}}}
\newcommand{\mh}{m_{\textrm{high}}}

\newcommand{\bd}{\mathbf{d}}
\newcommand{\bdmax}{\mathbf{d}_{\max}}

\newcommand{\sigmab}{{\bar{\sigma}}}

\newcommand{\QH}{{\mathcal{Q}_H}}

\title[Upper tails via high moments and entropic stability]{Upper tails via high moments and entropic stability}

\author{Matan Harel}
\address{Matan Harel, Department of Mathematics, Northeastern University, Boston, MA 02115, USA}
\email{m.harel@northeastern.edu}
\author{Frank Mousset}
\address{Frank Mousset, School of Mathematical Sciences, Tel Aviv University, Tel Aviv 6997801, Israel}
\email{moussetfrank@gmail.com}

\author{Wojciech Samotij}
\address{Wojciech Samotij, School of Mathematical Sciences, Tel Aviv University, Tel Aviv 6997801, Israel}
\email{samotij@tauex.tau.ac.il}

\thanks{This research is supported in part by the Israel Science Foundation grants 1147/14~(FM and WS), 1028/16~(FM), and 1145/18~(FM and WS), the ERC Starting Grant 633509 (FM) and 678520 (MH), and the Zuckerman Postdoctoral Fellowship Program (MH)}

\date{\today}

\begin{document}

\begin{abstract}
  Suppose that $X$ is a bounded-degree polynomial with nonnegative
  coefficients on the $p$-biased
  discrete hypercube. Our main result gives sharp estimates on the
  logarithmic upper tail probability of $X$ whenever an associated extremal problem
  satisfies a certain entropic stability property. We apply this result
  to solve two long-standing open problems in probabilistic
  combinatorics: the upper tail problem for the number of arithmetic
  progressions of a fixed length in the $p$-random subset of the integers
  and the upper tail problem for the number of cliques of a fixed size in
  the random graph $G_{n,p}$. 
  We also make significant progress on the upper tail
  problem for the number of copies of a fixed 
  regular graph $H$ in $G_{n,p}$. To accommodate readers who are interested in
  learning the basic method, we include a short, self-contained solution to the upper
  tail problem for the number of triangles in $G_{n,p}$ for all $p=p(n)$
  satisfying $n^{-1}\log n\ll p
  \ll 1$.
\end{abstract}

\maketitle

\setcounter{tocdepth}{1}
\tableofcontents

\section{Introduction}

Suppose that $Y = (Y_1,\dotsc,Y_N)$ is a sequence of independent Bernoulli random variables with mean $p$
and that $X=X(Y)$ is an $N$-variate polynomial with nonnegative real
coefficients. Perhaps the simplest question that can be asked about the
typical behaviour of $X$ is whether it satisfies a law of large numbers,
i.e., whether $X\to \e[X]$ in probability as $N\to\infty$.
Once this is established, it is natural to ask for quantitative estimates
of the probability of the event that $X$ differs from its mean
by a significant amount.
In the special case where $Y\mapsto X(Y)$ is a linear
function, this problem is addressed by the classical theory of large
deviations, see~\cite{boucheron2013book, dembo1998book}. This theory shows
that, under mild conditions
on the coefficients of the linear function $X$ and the assumptions $p\to 0$ and $Np\to \infty$,
\[ \Pr\big(|X-\e[X]|\geq \delta \e[X]\big)
 = e^{- (I(\delta) +o(1)) Np}\]
 for an explicitly computable function $I \colon (0, \infty) \to (0, \infty]$.
However, there are
many natural situations where one would like to consider nonlinear polynomials
$X(Y)$, as in the following two examples. We use the notation
$\br N = \{1,\dotsc,N\}$.

\begin{example}[Arithmetic progressions in random sets of integers]
  \label{ex:kap}
  A \emph{$k$-term arithmetic progression} is a sequence of $k$ integers of the
  form $\big(a,a+b,a+2b,\dotsc,a+(k-1)b\big)$, where $b > 0$. We write
  $\br N_p$ for the random subset of $\br N$ obtained by including
  every element of $\br N$ independently with probability $p$. Let $X_{N,p}^\kap$
  denote the number of $k$-term arithmetic
  progressions in $\br N_p$. Then $X_{N,p}^\kap$ can be considered as a polynomial
  with nonnegative coefficients and degree $k$ in the independent
  $\Ber(p)$ random variables $Y_1,\dotsc,Y_N$,
  where $Y_i$ is the indicator variable of the event that $i\in \br N_p$; explicitly,
  \[
    X_{N,p}^\kap = \sum_{b> 0}\sum_{\substack{ a\geq 1\\a+(k-1)b\leq N}}
    \prod_{i=0}^{k-1} Y_{a+ib}.
  \]
  We remark that, unlike~\cite{bhattacharya2016upper} and several other works, we count only
  genuine arithmetic progressions (i.e., we do not consider degenerate progressions of
  the form $(a,\dotsc,a)$) and we count every progression only once (as opposed
  to counting $(a_1,a_2,\dotsc,a_k)$ and
  $(a_k,a_{k-1},\dotsc,a_1)$ as two different progressions).
\end{example}

\begin{example}[Subgraph counts in random graphs] \label{ex:subgraph}
  Fix a nonempty graph $H$ and let $X^H_{n,p}$ be the number of copies of $H$
  in the random graph $G_{n,p}$, that is, the number of subgraphs of $G_{n,p}$
  isomorphic to $H$.
  Then $X_{n,p}^H$ can be written as a polynomial with nonnegative
  coefficients and degree $e_H$ in the 
  $N = \binom{n}{2}$ indicator random variables of
  the possible edges of $G_{n,p}$. More precisely, 
  fix an arbitrary bijection $\sigma_n\colon \binom{\br{n}}{2}\to \br{N}$
  (the precise choice is irrelevant) and, for every $i\in \br N$, let $Y_i$ 
  be the indicator variable of the event that $\sigma_n^{-1}(i)$ is an edge
  in $G_{n,p}$. Then
  $Y_1,\dotsc,Y_N$ are independent $\Ber(p)$ random variables and
  \[ X_{n,p}^H = \sum_{\substack{H'\subseteq K_n\\ H' \cong H}}
  \prod_{e\in E(H')} Y_{\sigma_n(e)}, \] where $K_n$ denotes the complete
  graph on the vertex set $\br{n}$ and ${\cong}$ denotes the isomorphism
  relation on graphs.
\end{example}

In this paper, we will always assume that $\delta$ is fixed and $p\to 0$ as $N\to \infty$.

The large deviation problem for the variables described above
is significantly more involved than the linear case;
in particular, the lower and upper tail probabilities---that is, $\Pr\big(X
\leq (1-\delta) \e[X]\big)$ and 
$\Pr\big(X \geq (1+\delta) \e[X]\big)$, respectively---exhibit dramatically different behaviours.
On the one hand, using a combination of Harris's inequality~\cite{Har60} and
Janson's inequality~\cite{janson1990poisson}, one can show that $X = X_{N,p}^\kap$ satisfies
\begin{equation}\label{eq:lowertail} 
e^{-C_1(\delta) \min{\{\e[X],Np\}}}
\leq \Pr\big(X\leq (1-\delta)\e[X]\big)
\leq 
e^{-C_2(\delta) \min{\{\e[X],Np\}}}
\end{equation}
for some positive $C_1(\delta)$ and $C_2(\delta)$;
similar bounds are available for $X = X_{n,p}^H$.\footnote{For more precise results, we refer the interested reader to \cite{janson2016lower,KozSam19,mousset2017probability,Zhao17}.}
On the other hand, there are no comparably simple tools that allow one to easily obtain
similar estimates on the logarithm of the upper tail
probability. The
standard concentration inequalities due to
Azuma--Hoeffding~\cite{hoeffding1963probability},
Talagrand~\cite{talagrand1995concentration}, or
Kim--Vu~\cite{kim2000concentration,vu2002concentration} yield bounds that are
far from tight in Examples~\ref{ex:kap} and~\ref{ex:subgraph}.
For a survey discussing these and other classical approaches to the `infamous upper tail'
problem, see~\cite{janson2002infamous}.

Unlike the lower tail, the upper tail is susceptible to the influence of small
structures whose appearance increases the value $X$ atypically, a phenomenon
that we refer to as \emph{localisation}. For example, in
the case of $X = X_{N,p}^\kap$ where $k\geq 3$, a typical subset of size $m=o(N)$ contains
$\Theta\big(N^2(m/N)^k\big) = o(m^2)$ $k$-term arithmetic progressions, whereas
some very rare
subsets (notably an interval of length~$m$) can contain as many as $\Theta(m^2)$ such progressions.
The event that $\br N_p$ contains an interval of length
$\Theta(\sqrt{\e[X]})$ thus provides a lower bound on the upper tail
probability. More precisely, $\Pr\big(X\geq (1+\delta)\e[X]\big) \geq
\exp\big(-O(\sqrt{\e[X]}\log(1/p))\big)$,
which is significantly larger than the lower tail probability
\eqref{eq:lowertail} for most $p$.
In order to properly analyse the upper tail event,
one must account for these local effects, which frequently requires understanding
the peculiar combinatorial nature of the random variable $X$. 

The last decade has seen the development of an increasingly powerful theory of
`nonlinear large deviations', which began with the work of Chatterjee--Dembo~\cite{chatterjee2016nonlinear} and was further developed by
Eldan~\cite{eldan2016gaussian}, Cook--Dembo~\cite{cook2018large},
Augeri~\cite{augeri2019transportation, augeri2018nonlinear}, and Cook--Dembo--Pham~\cite{cook2021hypergraph}. Whenever a general
function $f$ of i.i.d.\ random variables satisfies certain complexity and regularity
conditions, these results can be used to express the
upper tail probability $\Pr\big(f\geq (1+\delta)\e[f]\big)$ in terms
of an
associated variational problem. In the case where $f$ is a polynomial with
nonnegative coefficients on the hypercube, this variational problem is able to capture the
presence of localisation, if it occurs. In the two examples mentioned above,
nonlinear large deviation theory gives tight control of the upper tail probabilities whenever
$p\geq N^{-\alpha}$ for some constant $\alpha>0$. However, the best-known
constant $\alpha$ is not optimal in both examples.

Our main contribution is a general method for proving sharp bounds on the upper
tail probability of the polynomial $X = X(Y)$ in the presence of localisation. In many cases where localisation occurs, our
approach can also give a coarse description of the tail event. At the heart of our method lies an adaptation of the classical moment argument of Janson, Oleszkiewicz, and Ruci\'nski~\cite{janson2004upper}, which we use to formalise the intuition that the upper tail event is dominated by the appearance of near-minimisers of the combinatorial optimisation problem
\begin{equation}\label{eq:opt}
  \Phi_X(\delta) = \min{\big\{|I|\log(1/p) : I\subseteq \br N
\text{ and }\e[X\mid Y_i = 1 \text{ for all $i\in I$}] \geq (1+\delta)\e[X]\big\}}.
\end{equation}
Roughly speaking, we say that $I\subseteq \br N$ is a \emph{core} if it
is a feasible set for the above optimisation problem, its size is $O\big(\Phi_X(\delta)\big)$, and it satisfies a certain natural rigidity condition that arises from requiring every element in the set to contribute a sizeable amount to the expectation. The constraints used to define a core are loosely analogous to the complexity conditions used in nonlinear large deviation theory; we will give a more precise definition of a core and discuss its relations to earlier work in more detail in Section~\ref{sec:method}.

We show that the upper tail probability is approximately equal to the probability of the appearance of a core. In particular, when the number of cores of size $m$ is $(1/p)^{o(m)}$,
a~property we term \emph{entropic stability}, then a union bound implies that
$- \log \Pr\big(X \ge (1+\delta)\Ex[X]\big)$ is well approximated by
$\Phi_X(\delta)$. We will verify that the
random variables $X_{N,p}^\kap$ and $X_{n,p}^H$
(for a large class of graphs $H$)
satisfy the entropic stability
condition under optimal, or nearly optimal, assumptions on $p$.\footnote{We use the phrase \emph{entropic stability} in a similar sense to the notion of stability in extremal combinatorics. More precisely, we are considering situations where the probability that $\prod_{i\in I}Y_i=1$ for some minimiser $I$ of~\eqref{eq:opt} is asymptotically equal to the probability of appearance of one such minimiser---in other words, the \emph{energy} of such configurations (given by the number of elements involved) dominates over the \emph{entropy} (that is, the number of such configurations). Then \emph{entropic stability} means that the entropy term remains
  negligible even as we move away from true minimisers of~\eqref{eq:opt} to sets that are merely close to being minimisers (the cores).}

One important caveat that we have ignored so far is that the upper tail exhibits localisation only when the expectation of $X$ tends to infinity sufficiently quickly.
In fact, if $\e[X]$ is of constant order, then, under relatively 
mild conditions, $X$ converges in distribution to a Poisson random variable
and no localisation occurs. We show that, for the two examples discussed above,
the upper tail continues to have Poisson behaviour even when $\e[X]$ goes to
infinity sufficiently slowly.
In the cases of $k$-term arithmetic progressions in $\br{N}_p$ and cliques in $G_{n,p}$, our results
for the Poisson and localised regimes cover almost the whole range of densities $p\to 0$ with $\e[X]\to \infty$,
leaving the upper tail probability undetermined only at densities for which it is believed that the two behaviours coexist.

\subsection{Arithmetic progressions in random sets of integers}

Let $X = X_{N,p}^\kap$ denote the number of $k$-term arithmetic
progressions in $\br N_p$. It is not hard to
see that $\e[X] = \Theta(N^2p^k)$. 
Whenever this expectation vanishes, the upper tail event is 
commensurate with the probability of $X\geq 1$, which can be controlled using
Markov's inequality. More generally, if $\e[X]$ is bounded,
then it follows from standard techniques that $X$ is asymptotically Poisson~\cite{BarKocLiu19};
in this case, the large deviations of $X$ are those of a Poisson random variable
with mean $\e[X]$. For the remainder
of this section, we shall thus assume that $\e[X]\to \infty$, i.e.,
that $p^{k/2}\gg N^{-1}$.

Improving an earlier estimate
due to Janson and Ruciński~\cite{janson2011upper},
Warnke~\cite{warnke2017upper} proved that under fairly general assumptions
(in particular, for constant $\delta>0$ and all $p$ bounded
away from~$1$),
\begin{equation}\label{eq:lutzwarnke} -\log \Pr(X\geq (1+\delta)\e[X]) =
\Theta\left(\min\left\{\big((1+\delta)\log(1+\delta)-\delta\big)\e[X],
\sqrt{\delta \e[X]}\log(1/p)\right\}\right), \end{equation}
where the constants implicit in
the $\Theta$-notation are independent of $\delta$. Note that the two terms of the minimum
correspond to the dominance of the Poisson and the localised regimes, respectively.

Since then, it has been an open problem to determine the missing constant factor
in~\eqref{eq:lutzwarnke}. Using
the above-mentioned framework of Eldan~\cite{eldan2016gaussian},
Bhattacharya, Ganguly, Shao, and Zhao~\cite{bhattacharya2016upper} were able to
do so in the range $N^{-\frac1{12(k-1)}}(\log N)^{O(1)}\ll p^{k/2} \ll1$. This was
subsequently improved by Briët--Gopi~\cite{briet2017gaussian} to the slightly wider range
$N^{-\frac1{12\ceil{(k-1)/2}}} (\log N)^{O(1)}\ll p^{k/2}\ll1$, also using Eldan's result.
The two theorems below improve significantly on these results and determine the precise rate
of the upper tail for all $N^{-1}\ll p^{k/2}\ll 1$, excepting the case $p^{k/2} =
\Theta(N^{-1}\log N)$. The first result 
concerns the range where the minimum in~\eqref{eq:lutzwarnke} is $\sqrt{\delta \e[X]}
\log(1/p)$.

\begin{restatable}{theorem}{thmkap}\label{thm:kap}
  Let $k \geq 3$ be an integer and let $X=X_{N,p}^\kap$ denote the number of
  $k$-term arithmetic progressions in $\br{N}_p$. Then, for every fixed positive
  constant $\delta$ and all $p=p(N)$ satisfying $N^{-1}\log N \ll p^{k/2} \ll 1$,
  \[
    \lim_{N\to\infty} \frac{-\log \Pr \big(X \geq (1+\delta) \e[X]
    \big)}{Np^{k/2}\log(1/p)} = \sqrt{\delta}.
  \]
\end{restatable}

Observe that Theorem~\ref{thm:kap} shows that the upper tail probability is
well approximated by the probability of appearance of an interval
(or, more generally, an arithmetic progression) of
length $\sqrt{\delta N^2p^k}$ in~$\br N_p$. Since each such interval contains approximately
$\delta \Ex[X]$ arithmetic progressions of length $k$, it is not hard to see that conditioning $\br{N}_p$
on the appearance of such a set will cause the upper tail event to occur with sizable probability. Conversely,
our methods may be used to prove that the upper tail event is dominated by the appearance of some set of size
$(1+o(1))\sqrt{\delta N^2 p^k}$ that contains nearly $\delta \Ex[X]$ arithmetic progressions of length $k$.
It seems natural to guess that each such set is, in some sense, close to an arithmetic progression. However,
this is not the case, as was shown by Green--Sisask~\cite{green2008maximal}. We currently do not know
a structural characterisation of the sets described above, which prevents
us from proving a qualitative description of the upper tail event. For
further discussion, we refer to Section~\ref{sec:conclusion}.

The second result treats the complementary range $N^{-1}\ll p^{k/2}\ll N^{-1}\log N$,
where the upper tail has Poisson behaviour.

\begin{restatable}{theorem}{thmkappoisson}\label{thm:kappoisson}
  Let $k \geq 3$ be an integer and let $X=X_{N,p}^\kap$ denote the number of
  $k$-term arithmetic progressions in $\br{N}_p$. Then, for every fixed positive
  constant $\delta$ and all $p=p(N)$ satisfying $N^{-1}\ll p^{k/2} \ll N^{-1}\log N$,
  \[
    \lim_{N\to\infty} \frac{-\log \Pr \big(X \geq (1+\delta) \e[X]
    \big)}{\e[X]} = (1+\delta)\log(1+\delta)-\delta.
  \]
\end{restatable}

\subsection{Subgraph counts in random graphs}

Let $X = X_{N,p}^H$ be the number of copies of a~ fixed graph $H$ in
$G_{n,p}$. Note that $\e[X] = \Theta(n^{v_H}p^{e_H})$. Since controlling the distribution of $X$ for completely general
graphs involves many technical difficulties (see for example
\cite{bollobas1998random,vsileikis2018counterexample}),
 we will restrict our attention to connected, $\Delta$-regular graphs $H$.
If the expected value of $X$ is bounded, then 
$X$ converges to a Poisson random variable, as was shown 
independently
by Bollobás~\cite{bollobas1981threshold} 
and by Karoński--Ruciński~\cite{karonski1983number}.
In view of this, for the remainder of this section, we shall assume that $\e[X]\to \infty$, i.e., that $p^{\Delta/2} \gg n^{-1}$.
As mentioned before, we will also assume that $p \to 0$; the case where $p\in(0,1)$ is fixed, which is fundamentally different,
was addressed in~\cite{chatterjee2011large,lubetzky2015replica,MukBha18}.

The problem of controlling the upper tail of $X$ 
has a long history. A sequence of papers~\cite{janson2004deletion, kim2004divide, vu2002concentration},
culminating in the work of Janson, Oleszkiewicz, and Ruciński~\cite{janson2004upper},
resulted in upper and lower bounds on the logarithmic upper tail probability that differed
by a multiplicative factor of~$\log(1/p)$. In the case where $H$ is a clique, Chatterjee~\cite{chatterjee2012missing} and DeMarco--Kahn~\cite{demarco2012tight}
independently added the missing logarithmic factor to the upper bound, thus
establishing the order of magnitude of the logarithmic upper tail probability.
The theory of nonlinear large deviations (discussed above) provides a variational description of the dependence
of the upper tail probability $\Pr(X\geq (1+\delta)\e[X])$ on $\delta$
for a certain range of $p\to 0$, as established in~\cite{augeri2018nonlinear,
chatterjee2016nonlinear, cook2018large,cook2021hypergraph, eldan2016gaussian};
the strongest of these results require $p \gg n^{-1/(\Delta +1)}$
for general graphs of maximal degree $\Delta$~\cite{cook2018large, cook2021hypergraph}, and $p^{\Delta/2} \gg n^{-1/2}$
for the case where $H$ is a~cycle~\cite{augeri2018nonlinear,cook2018large} (disregarding
polylogarithmic factors).
The associated variational problem was solved 
by Lubetzky--Zhao~\cite{lubetzky2017variational} when $H$ is a clique and by
Bhattacharya, Ganguly, Lubetzky, and Zhao~\cite{bhattacharya2017upper} for general $H$.
For a more detailed overview of these techniques, we refer the reader
to the book of Chatterjee~\cite{Cha17}.

The solution to the variational problem is expressed in terms of the independence
polynomial of a graph. For any $H$, define $P_H(x) = \sum_k i_k(H)x^k$, where
$i_k(H)$ is the number of independent sets of $H$ of size $k$,
and let $\theta = \theta(\delta)$ be the unique positive solution to $P_H(\theta) = 1+\delta$.\footnote{We note that $i_0(H) = 1$ for every graph $H$, so that, for example, $P_{K_r}(x) = 1+rx$.}
There are two constructions that yield lower bounds for the upper tail
probability (see Figure~\ref{fig:H}). In both cases, one plants a `small'
subgraph whose presence ensures that $G_{n,p}$ contains
$(1+\delta)\e[X]$ copies of $H$ with good probability. The first of these subgraphs is
a clique on $\delta^{1/v_H} np^{\Delta/2}$ vertices (as in the left side of the figure), which
contains the extra $\delta\e[X]$ copies of $H$ required by the upper tail event (up to lower-order corrections).
The second subgraph (which is often called a `hub') is a complete bipartite graph with parts of size $\theta np^\Delta$ and
$n-\theta np^\Delta$, respectively (as in the right side of the figure); since we are implicitly
assuming that $\theta np^\Delta$ is an integer, rounding errors play a significant role here
unless $np^\Delta \gg 1$. A short calculation shows that the expected number of copies of $H$
which intersect this graph is approximately $\delta \e[X]$ and thus the actual
number of such copies is almost $\delta \e[X]$ with good probability.
In both cases, the complement of the planted subgraph typically contains approximately $\e[X]$ copies of $H$.
Formalising this argument, one obtains the two lower bounds
\[
  \Pr\big(X\geq (1+\delta)\e[X]\big) \ge p^{(\delta^{2/v_H} + o(1))n^2p^\Delta} \quad \text{and}  \quad
  \Pr\big(X\geq (1+\delta)\e[X]\big) \ge p^{(\theta + o(1))n^2p^\Delta},
\]
which correspond to planting the clique and the complete bipartite graph, respectively. (Recall that the latter bound is valid only when $np^\Delta \gg 1$.)

Our main result is that,
 when $H$ is not bipartite,
one of the above bounds is tight in nearly the whole range of
densities. When $H$ is bipartite, we prove tight bounds
on the logarithmic upper tail probability only when $p^{\Delta/2}\geq n^{-1/2-o(1)}$.    
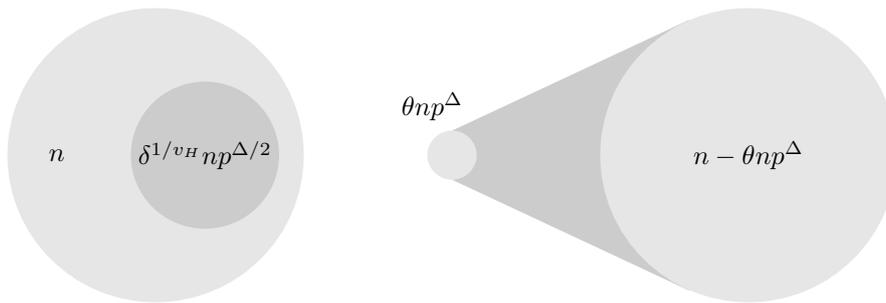
\begin{figure}
  \centering
  \begin{tikzpicture}[scale=1.3]
    \usetikzlibrary{patterns}
    \tikzset{every node/.style={draw=none,rectangle,fill=none,inner sep=0pt}}
    \begin{scope}
      \fill[black!10] (0,0) circle (1.5cm);
      \fill[black!20] (0.5,0) circle (0.75cm);
      \node at (0.5,0) {$\delta ^{1/v_H}np^{\Delta/2}$};
      \node at (-1,0) {$n$};
    \end{scope}
    \begin{scope}[xshift=4cm]
      \fill[black!20] (-1,0.25) -- (1.4,1.38) -- (1.4,-1.38) -- (-1,-0.25) -- (-1,0.25);
      \fill[black!10] (-1,0) circle (0.25cm);
      \fill[black!10] (2,0) circle (1.5cm);
      \node at (-1.2,0.5) {$\theta np^\Delta$};
      \node at (2,0) {$n - \theta np^\Delta$};
    \end{scope}
  \end{tikzpicture}
  \caption{The two constructions giving lower bounds for the upper
  tail probability of $X = X_{n,p}^{H}$.}\label{fig:H}
\end{figure}

\begin{theorem}\label{thm:H}
  Let $\Delta \ge 2$ be an integer, let
  $H$ be a connected, nonbipartite, $\Delta$-regular
  graph, and let $X = X_{N,p}^H$ denote the number of copies of
  $H$ in $G_{n,p}$. Then for every fixed positive constant~$\delta$
  and all $p = p(n)$ satisfying
  $n^{-1}(\log n)^{\Delta v_H^2} \ll p^{\Delta/2} \ll 1$,
  \[
    \lim_{n\to \infty} \frac{-\log \Pr \big( X\geq
    (1+\delta)\e[X]\big)}{n^2p^{\Delta}\log(1/p)} = \begin{cases}
    \delta^{2/v_H}/2 & \text{if $np^\Delta\to 0$,}\\
    \min{\{\delta^{2/v_H}/2,\theta\}} & \text{if $np^\Delta\to \infty$},
    \end{cases}
  \]
  where $\theta$ is the unique positive solution to $P_H(\theta) = 1+\delta$.
  Additionally, if $p^{\Delta/2} \geq n^{-1/2-o(1)}$, then the assumption that $H$ is
  nonbipartite is not necessary.
\end{theorem}

We note that the theorem leaves open the case where $np^\Delta \to c\in (0,\infty)$. In this regime,
the explicit dependence of the upper tail probability on $\delta$ involves various integrality conditions
and is therefore quite complicated. In the next subsection, we explicitly treat this regime when $H$ is a clique. The assumption of nonbipartiteness is not phenomenological, but only technical. The aforementioned entropic stability condition, which plays a crucial role in our proof, ceases to hold when $H$ is bipartite as soon as $p^{\Delta/2} \le n^{-1/2-\Theta(1)}$, see Section~\ref{sec:conclusion}.

Our next result concerns the Poisson regime of the upper tail.

\begin{restatable}{theorem}{thmHpoisson}\label{thm:Hpoisson}
  Let $\Delta\geq 2$ be an integer,
  let $H$ be a connected, $\Delta$-regular graph, and let $X = X_{n,p}^H$
  denote the number of copies of $H$ in $G_{n,p}$.
  Then, for every fixed positive constant $\delta$ and all $p=p(n)$
  satisfying $n^{-1} \ll p^{\Delta/2} \ll n^{-1} (\log n)^{\frac1{v_H-2}}$,
  \[
    \lim_{n\to\infty} \frac{-\log \Pr \big(X \geq (1+\delta) \e[X]
    \big)}{\e[X]} = (1+\delta)\log(1+\delta)-\delta.
  \]   
\end{restatable}

In the case where $H$ is a clique, DeMarco and Kahn~\cite{demarco2012tight} proved that the logarithmic upper tail probability is of order $\e[X]$ throughout the regime covered by Theorem~\ref{thm:Hpoisson}. For other $\Delta$-regular graphs $H$, the analogous fact was known only in the range $n^{-1} \ll p^{\Delta/2} \ll n^{-1} (\log n)^{c_H}$, for some positive constant $c_H < 1/(v_H-2)$, see~\cite{janson2004deletion,vsileikis2012upper,Vu00,Warnke-miss-log}.

Finally, we point out that the powers of the logarithms in the assumptions of Theorems~\ref{thm:H} and~\ref{thm:Hpoisson}
do not match. After a preprint of this paper appeared, Basak and Basu~\cite{basak2019upper} combined a generalised notion of entropic stability with a more refined version of the approach used in this paper to prove tight bounds on the logarithmic upper tail probability for the subgraph count of any $\Delta$-regular graph $H$ in the entire localised regime (see Section~\ref{sec:method} for a more detailed discussion). Specifically, they remove the assumption that $H$ is nonbipartite, and improve the assumed lower bound on the density $p$ in Theorem~\ref{thm:H} to $n^{-1}(\log n)^{1/(v_H-2)} \ll p^{\Delta/2}$, thus matching the assumptions of Theorem \ref{thm:Hpoisson}.

\subsection{Clique counts in random graphs}

We now consider the case of $X = X_{n,p}^H$ where $H$ is a~clique on $r \ge 3$
vertices. Thanks to the simpler structure of these graphs, we are able to prove
significantly stronger results in this setting. First, we are
able to determine the explicit dependence
of the logarithmic upper tail probability on $\delta$ even when $np^{r-1}\to c\in
(0,\infty)$. Moreover, we resolve the upper tail problem for the optimal range of densities
$n^{-1}(\log n)^{\frac1{r-2}}\ll p^{\frac{r-1}2} \ll 1$, complementing the range covered by Theorem~\ref{thm:Hpoisson}. Finally, we give a structural description
of $G_{n,p}$ conditioned on the upper tail event.

In order to formally state the theorem, it is convenient to define
\begin{equation}\label{eq:grdcx}
  \psi_r(\delta,c,x) = \frac{\big(\delta(1-x)\big)^{2/r}}2+
\frac{\floor{x\delta  c/r} + \{x\delta  c/r\}^\frac1{r-1}}{c},
\end{equation}
where $\delta$ and $c$ are nonnegative reals, $x \in [0,1]$, and $\{a\}$ denotes the
fractional part of $a$, and
\begin{equation}\label{eq:grdc}
  \varphi_r(\delta,c) = \min_{x\in [0,1]} \psi_r(\delta,c,x).
\end{equation}
For an intuitive explanation of the combinatorial meaning of these
definitions, we refer to the discussion at the beginning of
Section~\ref{sec:cliques}.
An easy convexity argument shows that the minimum in the definition
of $\varphi_r$ is attained
when $x \in \{0,r\floor{\delta c/r}/(\delta c),1\}$,
see Lemma~\ref{lem:grdcx}. This leads to the explicit formula
\[
  \varphi_r(\delta,c) = \min{\left\{ \frac{\delta^{2/r}}{2},
  \frac{\floor{\delta c/r}+\{\delta c/r\}^{1/(r-1)}}{c}, \frac{\floor{\delta
  c/r}}{c}+ \frac{(r\{\delta c/r\}/c)^{2/r}}{2} \right\}}.
\]

\begin{restatable}{theorem}{thmkrrate}
  \label{thm:krrate}
  Let $r \ge 3$ be an integer and let $X = X_{n,p}^{K_r}$ denote the number of
  $r$-vertex cliques in the random graph $G_{n,p}$. Then, for every fixed positive
  constant
  $\delta$ and all $p = p(n)$ satisfying $n^{-1} (\log
  n)^{\frac1{r-2}} \ll p^{\frac{r-1}{2}} \ll 1$,
  \[
    \lim_{n \to \infty} \frac{-\log \Pr\big(X \ge
    (1+\delta)\Ex[X]\big)}{n^2p^{r-1} \log(1/p)} =
    \begin{cases}
      \delta^{2/r}/2 & \text{if $np^{r-1} \to 0$}, \\
      \varphi_r(\delta, c) & \text{if $np^{r-1} \to c \in (0, \infty)$}, \\
      \min\big\{\delta^{2/r}/2,\delta/r\big\} & \text{if $np^{r-1} \to \infty$}.
    \end{cases}
  \]
\end{restatable}

Our next result describes the typical structure of the random graph $G_{n,p}$
conditioned upon the upper tail event. 
We write $G_{n,p}[U]$ for the subgraph of $G_{n,p}$ induced by $U$ and $e(A,B)$
for the number of edges in $G_{n,p}$ with one endpoint in $A$ and another in $B$.
Define the following three events:
\begin{enumerate}[{label=(\roman*)}]
\item
  Let $\UT(\delta)$ be the upper tail event $\{X \ge  (1+\delta)\e[X]\}$.
\item
  Let $\Clique_\eps(x)$ be the event that $G_{n,p}$ contains a set $U\subseteq \br n$ of 
    size at
    least $(1-\eps)x^{1/r}np^{(r-1)/2}$ such that
    $G_{n,p}[U]$ has minimum degree
    at least $(1-\eps)|U|$.
\item
  Let $\Hub_\eps(x)$ be the event that $G_{n,p}$ contains a set $U\subseteq \br n$
    such that at least $\floor{(1 - \eps)|U|}$ vertices in $U$ have degree at least $(1-
    \eps) n$ and 
    \[ e(U,\br n\setminus U) \geq (1-\eps)n
    \big( \floor{xnp^{r-1}/r}
    + \{xnp^{r-1}/r\}^{\frac{1}{r-1}}\big).\]
\end{enumerate}
Observe that $\Clique_\eps(0)$ and $\Hub_\eps(0)$ hold vacuously.

\begin{restatable}{theorem}{thmkrstructure}
  \label{thm:krstructure}
  Let $r \ge 3$ be an integer and let $\delta$, $\eps$, and $c$ be fixed positive
  constants. The following holds for all $p = p(n)$ satisfying 
  $n^{-1}
  (\log n)^{\frac1{r-2}} \ll p^{\frac{r-1}2} \ll 1$.
  \begin{enumerate}[label=(\roman*)]
  \item 
    If $np^{r-1} \to 0$, then
    \[
      \Pr\big(\Clique_\eps(\delta) \mid \UT(\delta)\big)\to 1.
    \]
  \item 
    If $np^{r-1}\to c$, then, letting $x^* = r\floor{\delta c/r}/(\delta c)$,
    \[
      \Pr\left(\bigcup_{x\in \{0,x^*,1\}} \Big(\Clique_\eps\big(\delta(1-x)\big) \cap \Hub_\eps(\delta x)\Big) \mid \UT(\delta)\right) \to 1,
    \]
    Moreover, if
    $x \mapsto \psi_r(\delta,c,x)$ has a unique minimiser $x \in \{0,x^*,1\}$,
    then
    \[
      \Pr\Big(\!\Clique_\eps\big(\delta(1-x)\big) \cap \Hub_\eps(\delta x)
      \mid \UT(\delta)\Big) \to 1.
    \]
  \item
    If $np^{r-1}\to \infty$, then
    \[
      \Pr\big(\Clique_\eps(\delta) \cup \Hub_\eps(\delta) \mid \UT(\delta)\big) \to 1.
    \]
    Moreover,
    \[
      \Pr\big(\Clique_\eps(\delta) \mid \UT(\delta)\big) \to
      \begin{cases}
        1 & \text{if $\delta^{2/r}/2< \delta/r$}, \\
        0 & \text{if $\delta^{2/r}/2> \delta/r$}.
      \end{cases}
    \]
  \end{enumerate}
\end{restatable}

Note that Theorem~\ref{thm:krstructure} remains agnostic about the exact structure of the conditional model in the case
where there are multiple minimisers to $x \mapsto \psi_r(\delta, np^{r-1}, x)$.
However, it is not too difficult to show that for every $r$, the set of $(\delta,c) \in
(0,\infty)^2$ for which $x \mapsto \psi_r(\delta, c, x)$ has multiple minimisers
has Lebesgue measure zero. Figure~\ref{fig:phasediagram} gives a
graphical representation of the assertion of Theorem~\ref{thm:krstructure} in
the case where $r=3$ and $np^{2}\to c$. As the figure illustrates,
the conditional model undergoes infinitely many phase transitions
if $\delta^{2/3}/2> \delta/3$ (that is, $\delta< 3.375$) and
no phase transition at all if
$\delta^{2/3}/2< \delta/3$.

\begin{figure}
  \includegraphics{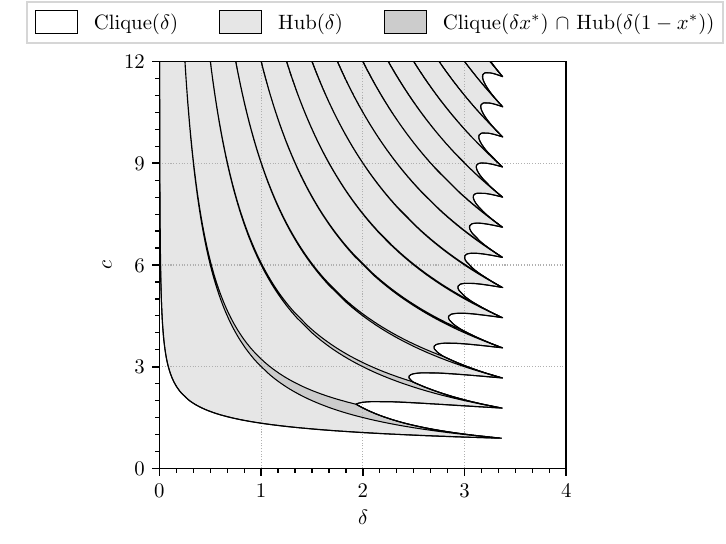}
  \caption{
    Asymptotic structure of $G_{n,p}$ conditioned upon the upper
    tail event $\UT(\delta) = \{X_{n,p}^{K_3} \geq
    (1+\delta)\e[X_{n,p}^{K_3}]\}$ when
    $np^2\to c$ as $n\to\infty$.
    In this conditional model,
    it is highly probable that we observe
    either
    $\Clique_\eps(\delta)$, 
    $\Hub_\eps(\delta)$, or
    $\Clique_\eps(\delta x^*) \cap \Hub_\eps((1-\delta)x^*)$
    with
    $x^* = 3\floor{\delta c/3}/(\delta c)\in (0,1)$,
    depending on the values of $\delta$ and $c$
    (all regions are open).
  }\label{fig:phasediagram}
\end{figure}

\subsection{Organisation of the paper}

In Section~\ref{sec:triangles}, we present a short and self-contained solution to the upper tail problem for triangle counts
in $G_{n,p}$. This section is somewhat redundant, since its content is just
a special case of the more general Proposition~\ref{prop:kr-ldp}. We include
it in order to demonstrate our method in a simple setting that conveniently
avoids many technical complications that arise in the general case.

Section~\ref{sec:method} introduces a concentration inequality that
gives a general condition under which the
logarithmic upper tail probability can be approximated by $\Phi_X(\delta)$,
the solution to the optimisation problem \eqref{eq:opt}. 

In Section~\ref{sec:aps}, we use the inequality developed in Section~\ref{sec:method}
to determine the asymptotics of
the logarithmic upper tail probability of $X_{N,p}^\kap$ in the complete range
of densities where localisation occurs. After collecting some graph-theoretic
tools in Section~\ref{sec:graph-theory-preliminaries}, we study the localised
regime of the upper tails of $X_{n,p}^{K_r}$ and $X_{n,p}^H$ for connected,
$\Delta$-regular graphs $H$ in Sections~\ref{sec:cliques} and \ref{sec:H},
respectively. We note that the three Sections \ref{sec:aps}, \ref{sec:cliques},
and \ref{sec:H} are logically independent and may be read in any order;
however, both Sections~\ref{sec:cliques} and~\ref{sec:H} rely on the tools
of Section~\ref{sec:graph-theory-preliminaries}.

In Section~\ref{sec:poisson}, we prove various results related to the
Poisson regime; in particular, we give the proofs of
Theorems~\ref{thm:kappoisson} and \ref{thm:Hpoisson}.
The arguments we use there do not rely on the methods developed
in Section~\ref{sec:method}, but rather on explicit calculations of high factorial moments.
Section~\ref{sec:nonpolynomial} contains a brief
discussion on extending the result from Section~\ref{sec:method} to the
more general case of nonnegative random variables on the hypercube. Finally,
in Section~\ref{sec:conclusion} we make some concluding remarks and discuss
open problems.

\subsection{Notation}
Before ending the introduction, we collect some notation which will be
used throughout the paper. We write $K_r$ for the complete graph
on $r$ vertices and $K_{s,t}$ for the complete bipartite graph with parts
of size $s$ and $t$. For any graph $G$, let
$V(G)$ and $E(G)$ denote the vertex and edge sets of $G$, respectively,
and set
$v_G = |V(G)|$ and $e_G =|E(G)|$. For two graphs $J$
and $G$, we let $N(J,G)$ be the number of copies of $J$ in $G$, and
$\Emb(J,G)$ be the set of embeddings of $J$ into $G$---i.e., injective
maps from $V(J)$ to $V(G)$ that map edges
of $J$ to edges of $G$. For an edge $uv \in E(G)$, we also let $N(J,G;uv)$
and $\Emb(J,G;uv)$ be the number of copies of $J$ that contain the edge
$uv$, and the set of embeddings that map
an edge of $J$ to $uv$, respectively.
Finally, for a subset
$I$ of $\br{N}$, we let $Y_I = \prod_{i \in I} Y_i$, and $\mathbb{E}_I[X] = \mathbb{E}[X \mid Y_I = 1]$.
If subsets $I\subseteq \br N$ can be identified with subgraphs $G\subseteq K_n$,
as in Example~\ref{ex:subgraph}, we will write $\e_G[X]$ instead
of $\e_I[X]$.

\section{Triangles in random graphs}\label{sec:triangles}

Assume that $n^{-1}\log n \ll p \ll 1$ and let $X$ denote the number of triangles in $G_{n,p}$.
Using the shorthand notation $\e_G[X] = \e[X\mid G\subseteq G_{n,p}]$, we define, for each positive $\delta$,
\begin{equation}\label{eq:triaXi}
  \Phi_X(\delta) = \min{\big\{e_G\log(1/p) : G\subseteq K_n\text{ and }\e_G[X]\geq
  (1+\delta)\e[X]\big\}}.
\end{equation}
Note that this agrees with the definition \eqref{eq:opt}.
The goal of this section is to prove that,
for every fixed positive $\eps$ and all large enough $n$,
\begin{equation}\label{eq:goaltriangles}
  (1-\eps)\Phi_X(\delta-\eps)
  \leq -\log \Pr\big(X\geq (1+\delta)\e[X]\big)
  \leq (1+\eps)\Phi_X(\delta+\eps).
\end{equation}
At this point, we do not address the issue of evaluating $\Phi_X(\delta)$.
For the sake of completeness, let us mention that a special case of a more general result of
Lubetzky--Zhao~\cite{lubetzky2017variational} is that, when $n^{-1} \ll p \ll 1$,
\[ \lim_{n\to\infty}\frac{\Phi_X(\delta)}{n^2p^2\log(1/p)} =
\begin{cases}
  \delta^{2/3}/2&
\text{if $np^2 \to 0$}\\
  \min\left\{\delta/3,\delta^{2/3}/2\right\} &
\text{if $np^2 \to \infty$.}
\end{cases} \]
In Section~\ref{sec:cliques}, we shall fill in the gap at
$p=\Theta(n^{-1/2})$ to obtain an asymptotic formula for
$\Phi_X(\delta)$ in the full range of interest.

We now give a proof of~\eqref{eq:goaltriangles}, where we may assume without loss
of generality that $\eps\leq \delta/10$. All statements that we make in this
section should be understood to hold only for sufficiently large~$n$. We start
with some easy observations.
First, for every graph $G$ with $O(1)$ edges, we have
$\e_G[X] - \e[X] = O(1+np^2)\ll (\delta-2\eps)\e[X]$, and so
$\Phi_X(\delta-2\eps)\gg\log(1/p)\gg 1$. Second, the condition in~\eqref{eq:triaXi}
is satisfied when $G$ is a clique on  $(1+\delta)^{1/3}np$ vertices, and therefore $\Phi_X(\delta-\eps) \leq
(1+\delta)^{2/3}n^2p^2\log(1/p)/2$. 

The easier of the two inequalities in~\eqref{eq:goaltriangles} is the upper bound. To
prove it, let $G$ be a graph attaining the minimum in the definition of
$\Phi_X(\delta+\eps)$ and let $\Pr_G(\cdot) = \Pr(\cdot\mid G\subseteq G_{n,p})$. Since $X$ never exceeds $\binom{n}{3}$,
\[
  (1+\delta+\eps)\e[X] \leq \e_G[X] \leq \binom{n}{3}\Pr_G\big(X\geq
  (1+\delta)\e[X]\big) + (1+\delta)\e[X],
\]
and so
\[
  \Pr_G\big(X\geq (1+\delta)\e[X]\big)\geq \eps \e[X]/\binom{n}{3} =\eps p^3.
\]
Hence,
\[ 
\begin{split}
  -\log \Pr\big(X\geq (1+\delta)\e[X]\big)
  &\leq -\log\Big(\Pr(G\subseteq G_{n,p})\cdot\Pr_G\big(X\geq
  (1+\delta)\e[X]\big)\Big)\\
  &\leq e_G\log(1/p) +\log(1/\eps p^3).
\end{split} \]
Since $e_G\log(1/p) = \Phi_X(\delta+\eps)
\geq \Phi_X(\delta-2\eps)\gg \log(1/p)$, this establishes the lower bound in~\eqref{eq:goaltriangles}.

We now turn to proving the lower bound. Let $C = C(\eps,\delta)$ denote a
sufficiently large positive constant. We call a graph $G\subseteq K_n$ a \emph{seed} if
\begin{enumerate}[label=(S\arabic*)]
  \item \label{item:triaseed-bias} $\e_G[X] \geq (1+\delta-\eps)\e[X]$ and
  \item \label{item:triaseed-size} $e_G \leq Cn^2p^2\log(1/p)$.
\end{enumerate}

We make the following claim:
\begin{claim}\label{cl:triaseed}
  $\Pr\big(X\geq (1+\delta)\e[X]\big)\leq (1+\eps) \Pr\big(G_{n,p} \text{
    contains a seed}\big)$.
\end{claim}
\begin{remark}\label{remark:SeedsSuck}
Given this claim, one may be tempted to simply apply the union bound over all seeds. Using such a strategy, one would find that 
\begin{align}
    \Pr\big(X\geq (1+\delta)\e[X]\big)
    &\stackrel{\text{Claim~\ref{cl:triaseed}}}{\leq} (1+\eps)
    \Pr\big(G_{n,p}\text{ contains a seed}\big)\nonumber\\
    &\stackrel{\phantom{\text{Claim~\ref{cl:triaseed}}}}{\leq}
    (1+\eps) \sum_{m=m_{\min}}^\infty p^m \cdot |\{G \subseteq K_n : \text{$G$ is a seed with $m$ edges}\}|,\label{eq:failedunion}
\end{align}
where $m_{\min} = \Phi_X(\delta - \eps)/\log(1/p)$ is the minimal number of edges in a seed. Unfortunately, such a strategy is bound to fail, as there are far too many seeds. To see this, we observe that if $G$ satisfies~\ref{item:triaseed-bias}, then so does any supergraph of $G$. In particular, if we take a seed with $\mmin$ edges and add an arbitrary set of $K m_{\min}$ edges (where $K$ is a large constant), the resulting graph remains a seed. Therefore, we can (rather loosely) bound the number of seeds with $\tilde{m} = (K+1) m_{\min}$ edges from below:
\begin{align*}
|\{G \subseteq K_n : \text{$G$ is a seed with $\tilde{m}$ edges}\}| \geq \binom{{\binom {n}{2} - m_{\min}}}{K m_{\min} }  \geq \left(\frac{1}{3 (1+\delta)^{2/3} K  p^2 }\right)^{K \mmin},
\end{align*}
using the bound $\binom{x}{y} \geq (x/2y)^y$ for any $y < x/2$ and $\mmin \leq (1+\delta)^{2/3} n^2 p^2/2$. If we choose $K$ to be large enough, we may conclude that 
\[
|\{G \subseteq K_n : \text{$G$ is a seed with $\tilde{m}$ edges}| \geq (1/p)^{3 \tilde{m}/2},
\] 
for all large enough $n$. This shows that the $\tilde{m}$th term of the final sum in~\eqref{eq:failedunion} is arbitrarily large, making the entire approach fruitless.
\end{remark}

From Remark~\ref{remark:SeedsSuck}, it is clear that seeds may include `extraneous' edges that have no structural role in the upper tail event. We wish to consider more constrained structures which exclude constructions like the one outlined above. To that end, we call a graph $G^*\subseteq K_n$ a \emph{core} if
\begin{enumerate}[label=(C\arabic*)]
  \item\label{item:triacore-bias} $\e_{G^*}[X] \geq (1+\delta-2\eps)\e[X]$,
  \item\label{item:triacore-size} $e_{G^*} \leq Cn^2p^2\log(1/p)$, and
  \item\label{item:triacore-mindeg}
    $\min_{e\in E({G^*})} \big(\e_{G^*}[X]-\e_{{G^*}\setminus e}[X]\big) \geq
    \eps\e[X]/\big(Cn^2p^2\log(1/p)\big)$.
\end{enumerate}
Condition~\ref{item:triacore-mindeg} requires that every edge contributed meaningfully to the expectation, and thus excludes the pathological seeds described above. 
\begin{claim}\label{cl:triacore}
  Every seed contains a core.
\end{claim}
Finally, we claim that the additional constraint~\ref{item:triacore-mindeg} allows us to get a very strong bound on the number of cores with a given number of edges, as it ensures that either the  product of the degrees of the endpoints of any edge in $G^*$ must be nearly as large as $e_{G^*}$, or the sum of these degrees is nearly linear in $n$ (note that the former condition holds when $G^*$ is a clique, and the latter when $G^*$ is a hub). This, in turn, implies that the number of cores with a given set of vertices and $m$ edges is $\exp(O(m \log \log (1/p))$. Once we show that the number of choices for the vertex set of a core with $m$ edges is $(1/p)^{o(m)}$, we prove the following claim.
\begin{claim}\label{cl:triacorecount}
  For every $m$, there are at most $(1/p)^{\eps m}$ cores with exactly $m$ edges.
\end{claim}
As will be discussed in Section~\ref{sec:method}, we refer to the property described in Claim~\ref{cl:triacorecount} as {\em entropic stability.}

Let us show how these three claims imply the lower bound in~\eqref{eq:goaltriangles}:
\[
  \begin{split}
    \Pr\big(X\geq (1+\delta)\e[X]\big)
    &\stackrel{\text{Claim~\ref{cl:triaseed}}}{\leq} (1+\eps)
    \Pr\big(G_{n,p}\text{ contains a seed}\big)\\
    &\stackrel{\text{Claim~\ref{cl:triacore}}}{\leq}
    (1+\eps)\Pr\big(G_{n,p}\text{ contains a core}\big)\\
    &\stackrel{\text{\phantom{Claim~2.2}}}{\le}
    (1+\eps) \sum_{m=0}^\infty p^m \cdot |\{G^* \subseteq K_n : \text{$G^*$ is a core with $m$ edges}\}|\\
    &\stackrel{\text{Claim~\ref{cl:triacorecount}}}{\leq}
    (1+\eps)\sum_{m=\mmin}^\infty p^{(1-\eps)m},
  \end{split}
\]
where $\mmin$ is the minimal number of edges in a core. Since~\ref{item:triacore-bias} implies that $\Phi_X(\delta-2\eps) \leq \mmin \cdot \log(1/p)$, the assumption $p \ll 1$ yields 
\[
  \Pr\big(X \ge (1+\delta)\e[X]\big) \le (1+2\eps) \exp\big(-(1-\eps)\Phi_X(\delta-2\eps)\big).
\]
Finally, as $\Phi_X(\delta-2\eps) \gg 1$, we obtain
\[
  -\log\Pr\big(X\geq (1+\delta)\e[X]\big) \geq (1-2\eps) \Phi_X(\delta-2\eps),
\]
thus proving \eqref{eq:goaltriangles} with $2\eps$ instead of $\eps$.

\begin{remark}
Conditions~\ref{item:triaseed-size} and~\ref{item:triacore-size} require both seeds and cores to have no more than $C n^2p^2 \log(1/p)$ edges, forcing us to consider graphs which are larger than the minimiser of $\Phi_X(\delta)$ by a multiplicative factor of $\log(1/p)$. This is optimal in the following sense. On the one hand, replacing $\log(1/p)$ with a larger function would weaken the conclusion of Claim~\ref{cl:triaseed} and make Claim~\ref{cl:triacorecount} hold only for a smaller range of densities $p$. On the other hand, if we could replace $\log(1/p)$ by $\ell(p) = o(\log(1/p))$ without altering the validity of Claim~\ref{cl:triaseed}, then Claim~\ref{cl:triacorecount} (with an appropriately adjusted definition of a core) would hold for any $p \gg \ell(p)/n$; this would imply an upper bound on the upper tail probability that is stronger than the lower bound proven in Theorem~\ref{thm:Hpoisson} for $\ell(p)/n \ll p \ll \log n /n$.
\end{remark}

\begin{proof}[Proof of Claim~\ref{cl:triaseed}]
  We refine a classical argument due to Janson, Oleszkiewicz,
  and Ruci\'nski~\cite{janson2004upper}. Let $Z$  be the indicator
  random variable of the event that $G_{n,p}$ does not contain a seed and let $\ell
  = \lceil (C/3)n^2p^2\log(1/p) \rceil$. Since $XZ\geq 0$ and $Z^\ell =Z$, 
  Markov's inequality
  gives
  \begin{equation}\label{eq:stability1}
      \Pr\big(X \geq 
      (1+\delta)\e[X] \text{ and $G_{n,p}$ contains no seed}\big)
      =
      \Pr\big(X Z \geq (1+\delta)\e[X]\big)
      \leq 
      \frac{\e[X^\ell Z]}{(1+\delta)^\ell \e[X]^\ell}.
  \end{equation}
  We write $X = \sum_T Y_T$, where the sum ranges over all triangles $T$ in
  $K_n$ and $Y_T$ is the indicator random variable of the event that
  $T$ is contained in $G_{n,p}$.
  For every subgraph $G\subseteq K_n$, let $Z_G$ be the indicator
  random variable of the event that $G\cap G_{n,p}$ does not contain a seed.
  Observe that $Z_{G'}\leq Z_G$ whenever $G\subseteq G'$. In particular,
  since $Z = Z_{K_n}$, we have, for every $k\in \br{\ell}$,
  \[
    \begin{split}
      \e[X^k Z] &= \sum_{T_1,\dotsc,T_k} \e[Y_{T_1}\dotsb
    Y_{T_k} \cdot Z]\\
      &\leq \sum_{T_1,\dotsc,T_k}
    \e[Y_{T_1}\dotsb Y_{T_k}\cdot Z_{T_1\cup \dotsb \cup T_k}]\\
    & \leq \kern-5pt
    \sum_{T_1, \dotsc, T_{k-1}}
    \kern-4pt
      \e[Y_{T_1} \dotsb Y_{T_{k-1}}
      \cdot Z_{T_1\cup \dotsb \cup T_{k-1}}] 
      \cdot
      \sum_{T_k}\e[Y_{T_k}
      \mid Y_{T_1} \dotsc Y_{T_{k-1}}
      \cdot Z_{T_1\cup \dotsb \cup T_{k-1}}
      = 1],
    \end{split}\]
  where we can let the first sum in the last line range only over sequences $T_1,\dotsc,T_{k-1}$
  for which the event
  $\big\{Y_{T_1}\dotsb Y_{T_{k-1}}\cdot Z_{T_1\cup \dotsb \cup T_{k-1}}=1\big\}$
  has positive probability. This is equivalent to saying that
  the graph $T_1\cup \dotsb\cup T_{k-1}$ does not contain a seed and thus
  $Y_{T_1}\dotsb Y_{T_{k-1}}\cdot Z_{T_1\cup \dotsb \cup T_{k-1}}
  =Y_{T_1}\dotsb Y_{T_{k-1}}$.
  Moreover, since $e_{T_1\cup \dotsb \cup T_{k-1}} \leq 3(k-1) \leq 3(\ell-1) \leq Cn^2p^2\log(1/p)$,
  \[
    \sum_{T_k}\e[Y_{T_k} \mid Y_{T_1} \dotsc
    Y_{T_{k-1}} = 1] =  \e_{T_1\cup \dotsb \cup T_{k-1}} [X]
    <(1+\delta-\eps)\e[X],
  \]
  as otherwise $T_1 \cup \dotsb \cup T_{k-1}$ would be a seed, see~\ref{item:triaseed-bias} and \ref{item:triaseed-size}.
  Therefore,
  \[
    \sum_{T_1,\dotsc,T_k}
    \e[ Y_{T_1}\dotsb Y_{T_k}\cdot Z_{T_1\cup \dotsb\cup T_k}]
      <
      (1+\delta-\eps)\e[X]
      \cdot \sum_{T_1,\dotsc,T_{k-1}}
      \e[ Y_{T_1}\dotsb Y_{T_{k-1}}\cdot Z_{T_1\cup \dotsb\cup T_{k-1}}]
  \]
  and it follows by induction that $\e[X^\ell Z] 
  < (1+\delta-\eps)^\ell\e[X]^\ell$. Substituting this inequality into~\eqref{eq:stability1} gives
  \[
  \Pr\big(X \geq (1+\delta)\e[X] \text{ and $G_{n,p}$ contains no seed}\big) \leq
  \left(\frac{1+\delta-\eps}{1+\delta}\right)^\ell.
  \]
  Since the probability that $G_{n,p}$ contains a seed is at least $e^{-\Phi_X(\delta-\eps)}$, the probability that $G_{n,p}$ contains a given
  seed of smallest size, the bounds $1 \ll \Phi_X(\delta-\eps)  \leq (1+\delta)^{2/3}n^2p^2\log(1/p)/2$ imply that, for all sufficiently large $n$, 
  \[
    \left(\frac{1+\delta-\eps}{1+\delta}\right)^\ell
    \le \left(\frac{1+\delta-\eps}{1+\delta}\right)^{(C/3)n^2p^2\log(1/p)}
    \leq \eps  e^{-\Phi_X(\delta - \eps) }\leq  \eps
    \Pr\big(\text{$G_{n,p}$ contains a seed}\big)
  \]
  whenever the constant $C$ is sufficiently large\footnote{We note that the requirement for $C$ to be large is only used here.}. This implies the assertion of the claim.
\end{proof}

\begin{proof}[Proof of Claim~\ref{cl:triacore}]
  Let $G$ be a seed. Define a sequence $G = G_0\supseteq G_1\supseteq \dotsb
  \supseteq G_s = G^*$ of subgraphs of $G$ by repeatedly
  setting $G_{i+1}= G_i\setminus e$ for some edge $e\in G_i$ such that $\e_{G_i}
- \e_{G_i \setminus e}[X] < \eps \e[X]/\big(Cn^2p^2\log(1/p)\big)$, as long as such an
edge $e$
  exists. By construction, $G^*$ clearly satisfies
  \ref{item:triacore-mindeg}.
  Since $e_{G^*} \leq e_G \leq Cn^2p^2\log(1/p)$,
  we see that \ref{item:triacore-size} holds as well.
  Finally, as $s \leq e_G \leq Cn^2p^2\log(1/p)$, we have
  \[ \e_G[X] - \e_{G^*}[X] = \sum_{i=0}^{s-1}\big(\e_{G_i}[X] - \e_{G_{i+1}}[X]\big)
  < \eps\e[X]. \]
  Since $G$ is a seed, $\e_G[X]\geq (1+\delta-\eps)\e[X]$ and we obtain \ref{item:triacore-bias}.
\end{proof}

\begin{proof}[Proof of Claim~\ref{cl:triacorecount}]
  \begin{figure}
    \centering
    \begin{tikzpicture}[scale=1.3]
      \usetikzlibrary{patterns}
      \tikzset{every node/.style={draw,circle,fill=black,inner sep=0pt,minimum
      size=3pt}}
      \tikzset{plain/.style={draw=none,rectangle,fill=none,inner sep=0pt}}
      \coordinate (x) at (0,0) {};
      \coordinate (y)  at (1,0) {};
      \coordinate (a) at (0,1) {};
      \coordinate (b) at (0.5,-0.8) {};
      \coordinate (c) at (1,1) {};
      \fill[black!10,xshift=0.25cm] (0.75,-1.25) .. controls ++(180:-1) and
      ++(0:1) .. (0.75,0.25) .. controls (0.25,0.25) .. (0.25,0.75) .. controls
      ++(90:0.75) and ++(90:0.75) .. (-0.75,0.75) to (-0.75,-0.5) ..controls ++(90:-1/2)
      and ++(180:1/2).. (0,-1.25) to (0.75,-1.25) ;
      \draw (x) -- (b) -- (y);
      \draw[dashed] (a) -- (y);
      \draw (x)--(a);
      \draw[dashed] (x) -- (c)  -- (y);
      \draw[very thick] (x) -- node[fill=none,draw=none,yshift=-0.15cm] {$e$} (y);
      \draw node at (x) {} node at (y) {} node at (a) {} node at (b) {} node
      at (c) {};
      \node[plain] at (-0.75,0) {$G^*$};

      \begin{scope}[xshift=-5cm]
        \fill[black!10,rounded corners=1cm] (-2,-2) rectangle (3.5,2);
        \fill[black!20] (0,0) circle (1.5cm);
        \fill[black!30] (0.75,0) circle (0.75cm);
        \node[plain] at (0.5,0) {$B_{G^*}$};
        \node[plain] at (-0.75,0) {$A_{G^*}$};
        \node[plain] at (3,0) {$G^*$};
        \draw (-0,-0.5) node {}--(0.5,-0.5) node {};
        \draw (0.4,0.5) node {}--++(0.4,-0.2) node {};
        \draw (-0.6,0.5) node {}--++(0,0.5) node {};
        \foreach \i/\j in {1/-1.1,2/-0.5,3/0,4/0.5,5/1.1} {
          \node (a\i) at (2,\j) {};
        }
        \begin{scope}[xshift=0.75cm]
          \foreach \i/\j in {1/-45,2/-22.5,3/0,4/22.5,5/45} {
            \draw (\j:0.5) node {}-- (a\i);
          }
        \end{scope}
      \end{scope}
    \end{tikzpicture}
    \caption{
      Left: The sets $A_{G^*}$ and $B_{G^*}$ of high-degree
    vertices capture the edges of the core.
      Right: Three different types of triangles containing an edge $e$ in
      the core $G^*$.}\label{fig:triacore}
  \end{figure}
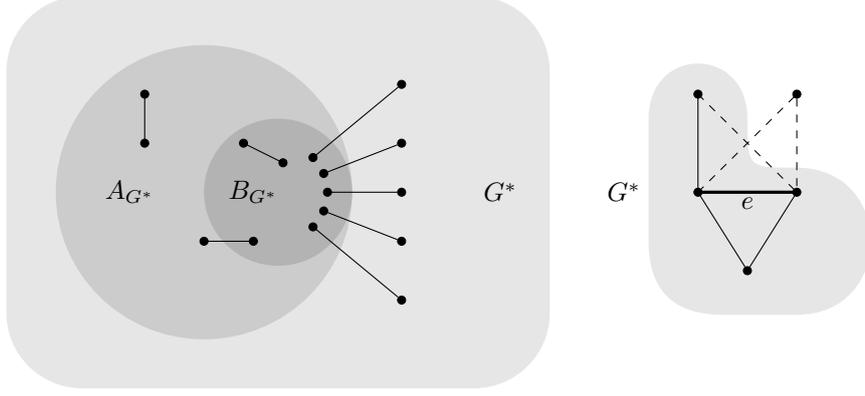
  We bound the number of cores with $m$ edges from above.
  This number is zero whenever $m> Cn^2p^2\log(1/p)$, by~\ref{item:triacore-size}.
  We may thus assume that $m\leq Cn^2p^2\log(1/p)$.
  Given a core $\Gcore$, we denote by $A_\Gcore$ the set of vertices of
  $\Gcore$ with degree at least $\eps np/\big(30C\log(1/p)\big)$ and by
  $B_\Gcore\subseteq A_\Gcore$ the set of vertices of $\Gcore$ with degree
  at least $\eps n/\big(30C\log(1/p)\big)$, where the degree is taken in $G^*$. Since $G^*$ has $m$ edges,
  \[
    |A_{G^*}|\leq a := \frac{60Cm\log(1/p)}{\eps np}
    \quad\text{and}\quad
    |B_{G^*}| \leq b := \frac{60Cm\log(1/p)}{\eps n}.
  \]
  The key observation, which we will
  now verify, is that every edge of $\Gcore$ is either contained in $A_\Gcore$
  or has an endpoint in $B_\Gcore$, see Figure~\ref{fig:triacore} for an
  illustration. To see this, consider some edge $e\in E(G^*)$.
  For every nonempty graph $F \subseteq K_3$, let $N(F,G^*; e)$ denote the
  number of copies
  of $F$ in $G^*$ that contain $e$.
  By considering how the $n-2$ triangles in $K_n$
  that contain $e$ intersect $G^*$ (see Figure~\ref{fig:triacore}),
  one can see that
  \[
    \e_{G^*}[X] -
    \e_{G^*\setminus e}[X] \leq \big(N(K_3,G^*;e)  + N(K_{1,2},G^*;e) \cdot p + np^2 \big) \cdot (1-p).
  \]
  Using
  $\e_{G^*}[X] - \e_{G^* \setminus e}[X]\geq \eps \e[X]/(Cn^2p^2\log(1/p))$ and
  $\e[X] \geq (1-o(1))n^3p^3/6$, we thus get
  \[
    \frac{\eps np}{7C\log(1/p)} \leq N(K_3,G^*; e) + N(K_{1,2},G^*; e) \cdot p + np^2.
  \]
  Since $p\ll 1$ implies that $np^2\ll np/\log(1/p)$, we find that either
  \begin{equation}
    \label{eq:tria-core-degrees}
    N(K_3,G^*; e)\geq \frac{\eps np}{15C\log(1/p)}
    \quad\text{or}\quad
    N(K_{1,2},G^*;e) \geq \frac{\eps n}{15C\log(1/p)}.
  \end{equation}
  Since $N(K_3,G^*; uv) \le \min\{ \deg_{G^*} u, \deg_{G^*} v\}$ and $N(K_{1,2},G^*;  uv) \le \deg_{G^*} u + \deg_{G^*} v$, the first inequality in~\eqref{eq:tria-core-degrees} implies that $e$ is contained in $A_{G^*}$ whereas the second inequality implies that $e$ has an endpoint in $B_{G^*}$, as claimed.

  Our key observation implies that for fixed sets $B \subseteq A \subseteq \br{n}$ with $|A| = a$ and $|B| = b$, there are at most $\binom{a^2/2 + bn}{m}$ cores $\Gcore$ with $m$ edges that satisfy $A_{\Gcore} \subseteq A$ and $B_{\Gcore} \subseteq B$. We can thus (generously) upper bound the number of cores with~$m$ edges by
  \[
    \binom{n}{a}\binom{n}{b}\binom{a^2/2 + bn}{m}.
  \]
  Recalling the inequality $\binom{x}{y} \le (ex/y)^y$, valid for all nonnegative integers $x$ and $y$, we may conclude that the number of cores with $m$ edges is at most
  \[
    n^{\frac{120Cm\log(1/p)}{\eps np}} \cdot \left(
    \frac{e(60C)^2m\big(\log(1/p)\big)^2}{2\eps^2 n^2p^2} +\frac{e60C\log(1/p)}{\eps} \right)^{m}.
  \]
  Since $p\gg n^{-1}\log n$, the first factor is at most $e^{o(m\log(1/p))}$.
  Using $m \leq Cn^2p^2\log(1/p)$, the second factor is at most
  $e^{O(m\log\log(1/p))} = e^{o(m\log(1/p))}$. This shows that the number of
  cores with $m$ edges is indeed at most $(1/p)^{\eps m}$, as claimed.
\end{proof}

\section{The main technical result: `entropic stability implies localisation'}
\label{sec:method}

The goal of this section is to state a general result that allows one, in many cases of interest, to reduce the
problem of determining the precise asymptotics of the logarithmic upper tail probability of a polynomial
(with nonnegative coefficients) of independent Bernoulli random variables to a counting problem.
Since the main technical lemmas also apply to non-product measures on the hypercube, we phrase the
basic definitions in this broader context.

We denote by $Y$ a random variable taking values
in the discrete $N$-dimensional cube $\{0,1\}^N$ and by $X=X(Y)$ a real-valued, increasing
function of $Y$ with positive expectation. Given a~subset $I\subseteq \br N$, we write 
$Y_I=\prod_{i\in I}Y_i$ for the indicator random variable of the event
$\big\{Y_i=1 \text{ for all }i\in I\big\}$. Using the shorthand notation $\e_I[X] = \e[X\mid
Y_I
=1]$,\footnote{Strictly speaking, $\Ex_I[X]$ is well defined only if $\Pr(Y_I=1)>0$.
However, the value of
$\e_I[X]$ for sets $I$ with $\Pr(Y_I=1)=0$
does not affect any of our statements.} we define
a function $\Phi_X\colon \RR \to \RR_{\geq 0}\cup \{\infty\}$ by\footnote{We use the standard convention that $\min \emptyset = \infty$.}
\begin{equation}\label{eq:phi} \Phi_X(\delta) = \min\big\{-\log \Pr(Y_I=1):
I\subseteq \br N \text{ and } \e_I[X]\geq (1+\delta)\e[X]\big\}. \end{equation}
It is easy to see that $\Phi_X$ is a nondecreasing function satisfying
$\Phi_X(\delta) > 0$ for all $\delta>0$. We say that a function $X\colon \{0,1\}^N \to \RR_{\geq 0}$
is a polynomial with nonnegative coefficients and
degree at most $d$ if it admits a representation $X = \sum_{I\subseteq
\br N} \alpha_IY_I$, where each coefficient $\alpha_I$ is nonnegative and
$\alpha_I=0$ whenever $|I| > d$.

Let $\mathcal{I}$ be a collection of subsets $I \subset \br{N}$. Given $\eps>0$ and $p >0$, we say that $\mathcal{I}$ is an {\em entropically stable} family (with respect to $\eps$ and $p$) if, for every integer $m$,
\[
|\{ I \in \mathcal{I}: |I| =m \} | \leq (1/p)^{\varepsilon m /2}.
\]
For the sake of brevity, we will suppress the dependence of this property on $\eps$ and $p$.

The following statement is the main technical result of our work.

\begin{theorem}
  \label{thm:packaged}
  For every positive integer $d$ and all positive real numbers
  $\eps$ and $\delta$ with $\eps<1/2$, there is a positive
  $K=K(d,\eps,\delta)$ such that the following holds. Let $Y$ be a sequence
  of $N$ independent $\Ber(p)$ random variables
  for some $p \in (0, 1-\eps]$ and let
  $X=X(Y)$ be a nonzero polynomial with nonnegative coefficients and
  degree at most $d$ such that
  $\Phi_X(\delta-\eps)\geq K\log(1/p)$. Denote by $\core$ the collection of all
  subsets $I\subseteq \br N$ satisfying
  \begin{enumerate}[label=(C\arabic*)]
    \item\label{item:thmcore-bias}
      $\e_{I}[X] \geq (1+\delta-\eps)\e[X]$,
    \item\label{item:thmcore-size}
    $|I|\leq K \cdot \Phi_X(\delta+\eps)$, and
    \item\label{item:thmcore-mindeg}
      $\min_{i\in I}\left(\e_{I}[X]-\e_{I\setminus \{i\}}[X]\right)
      \geq \e[X]/\big(K \cdot \Phi_X(\delta+\eps)\big)$,
  \end{enumerate}
  and assume $\core$ is an entropically stable family.
  Then
  \begin{equation}\label{eq:thm-rate}
    (1-\eps)\Phi_X(\delta-\eps)\leq -\log\Pr\big(X\geq (1+\delta)\e[X]\big) \leq
    (1+\eps)\Phi_X(\delta+\eps)
  \end{equation}
  and, writing $\mincore$ for the collection of those $I\in \core$ with
  $-\log\Pr(Y_I=1)\leq (1+\eps)\Phi_X(\delta+\eps)$,
  \begin{equation}\label{eq:thm-structure}
    \Pr\big( X\geq (1+\delta)\e[X] \text{ and $Y_I = 0$ for all $I\in \mincore$} \big)
    \leq \eps \Pr\big(X\geq (1+\delta)\e[X]\big).
  \end{equation}
\end{theorem}

\begin{remark}
  Observe that~\eqref{eq:thm-rate} gives tight bounds on the logarithmic upper tail probability of $X$, provided that
  $\Phi_X(\delta) = \big(I(\delta)+o(1)\big) f(N,p)$ for a continuous, positive function $I$ and some function $f$.
  Equation \eqref{eq:thm-structure} states that the upper tail  event is (almost) contained in the event that
  $Y_I=1$ for some $I \in \mincore$; note that each such $I$ is a near-minimiser of $\Phi_X(\delta-\eps)$. In some cases, it is possible to
  classify these near-minimisers and thereby obtain a rough structural characterisation of
  the upper tail event.
\end{remark}

\begin{remark}
  The assumption $\Phi_X(\delta-\eps) \geq K\log (1/p)$ means
  that conditioning on $Y_I=1$ for any constant-size subset $I\subseteq \br N$
  cannot increase the expected value of $X$ by more than $(\delta-\eps) \e[X]$;
  it is very easy to verify this for the applications we have in mind.
  The more onerous task is verifying that $\core$ is an entropically stable family. In fact, a large part of this paper
  is dedicated to counting {\em cores} (as we call the elements of $\core$). Frequently, there are very few minimisers of $\Phi_X(\delta)$,
  for every $\delta > 0$. Entropic stability quantifies the notion that there are few near-minimisers as well.
  
\end{remark}

\begin{remark}
  In the following sections, we will compute the logarithmic upper tail probabilities in various settings
  by estimating the function $\Phi_X$ and verifying that the random variable $X$ in question satisfies the assumptions
  of Theorem~\ref{thm:packaged}. As will be seen in the proof of Theorem~\ref{thm:packaged}, entropic stability implies that
  \begin{equation}\label{eq:NonUnionBound}
  -\log \mathbb{P}\left(Y_{I^*} = 1 \text{ for some } I^* \in \core\right) \geq  (1 - \eps/2) \cdot\Phi_X(\delta -\eps).
  \end{equation} 
  However, there are many natural contexts  where the entropic stability assumption fails despite the fact that \eqref{eq:NonUnionBound} remains true. For example, when $H = C_4$, then every copy of the complete bipartite graph $K_{2,cn^2p^2}$ in $K_n$ is a core, provided that $c$ is large enough. There are $\binom{n}{cn^2p^2}$ such copies in $K_n$ and $\binom{n}{cn^2p^2}$ is larger than
    $(1/p)^{\Omega(n^2p^2)}$ when $p \ll n^{-1/2 - \xi}$ for some small $\xi >0$. In order to study such scenarios, one may search out a weaker condition which still implies \eqref{eq:NonUnionBound}, and employ Theorem~\ref{thm:packaged} {\em mutatis mutandis}. One such modification is to allow the number of cores with $m$ edges to be as large as $(1/p)^{m - m_{\min} +o(m_{\min})}$; such a generalisation was introduced in the work of Basak--Basu~\cite{basak2019upper}.
\end{remark}

\begin{remark}
Let $\mathcal{M}(\{0,1\}^N)$ be the set of measures on the $N$-dimensional hypercube. For any random variable $X = X(Y)$ defined on the $p$-biased hypercube, it is possible to give an abstract description of the probability of the upper tail event via the Gibbs variational principle, which states that 
\[
-\log \Pr\left(X \geq (1 + \delta) \Ex[X] \right) \geq \inf_{\nu \in \mathcal{M}_\delta^X} D_{\text{KL}}( \nu \mathrel{\Vert} \mu_p),
\]
where $\mu_p$ is the product of $N$ Bernoulli measures of mean $p$, $D_{\text{KL}}(\cdot \mathrel{\Vert} \cdot)$ is the relative entropy
\[
  D_{\text{KL}}(\nu \mathrel{\Vert} \mu) = \sum_{y \in \{0,1\}^N} \nu(y) \log \left(\frac{\nu(y)}{\mu(y)}\right),
\]
and 
\[
\mathcal{M}_\delta^X := \Big\{ \nu \in \mathcal{M}(\{0,1\}^N) : \sum_{y \in \{0,1\}^N} X(y) \nu(y) \geq (1 + \delta) \Ex[X]\Big\}.
\] 
The {\em na\"ive mean field approximation} holds for the upper tail of $X$ if one can replace the infimum over all measures in $\mathcal{M}_\delta^X$ by an infimum over all {\em product} measures in $\mathcal{M}_\delta^X$, while incurring only lower order errors.

The seminal work of Chatterjee--Dembo~\cite{Cha17}, further developed by Eldan~\cite{eldan2016gaussian} and Augeri~\cite{augeri2018nonlinear}, provides very general sufficient conditions that imply the na\"ive mean field approximation for a general function $f$ of Bernoulli random variables, stated in terms of the `smoothness' of $f$ and the `complexity' of its gradient. Although the various works consider different notions of low complexity, all of them seem to imply the heuristic statement that the set of all directions of the gradient of $f$ is an extremely sparse subset of the $(N-1)$-dimensional sphere. An alternate approach to the na\"ive mean field approximation is used by Cook--Dembo~\cite{cook2018large}, which specializes to the case of subgraph counts in Erd\H{o}s--R\'enyi random graphs. Instead of appealing to gradient complexity bounds, they construct an efficient covering of (most of) the hypercube by convex sets on which the subgraph counts are nearly constant, in the spirit of the regularity-based approach of Chatterjee--Varadhan~\cite{chatterjee2011large}.

Although its formulation is rather different, Theorem~\ref{thm:packaged} can also be considered in the context of the na\"ive mean field approximation, coverings of the hypercube, and low-complexity gradients. Given a subset $I \subseteq \br N$, we define a product measure $\nu_I$ by
\[
\nu_I(Y_i=1) = \begin{cases}
  1 & \text{if $i\in I$,}\\
  p & \text{otherwise}.
\end{cases}
\]
A straightforward computation shows that $D_{\text{KL}}(\nu_I \mathrel{\Vert} \mu_p)
=  |I|\log(1/p)$, and so
\[ \Phi_X(\delta) = \inf \big\{ D_{\text{KL}}(\nu_I \mathrel{\Vert} \mu_p) : \nu_I\in \mathcal M_\delta^X \big\}. \]
In these terms, Theorem~\ref{thm:packaged} shows that, if $\core$ an entropically stable family, then a particularly simple form of the na\"ive mean field approximation holds: one must only consider product measures that assign edges probability $p$ or $1$. Our adaptation of the high moment argument of Janson, Oleszkiewicz, and Ruciński, used in Lemma~\ref{lem:stability}, constructs a covering of (most of) the upper tail event by a family $\mathcal{I}$ of small subsets $I$ with $\nu_I \in \mathcal{M}_\delta^X$. The extraction of cores from these subsets corresponds to identifying the directions in which the possible partial derivatives are large; in this perspective, entropic stability is analogous to the sparseness property that is encoded by the low-complexity gradient condition.

The final stipulation of Theorem~\ref{thm:genpackaged} gives a structural description of the measure conditioned on the upper tail event. More specifically, a sample from the conditional measure will contain an element of $\mincore$ with high probability. Working from the na\"ive mean-field approximation, one may consider the more subtle question of the exact relationship between the conditional measure and the family of product measures in $\mathcal{M}_\delta^X$ that attain the minimal relative entropy from $\mu_p$. The work of Eldan--Gross~\cite{eldan2018decomposition} shows that, under certain conditions, the conditional measure is close to a mixture of such product measures, in the sense of optimal transport; Austin~\cite{austin2018structure} proves similar results for a broader class of measures (not necessarily on the hypercube). 

\end{remark}

The upper bound on $-\log \Pr\big(X\geq(1+\delta)\e[X]\big)$ stated in~\eqref{eq:thm-rate} will follow easily
from the following simple lemma.

\begin{lemma}
  \label{lem:lower}
  Let $Y$ be a random variable taking values in $\{0,1\}^N$ and let $X =
  X(Y)$ be a real-valued function of $Y$. Suppose that $\Ex[X] > 0$ and that
  $X\leq M$ always. Then for all positive $\eps$ and $\delta$,
  \[
    - \log \Pr\big(X\geq (1+\delta)\e[X]\big) \leq \Phi_X(\delta+\eps) + \log \left(\frac{M}{\eps \e[X]}\right).
  \]
\end{lemma}

\begin{proof}
  Let $t = (1+\delta)\e[X]$.
  If $\Phi_X(\delta+\eps) = \infty$, then
  the assertion of the lemma is vacuous. 
  Otherwise, there exists a set $I\subseteq \br N$ with $-\log {\Pr(Y_I=1)} =
  \Phi_X(\delta+\eps)$ and $\e_I[X] \geq t+\eps\e[X]$.
  As $\e_I[X] \leq M \cdot
  \Pr(X\geq t\mid Y_I=1) + t$, it follows that \[ \Pr(X\geq t) \geq
  \Pr(Y_I=1)
  \cdot \Pr(X\geq t \mid Y_I = 1) \geq \Pr(Y_I=1) \cdot \frac{\eps \e[X]}{M}. \]
  Taking the negative logarithm of both sides gives the assertion of the lemma.
\end{proof}

The next lemma lies at the heart of the matter.  In very broad
terms, it states that the upper tail event $\big\{X \ge (1+\delta)\Ex[X]\big\}$,
viewed as a subset of the cube $\{0,1\}^N$, may be covered almost completely
by a union of subcubes of small codimension, where, crucially, the average
value of $X$ on each of these subcubes is at least $(1+\delta - \eps)\Ex[X]$.
The proof uses a variant of the moment argument of Janson,
Oleszkiewicz, and Ruciński~\cite{janson2004upper}.

\begin{lemma}
  \label{lem:stability}
  Let $Y$ be a random variable taking values in $\{0,1\}^N$ and let $X = X(Y)$
  be a~nonzero polynomial with nonnegative coefficients and degree at most $d$. Then for
  every positive integer~$\ell$ and all positive real numbers $\eps$ and $\delta$,
  \[
    \Pr\big(X \geq (1+\delta)\e[X] \text{ and } Y_I = 0 \text{ for all $I\in \cI$}\big)
    \leq \left(\frac{1+\delta-\eps}{1+\delta}\right)^\ell,
  \]
  where $\cI = \big\{I\subseteq \br N: |I|\leq d\ell \text{ and } \e_I[X]\geq
  (1+\delta-\eps)\e[X]\big\}$.
\end{lemma}

\begin{proof}
  Given $S\subseteq \br N$, let
  $Z_S$ be the indicator random variable of the event that $Y_I=0$ for all
  $I\in\cI$ with $I\subseteq S$.
  Note that $I'\subseteq I$ implies $Z_I\leq Z_{I'}$ and
  let $Z = Z_{\br N}$.
  Since $XZ\geq 0$ and $Z^\ell=Z$, Markov's inequality gives
  \begin{equation}\label{eq:stability2}
    \Pr\big(X \geq (1+\delta)\e[X]\text{ and } Z=1\big)
    = \Pr\big(X Z \geq (1+\delta)\e[X]\big) \leq 
    \frac{\e[X^\ell Z]}{\big((1+\delta)\e[X]\big)^\ell}.
  \end{equation}
  Write $X = \sum_I \alpha_IY_I$,
  where the sum ranges over all subsets $I\subseteq \br N$,
  each coefficient $\alpha_I$ is nonnegative, and  
  $\alpha_I = 0$ unless $|I|\leq
  d$. Then for every $k\in \br{\ell}$,
  \[ 
  \begin{split}
    \e[X^k Z] &= \sum_{I_1,\dotsc,I_k}
  \alpha_{I_1}\dotsb\alpha_{I_k}\e[Y_{I_1}\dotsb Y_{I_k}
  \cdot Z]\\
    &\leq \sum_{I_1,\dotsc,I_k}
  \alpha_{I_1}\dotsb\alpha_{I_k} 
  \e[Y_{I_1}\dotsb Y_{I_k}\cdot Z_{I_1\cup \dotsb \cup I_k}]\\
    &
    \leq \kern-5pt\sum_{I_1,\dotsc,I_{k-1}}
    \kern-6pt
    \alpha_{I_1}\dotsb\alpha_{I_{k-1}} 
    \e[Y_{I_1}\dotsb Y_{I_{k-1}}\cdot Z_{I_1\cup \dotsb \cup I_{k-1}}]
    \cdot \e[X
    \mid Y_{I_1}\dotsb Y_{I_{k-1}}\cdot Z_{I_1\cup \dotsb \cup I_{k-1}}
    =1],
  \end{split}\]
  where we may let the third sum range only over sequences $I_1,\dotsc,I_{k-1}$
  for which the event $\big\{Y_{I_1}\dotsb Y_{I_{k-1}}\cdot Z_{I_1\cup \dotsb \cup I_{k-1}}=1\big\}$
  has a positive probability of occurring. Note that for any such sequence,
  $Y_{I_1}\dotsb Y_{I_{k-1}}\cdot Z_{I_1\cup \dotsb \cup I_{k-1}}
  = Y_{I_1}\dotsb Y_{I_{k-1}}$
  and
  $I_1\cup \dotsb \cup I_{k-1}\notin \cI$.
  Since $|I_1\cup \dotsb \cup I_{k-1}| \leq d(k-1) \leq d\ell$, we have
  \[
    \e[X
    \mid Y_{I_1}\dotsb Y_{I_{k-1}}=1]
    = \e_{I_1\cup \dotsb \cup I_k}[X]
    < (1+\delta-\eps)\e[X],
  \]
  as otherwise $I_1 \cup \dotsb \cup I_{k-1}$ would belong to $\cI$.
  It follows that
  \begin{multline*}
    \sum_{I_1,\dotsc,I_k}
    \alpha_{I_1}\dotsb\alpha_{I_k} 
    \e[Y_{I_1}\dotsb Y_{I_k}\cdot Z_{I_1\cup \dotsb \cup I_k}]\\
    < (1+\delta-\eps)\e[X]\cdot 
    \sum_{I_1,\dotsc,I_{k-1}}
    \alpha_{I_1}\dotsb\alpha_{I_{k-1}} 
    \e[Y_{I_1}\dotsb Y_{I_{k-1}}\cdot Z_{I_1\cup \dotsb \cup I_{k-1}}].
  \end{multline*}
  By induction, we see that $\e[X^\ell Z] 
  < \big((1+\delta-\eps)\e[X]\big)^\ell$. Substituting this inequality into~\eqref{eq:stability2} completes
  the proof.
\end{proof}

The following easy lemma will be used to relate the family $\cI$ from the statement of Lemma~\ref{lem:stability}
to the family $\core$ of cores.

\begin{lemma}
  \label{lem:core}
  Let $Y$ be a random variable taking values in $\{0,1\}^N$ and let $X = X(Y)$
  be a real-valued function of $Y$. Then for every $I \subseteq \br N$ and
  every nonnegative real number $s$, there exists some $I^*\subseteq I$ such that
  \begin{enumerate}[label={(\roman*)}]
  \item
    \label{item:core-bias}
      $\e_{I^*}[X] \geq \e_I[X]-s$ and
  \item
    \label{item:core-mindeg}
      $\min_{i\in I^*}\big( \e_{I^*}[X]-\e_{I^*\setminus \{i\}}[X]\big)\geq
    s/|I|$.
  \end{enumerate}
\end{lemma}

\begin{proof}
  Define a sequence $I = I_0\supseteq I_1\supseteq \dotsb \supseteq I_r = I^*$
  by repeatedly setting $I_{k+1}= I_k\setminus \{i\}$ for some $i\in I_k$ such that
   $\e_{I_k}[X] - \e_{I_k\setminus \{i\}}[X] < s/|I|$, as long as such an $i$ exists. By construction,
  the set $I^*$ satisfies~\ref{item:core-mindeg}. 
  Finally, since $r\leq |I|$, we have
  \[ \e_I[X] - \e_{I^*}[X] = \sum_{k= 0}^{r-1}\big(\e_{I_k}[X] - \e_{I_{k+1}}[X]\big) \leq
  s, \]
which is~\ref{item:core-bias}.
\end{proof}

\begin{proof}[Proof of Theorem~\ref{thm:packaged}]
  Let $t = (1+\delta)\e[X]$.
  We first prove the upper bound in \eqref{eq:thm-rate}. Let $\one$ denote the
  $N$-dimensional all-ones vector. Since $X$ is an
  increasing function of $Y$, we have $X \le X(\one)$ always.
  In particular, Lemma~\ref{lem:lower} implies that
  \[ -\log\Pr(X\geq t) \leq \Phi_X(\delta+\eps) + \log\left(\frac{X(\one)}{\eps
  \e[X]}\right). \]
  As $X$ has degree at most $d$ and nonnegative coefficients, we have $\e[X] \geq X(\one) \cdot p^d$ and thus
  \begin{equation}
    \label{eq:thm-lower}
    -\log\Pr(X\geq t) \leq \Phi_X(\delta+\eps) + \log\big(1/(\eps p^d)\big)
    \leq(1+\eps/8)\cdot\Phi_X(\delta+\eps), 
  \end{equation}
  where the second inequality holds provided that $K$ is sufficiently large, as
  we have assumed that $\Phi_X(\delta+\eps) \geq K\log(1/p)$ and $p\leq
  1-\eps$.

  For the rest of the proof let $\ell = \lceil \eps(3d)^{-1}K \cdot \Phi_X(\delta+\eps) \rceil$ and define
  \[
    \cI = \big\{I\subseteq \br N : |I|\leq d\ell \text{ and } \e_I[X] \geq (1+\delta-\eps/2)\e[X]\big\}.
  \]
  It follows from Lemma~\ref{lem:stability} (invoked with $\eps$ replaced by $\eps/2$) that
  \[
    \Pr(X\geq t\text{ and }Y_I=0\text{ for all $I\in \cI$}) \leq \left(1-\frac{\eps/2}{1+\delta}\right)^{\ell}.
  \]
  Since we have already shown that $\Pr(X\geq t)
  \geq \exp\big(-(1+\eps)\Phi_X(\delta+\eps)\big)$, see~\eqref{eq:thm-lower}, we find that letting $K$ be sufficiently large ensures
  \[
    \Pr(X\geq t \text{ and }Y_I=0\text{ for all $I\in \cI$}) \leq \big(1-\eps/(2+2\delta)\big)^\ell \leq (\eps/2) \cdot \Pr(X\geq t).
  \]
  Note next that every $I\in \cI$ satisfies $|I|\leq d\ell \leq (\eps
  K/2) \cdot \Phi_X(\delta+\eps)$ and hence, by Lemma~\ref{lem:core} applied with $s=
  \eps\e[X]/2$, there is a subset $I^*\subseteq I$ satisfying the conditions \ref{item:thmcore-bias},
  \ref{item:thmcore-size}, and \ref{item:thmcore-mindeg}. It follows that
  \begin{equation}\label{eq:thm-step2}
    \Pr(X\geq t \text{ and }Y_{I^*}=0\text{ for all $I^*\in \core$}) \leq (\eps/2) \cdot
    \Pr(X\geq t).
  \end{equation}
  Let $\core_m := \{I^*\in \core : |I^*| = m\}$ and recall that we assume
  $|\core_m|\leq (1/p)^{\eps m/2}$ for all $m\in \NN$.

  We now prove the upper bound in \eqref{eq:thm-rate}. It follows from
  \eqref{eq:thm-step2} that
  \[
    \Pr(X\geq t)\leq 
    (1-\eps/2)^{-1} \cdot \Pr\big(Y_{I^*}=1\text{ for some $I^*\in \core$}\big).
  \]
  Moreover, the definitions of $\core$ and $\Phi_X(\delta-\eps)$
  imply that every core $I^*\in \core$ satisfies $|I^*|\log(1/p)
  = -\log\Pr(Y_{I^*}=1) \geq \Phi_X(\delta-\eps)$, see~\ref{item:thmcore-bias}. Hence,
  taking the union bound over all cores and using $|\core_m|\leq (1/p)^{\eps m/2}$, we find that
  \[
    \begin{split}
      \Pr(X\geq t) & \leq (1-\eps/2)^{-1} \sum_{I^* \in \core} p^{|I^*|} \le 
      (1-\eps/2)^{-1} \sum_m |\core_m| \cdot p^m \\
      & \le    (1-\eps/2)^{-1}
      \sum_{m= \frac{\Phi_X(\delta-\eps)}{\log(1/p)}}^{\infty}
      p^{(1-\eps/2)m} = \frac{e^{-(1-\eps/2)\Phi_X(\delta-\eps)}}{(1-\eps/2)(1-p^{1-\eps/2})}.
    \end{split}
  \]
  Taking logarithms and using $p\leq 1-\eps$
  and $\Phi_X(\delta-\eps)\geq K\log(1/p)$, we see that a large enough
  choice of $K$ ensures that $-\log \Pr(X\geq t) \geq
  (1-\eps)\Phi_X(\delta-\eps)$, as required.

  Finally, let us prove \eqref{eq:thm-structure}. Using \eqref{eq:thm-step2},
  we obtain
  \[
    \Pr(X\geq t \text{ and }Y_{I^*}=0\text{ for all $I^*\in \mincore$})
    \leq (\eps/2) \cdot \Pr(X\geq t)+ \Pr(Y_{I^*}=1\text{ for some $I^*\in
      \core\setminus \mincore$}).
  \]
  Noting that every $I^*\in \core \setminus \mincore$ satisfies $|I^*|\log(1/p)= -\Pr(Y_{I^*}=1) > (1+\eps)\Phi_X(\delta+\eps)$,
  we may employ a union bound again to show that
  \[
    \Pr(Y_{I^*}=1\text{ for some $I^*\in \core \setminus \mincore$})\leq \sum_{I^* \in \core \setminus \mincore} p^{|I^*|} \leq \frac{e^{-(1-\eps/2)(1+\eps)\Phi_X(\delta+\eps)}}{1-p^{1-\eps/2}}.
  \]
  In order to complete the proof, it now suffices to show that
  \begin{equation}
    \label{eq:core-0-upper}
    \frac{e^{-(1+\eps)(1-\eps/2)\Phi_X(\delta+\eps)}}{1-p^{1-\eps/2}} \le (\eps/2) \cdot \Pr(X\geq t).
  \end{equation}
  To see that this inequality holds, note first that $(1+\eps)(1-\eps/2)> 1+\eps/4$ as $\eps < 1/2$ and therefore, by
  \eqref{eq:thm-lower},
  \[
    e^{-(1+\eps)(1-\eps/2)\Phi_X(\delta+\eps)} \le \Pr(X \ge t) \cdot e^{-(\eps/8) \Phi_X(\delta+\eps)}.
  \]
  As $p\leq 1-\eps$  and $\Phi_X(\delta+\eps)\geq K\log (1/p)$, we can choose $K$ so large that $e^{-(\eps/8) \Phi_X(\delta+\eps)} / (1-p^{1-\eps/2}) \le \eps/2$,
  proving~\eqref{eq:core-0-upper}.
\end{proof}

\section{Arithmetic progressions in random sets of integers}
\label{sec:aps}

Fix an integer $k \ge 3$ and let $X=X_{N,p}^\kap$ be the number of $k$-term arithmetic progressions ($k$-APs) in
the random set $\br{N}_p$. The goal of this section is to study the upper tail
of $X$ in the regime where Theorem~\ref{thm:packaged} is applicable. In
particular, we will prove Theorem~\ref{thm:kap}, which we restate here for
convenience.

\thmkap*

To prove the theorem, we will use Theorem~\ref{thm:packaged} to
relate $-\log\Pr\big(X \geq (1+\delta)\e[X]\big)$ to the solution of the optimisation problem
\[
  \Phi_X(\delta) =
  \min{\big\{|I|\log(1/p): I\subseteq \br N \text{ and }\e_I[X]\geq (1+\delta)\e[X]\big\}}.
\]
More precisely, we shall prove the following statement, which is the main result of this section.

\begin{proposition}\label{prop:kap-ldp}
  For every integer $k \geq 3$ and all positive real numbers $\eps$ and
  $\delta$, there exists a positive constant $C$ such that the following
  holds. Suppose that $N\in \NN$ and $p\in (0,1)$ satisfy 
  $CN^{-1}\log N \leq p^{k/2}\leq 1/C$. Then $X = X_{N,p}^\kap$ satisfies
  \[ (1-\eps)\Phi_X(\delta-\eps)\leq -\log \Pr\big(X\geq (1+\delta)\e[X]\big) \leq
  (1+\eps)\Phi_X(\delta+\eps).\]
\end{proposition}

The variational problem $\Phi_X(\delta)$ is a discretisation of the variational problem considered by Bhattacharya, Ganguly, Shao, and Zhao~\cite{bhattacharya2016upper}. In their setup, one minimizes over the set of all product measures on $\{0,1\}^{\br{N}}$, whereas we only consider `planting' constructions; in other words, we restrict our attention to products of $\Ber(p)$ and $\Ber(1)$ measures. The result below can be easily deduced from~\cite[Theorem 2.2]{bhattacharya2016upper}, but we will reprove it in Section~\ref{sec:EstimatingPhi}, for completeness. 
\begin{proposition}
  \label{prop:kap-rate}
  For every integer $k \geq 3$ and all positive real numbers
  $\eps$ and $\delta$, there exists a positive constant $C$ such that the
  following holds. Suppose that $N\in \NN$ and $p\in (0,1)$
  satisfy 
  $CN^{-1} \leq p^{k/2}\leq 1/C$. Then $X = X_{N,p}^\kap$ satisfies
  \[ 1-\eps \leq \frac{\Phi_X(\delta)}{ \sqrt{\delta }\cdot Np^{k/2}\log(1/p)}
  \leq 1+ \eps. \]
\end{proposition}
Clearly, Propositions~\ref{prop:kap-ldp} and~\ref{prop:kap-rate} imply Theorem~\ref{thm:kap}.

\subsection{Proof outline}
The proof of Proposition~\ref{prop:kap-rate} will be relatively straightforward: On the one hand, since every interval (or, more generally,
  every arithmetic progression) of length $\sqrt{\delta}Np^{k/2}$ contains approximately $\delta \Ex[X]$
  arithmetic progressions of length $k$, we have $\Ex_I[X] \ge (1+\delta-o(1)) \Ex[X]$ for each such interval $I$.
  Consequently, $\Phi_X(\delta) \le (1+o(1)) \sqrt{\delta}Np^{k/2}\log(1/p)$. On the other hand, a simple
  calculation shows that, for every set $I \subseteq \br{N}$ with $O(Np^{k/2})$ elements,
  $\Ex_I[X]- \Ex[X]$ is asymptotically equal to the number of $k$-APs in $I$. Therefore,
  $\Phi_X(\delta) / \log(1/p)$ is   bounded from below (asymptotically) by the minimal size of a set of integers that contains at least $\delta \Ex[X]$ $k$-APs. We will show that this minimum is achieved by an interval, see Theorem~\ref{thm:kap-extremal} below; thus, we conclude that   $\Phi_X(\delta) \ge (1-o(1))\sqrt{\delta}Np^{k/2} \log(1/p)$.

In order to derive Proposition~\ref{prop:kap-ldp} from Theorem~\ref{thm:packaged}, we will need to
  show that the family $\core$ of cores from the statement of the theorem is entropically stable. In order
  to bound the number of cores of a given size, we will first observe that, for every $I \in \core$ and each $i \in I$,
  the difference $\Ex_I[X] - \Ex_{I \setminus \{i\}}[X]$ is asymptotically equal to the number of $k$-APs in $I$ that
  contain the element~$i$. In particular, condition~\ref{item:thmcore-mindeg} and $\Phi_X(\delta) = O\big(Np^{k/2}\log(1/p)\big)$ can be combined to conclude that each element of every cores lies in $\Omega\big(Np^{k/2}/\log(1/p)\big)$
  arithmetic progressions of length $k$ that are fully contained in the core; see Claim~\ref{claim:kap-core-property} below.

  The heart of the proof of the proposition is a counting argument showing that very few sets
  have this combinatorial property. Let us first sketch a simplified version of this argument
  that would be sufficient to prove the proposition under the slightly stronger assumption
  that $N p^{k/2} \gg (\log N)^{k+1}$. 
  Suppose that $I$ is a core of cardinality $m$ and let $I'$
  be a random subset of $I$ with $m/\log(1/p)$ elements. For every $i \in I$, we expect that there will be $\Omega\big(Np^{k/2} / \log(1/p)\big)$ arithmetic progressions of length $k$ in $I$ that contain $i$, and that a $(\log(1/p))^{1-k}$-proportion of these will be contained in $I' \cup \{i\}$.
A standard application of Janson's inequality yields that the above description holds with
  probability very close to one, simultaneously for all $i \in I$. In particular, $I$ contains some
  subset $I'$ with $m/\log(1/p)$ elements and the property that, for every $i \in I \setminus I'$,
  there are at least $\Omega\big(Np^{k/2}/(\log(1/p))^k\big)$ arithmetic progressions of length $k$ that comprise $i$ and
  some $k-1$ elements of $I'$.

  We may now enumerate all possible cores $I$ in two steps: First, there are at most $\binom{N}{m/\log(1/p)} \le \exp\big(O(m)\big)$
  choices of $I'$. Second, since $I'$ intersects at most $O(|I'|^2)$ arithmetic progressions of length $k$ at $k-1$ elements and each $i \in I$ is
  contained in at least $\Omega\big(Np^{k/2}/(\log(1/p))^k\big)$ such progressions, the elements of $I \setminus I'$ must all
  come from a set of size
  \[
    O\left(\frac{|I'|^2(\log(1/p))^k}{Np^{k/2}} \right)\le O\big(m(\log(1/p))^{k-1}\big)
  \]
  that depends solely on $I'$.  Thus, the number of choices for $I \setminus I'$, the remaining elements of the core, is at most
  $\binom{O(m(\log(1/p))^{k-1})}{m} = \exp\big(O(m \log\log(1/p))\big)$.

  In the proof of Proposition~\ref{prop:kap-ldp} below, we give a more refined version of the above argument that allows us to recover the optimal power of the logarithm in the lower bound assumption on~$p$. Instead of constructing cores in two steps, we build them element-by-element. This enables
  a~finer control of the number of choices for each next element, given all the elements chosen so far. Roughly speaking, as we add more
  elements to $I$, the set $I'$ from the previous paragraph is gradually increasing its size.

\subsection{Estimating $\Phi_X$}\label{sec:EstimatingPhi}
As mentioned above,
we use the following extremal result about
the largest number of $k$-APs in a set of integers of a
given size, proved in the case $k=3$ by Green--Sisask~\cite{green2008maximal} and
later extended in~\cite{bhattacharya2016upper} to arbitrary $k \ge 3$; the corresponding
statement in the case where $k \in \{1, 2\}$ is trivial. For a set $I
\subseteq \ZZ$, we denote by $A_k(I)$ the number of $k$-APs in $I$. Recall
that we only count $k$-APs with positive common difference.

\begin{theorem}[\cite{bhattacharya2016upper, green2008maximal}]
  \label{thm:kap-extremal}
  For every positive integer $k$ and $I \subseteq \ZZ$, we
  have $A_k(I) \leq A_k\big(\br{|I|}\big)$.
\end{theorem}

We reproduce the proof here for the sake of completeness.

\begin{proof}
  We prove the statement by induction on $k$. The cases $k=1$ and $k=2$ are trivial as $A_k(I) = A_k\big(\br{|I|}\big)$ for every set $I$, so we may assume that $k \ge 3$. Suppose that $|I| = n$ and let $a_1, \dotsc, a_n$ be the elements of $I$ listed in increasing order. We partition the set of \kap s in $I$ into two parts depending on the location of the $(k-1)$st element. More precisely, we let $m = \lceil (k-2)n/(k-1) \rceil$, let
  \[
    \cA_1 = \big\{(i_1, \dotsc, i_k) \in \br{n}^k : (a_{i_1}, \dotsc, a_{i_k}) \text{ is a \kap\ and } i_{k-1} \le m\big\},
  \]
  and let $\cA_2$ comprise the remaining \kap s (that is, ones with $i_{k-1} > m$). The removal of the $k$th term from a progression in $\cA_1$ maps it to a $(k-1)$-AP contained the set $\{a_1, \dotsc, a_m\}$ and therefore $|\cA_1| \le A_{k-1}(\{a_1, \dotsc, a_m\}) \le A_{k-1}\big(\br{m}\big)$, by the induction hypothesis. On the other hand, we observe that for every $i > m$, there are at most $n-i$ arithmetic progression of length $k$ such that $i_{k-1} = i$ and thus
  \[
    |\cA_2| \le \sum_{i=m+1}^n n-i.
  \]
  In order to complete the proof, it is sufficient to verify that our choice of $m$ ensures that
  \[
    A_{k-1}\big(\br{m}\big) + \sum_{i=m+1}^n n-i = A_k(\br{n}).
  \]
  Indeed, $m$ satisfies the following two inequalities:
  \[
    m + \left\lfloor\frac{m-1}{k-2}\right\rfloor \le n  \qquad \text{and} \qquad m+1 - (k-2)(n-m-1) \ge 1.
  \]
  The first inequality implies that extending any arithmetic progression $(i_1, \dotsc, i_{k-1})$ contained in $\br{m}$ by adjoining to it the element $i_k = 2i_{k-1} - i_{k-2}$ yields a \kap\ contained in $\br{n}$, whereas the second inequality implies that $\sum_{i=m+1}^n n-i$ is precisely the number of \kap s in $\br{n}$ whose $(k-1)$st term exceeds $m$.
\end{proof}

For future reference, let us note that $A_k\big(\br{i}\big) - A_k\big(\br{i-1}\big) = \lfloor \frac{i-1}{k-1} \rfloor$ for all positive integers $i$ and $k \ge 2$ and, consequently,
\begin{equation}
  \label{eq:tk}
  A_k(\br{m}) = \sum_{i=1}^m  \left\lfloor \frac{i-1}{k-1} \right\rfloor =
  \frac{m^2}{2(k-1)} - \frac{(k-1)m}{2} \pm k^2.
\end{equation}
Using Theorem~\ref{thm:kap-extremal}, it is not difficult to compute
the asymptotic value of $\Phi_X(\delta)$ and complete the proof of
Theorem~\ref{thm:kap}.

\begin{proof}[Proof of Proposition~\ref{prop:kap-rate}]
  Without loss of generality, we may assume that $\eps \le 1$.
  Given a subset $I\subseteq \br{N}$, let
  $a_j(I)$ denote the number of $k$-APs in $\br{N}$ that intersect $I$ in
  exactly $j$ elements. Note that
  \begin{equation}
    \label{eq:ExX-ExIX}
    \e_I[X] = \sum_{j = 0}^k a_j(I) p^{k-j} \qquad \text{and that} \qquad
    \Ex[X]  = A_k\big(\br{N}\big) p^k = \sum_{j=0}^k a_j(I) p^k.
  \end{equation}
  It follows that
  $\e_I[X] - \Ex[X] \geq (1-p^k) a_k(I)
  = (1-p^k) A_k(I)$ for every $I \subseteq \br{N}$.
  In particular, whenever $(1-p^k)A_k(\br{m})\geq \delta p^k A_k\big(\br{N}\big)$, then
  $\e_{\br{m}}[X]\geq (1+\delta)\e[X]$.
  Therefore, 
  \[
    \Phi_X(\delta) \le \min\left\{m\log(1/p):
    A_k\big(\br{m}\big) \ge \frac{\delta p^k A_k\big(\br{N}\big)
}{1-p^k}\right\} \le (1+\eps)\cdot\sqrt{\delta } \cdot Np^{k/2} \log(1/p),
  \]
  where the last inequality follows from~\eqref{eq:tk} and our assumption
  $N^2p^k \ge C^2$ for a sufficiently large constant $C$.
  It remains to prove the matching lower bound.

  Suppose that $I$ is a smallest subset of $\br{N}$ with
  $\Ex_I[X] \ge (1+\delta)\Ex[X]$.
  Then~\eqref{eq:ExX-ExIX} implies
  \begin{equation}
    \label{eq:PsiNpkap-lower}
    \delta A_k\big(\br{N}\big) p^k = \delta \Ex[X] \le \Ex_I[X] - \Ex[X] \le
    \sum_{j=1}^k a_j(I) p^{k-j}.
  \end{equation}
  Since every pair of distinct numbers in $\br{N}$ is contained in at most
  $\binom{k}2$ arithmetic progressions of length $k$, it follows that
  $a_1(I) \le |I|  \cdot N k^2$ and $\sum_{j=2}^{k-1} a_j(I)
  \le |I|^2 \cdot k^2$. Since we already know that $|I| \le 2\sqrt{\delta}
  \cdot Np^{k/2}$, inequality~\eqref{eq:PsiNpkap-lower} gives
  \[
    \delta A_k(\br N) p^k \leq 2\sqrt{\delta}k^2N^2 p^{3k/2-1} + 4\delta k^2 N^2p^{k+1}+A_k(I).
  \]
  We now invoke Theorem~\ref{thm:kap-extremal}
  and~\eqref{eq:tk} to obtain
  \[ 
  (1-\eps)\cdot \frac{\delta N^2p^k}{2(k-1)} \leq \delta A_k(\br N) p^k -
  2\sqrt{\delta}k^2N^2 p^{3k/2-1} - 4\delta k^2 N^2p^{k+1}\leq A_k(\br I)
\leq \frac{|I|^2}{2(k-1)},
  \]
  where we use the
  assumptions 
  $p\leq C^{-2/k}$ and $N^2 p^k \ge C^2$
  for a large enough $C$.
  Thus
  $\Phi_X(\delta) = |I| \log(1/p)\ge (1-\eps) \cdot \sqrt{\delta} \cdot N
  p^{k/2} \log(1/p)$,
  as required.
\end{proof}

\subsection{Janson's inequality}
It remains to prove Proposition~\ref{prop:kap-ldp}. The proof uses the
following version of Janson's inequality for hypergeometric random variables.
It follows from the (original version of) Janson's inequality for binomial distributions~\cite[Theorem~1]{janson1990poisson}
and the fact that the median of a binomial random variable whose mean is an integer is
equal to its mean. Our argument is an adaptation of~\cite[Lemma~3.1]{BalMorSamWar2016}.

\begin{lemma}
  \label{lemma:hyper-Janson}
  Suppose that $\{B_\alpha\}_{\alpha \in A}$ is a family of subsets of a
  $t$-element set $\Omega$. Let $s \in \{0, \dotsc, t\}$ and let
  \[
    \mu = \sum_{\alpha \in A} \left(\frac{s}{t}\right)^{|B_\alpha|} \qquad
    \text{and} \qquad \Delta = \sum_{\alpha \sim \beta}
    \left(\frac{s}{t}\right)^{|B_\alpha \cup B_\beta|},
  \]
  where the second sum is over all ordered pairs $(\alpha, \beta) \in A^2$ such
  that $\alpha \neq \beta$ and $B_\alpha \cap B_\beta \neq \emptyset$. Let $S$
  be the uniformly chosen random $s$-element subset of $\Omega$ and let $Z$
  denote the number of $\alpha \in A$ such that $B_\alpha \subseteq S$. Then
  for every $\eps \in (0,1]$,
  \[
    \Pr\big(Z \le (1-\eps)\mu\big) \le 2 \exp\left(-\frac{\eps^2}{2}
    \cdot \frac{\mu^2}{\mu + \Delta}\right).
  \]
\end{lemma}
\begin{proof}
  For every $k \in \{0, \dotsc, t\}$, let $S_k$ be the uniformly chosen random $k$-element subset of $\Omega$ and let $Z_k$ denote the number of $\alpha \in A$ such that $B_\alpha \subseteq S_k$, so that $Z = Z_s$, and note that there exists a natural coupling under which $Z_k \le Z_{k+1}$ for every $k$. Let $S'$ be the $(s/t)$-random subset of $\Omega$, that is the random subset of $\Omega$ formed by keeping each of its elements with probability $s/t$, independently of others, and let $Z'$ denote the number of $\alpha \in A$ such that $B_\alpha \subseteq S'$. Since $\Ex\big[Z' \mid |S'|\big] = Z_{|S'|}$, the stochastic ordering of the $Z_k$s implies that, for any number~$M$, the function $k \mapsto \Pr(Z' \le M \mid |S'| = k)$ is decreasing. Hence,
  \[
    \begin{split}
      \Pr\big(Z' \le (1-\eps) \mu\big) & = \sum_{k=0}^t \Pr\big(Z' \le (1-\eps)\mu \mid |S'| = k\big) \cdot \Pr(|S'| = k) \\
      & \ge \Pr\big(Z' \le (1-\eps)\mu \mid |S'| = s\big) \cdot \Pr(|S'| \le s) \\
      & = \Pr\big(Z \le (1-\eps)\mu\big) \cdot \Pr(|S'| \le s) \ge \Pr(Z \le (1-\eps)\mu)/2,
    \end{split}
  \]
  where the last inequality follows from the well-known fact that if $np$ is an integer, then it is the median of the binomial distribution with parameters $n$ and $p$. We can now invoke the classical version of Janson's inequality and conclude that
  \[
    \Pr\big(Z \le (1-\eps)\mu\big) \le 2 \Pr\big(Z' \le (1-\eps)\mu\big) \le 2 \exp\left(-\frac{\eps^2}{2} \cdot \frac{\mu^2}{\mu + \Delta}\right).\qedhere
  \]
\end{proof}

\subsection{Proof of Proposition~\ref{prop:kap-ldp}}
We may assume without loss of generality that 
$\eps$ is sufficiently small, say $\eps<\min{\{1/2,\delta/2\}}$.
Note also that the case $N\leq 2$ is trivial; indeed, in that case $X$ is identically
zero and thus $\log \Pr(X\geq (1+\delta)\e[X])= 0 = \Phi_X(\delta)$ for every
$\delta\in \RR$. We may therefore assume that $N\geq 3$, which, in turn, implies
that $N^2p^k\geq C^2$.

Denote by $Y_i$ the indicator random variable of the event that $i\in \br N_p$. Then
$Y=(Y_1,\dotsc,Y_N)$ is a vector of independent $\Ber(p)$ random variables and
$X$ is a nonzero polynomial with nonnegative coefficients and degree at most $k$
in the coordinates of $Y$. Let $K = K(k,\eps,\delta)$ be the constant given by Theorem~\ref{thm:packaged}.
The proposition follows once we verify that $X$ satisfies the various assumptions of the
theorem.

First, our assumption on $p$ implies that $p\leq 1-\eps$ whenever $C$ is large enough.
Second, it follows from
Proposition~\ref{prop:kap-rate} and the inequality $N^2p^k\geq C^2$ that, whenever $C$ is large
enough, $\Phi_X(\delta-\eps) \geq \Phi_X(\delta/2)\geq 
K\log(1/p)$. Recall that a subset $I\subseteq \br N$ is called a core if
\begin{enumerate}[label=(C\arabic*)]
  \item\label{item:kapcore-bias}
    $\e_{I}[X] \geq (1+\delta-\eps)\e[X]$,
  \item\label{item:kapcore-size}
    $|I|\leq K \cdot \Phi_X(\delta+\eps)$, and
  \item\label{item:kapcore-mindeg}
    $\min_{i\in I}\left(\e_{I}[X]-\e_{I\setminus \{i\}}[X]\right)
    \geq \e[X]/\big(K \cdot \Phi_X(\delta+\eps)\big)$.
\end{enumerate}
The final assumption of Theorem~\ref{thm:packaged} is that, for every
integer $m$, there are at most $(1/p)^{\eps m/2}$ cores of size $m$.

In order to count the cores, we must first unravel the meaning of 
\ref{item:kapcore-bias}, \ref{item:kapcore-size}, and \ref{item:kapcore-mindeg},
and show that each core enjoys a simple combinatorial property. 
Proposition~\ref{prop:kap-rate} supplies a constant $K' =
K'(K,k,\eps,\delta)$ such that, whenever $C$ is sufficiently large,
\begin{equation}\label{eq:kap-phibound}
  4kK \cdot \Phi_X(\delta+\eps) \leq 
  4kK \cdot (1+\eps)\sqrt{\delta+\eps} \cdot Np^{k/2}\log(1/p) \leq K' \cdot
  Np^{k/2}\log(1/p).
\end{equation}
Given a set $I \subseteq \br{N}$ and an $i \in \br{N}$, we write $A_k(I; i)$
for the number of $k$-term arithmetic progressions in $I \cup \{i\}$ that
contain the element $i$. The proof of the following claim is similar to the argument
used to prove Proposition~\ref{prop:kap-rate}.

\begin{claim}
  \label{claim:kap-core-property}
  For every core $I$ of size $m$ and all $i \in I$,
  \[ A_k(I; i) \ge \frac{Np^{k/2}}{K'\log(1/p)}. \]
\end{claim}

\begin{proof}
  Given an $i \in I$, let $a_j(I;i)$ denote the number of $k$-APs in $\br{N}$
  that intersect $I$ in exactly $j$ elements, one of which is $i$. With this
  notation, $A_k(I; i) = a_k(I;i)$ and we may write $\Ex_I[X] - \Ex_{I \setminus
  \{i\}}[X] = \sum_{j=1}^k a_j(I;i) \cdot p^{k-j}(1-p)$. Since every pair of
  distinct numbers in $\br{N}$ is contained in at most $\binom{k}{2}$
  arithmetic progressions of length $k$, we have $a_1(I;i) \le Nk^2$ and
  $\sum_{j=2}^{k-1} a_j(I;i) \le m k^2$. In particular, as $m \leq 
  K\cdot\Phi_X(\delta+\eps)$ by~\ref{item:kapcore-size}, we get
  \[ \Ex_I[X] - \Ex_{I \setminus \{i\}}[X] \le k^2 N p^{k-1} + k^2K \cdot \Phi_X(\delta+\eps) \cdot p + A_k(I; i). \]
  On the other hand, it follows from \ref{item:kapcore-mindeg} that
  \[ \Ex_I[X] - \Ex_{I \setminus \{i\}}[X] 
    \geq \frac{\e[X]}{K \cdot \Phi_X(\delta+\eps)}.
  \]
  By~\eqref{eq:tk}, we have
  $\e[X] = A_k(\br{N})p^k \ge N^2p^k/(2k)$, since $N\geq C$ and $C$ is large.
  Combining the upper and lower bounds on
  $\e_I[X]-\e_{I\setminus \{i\}}[X]$  and using \eqref{eq:kap-phibound},
  we obtain
  \[
    A_k(I;i) \ge \frac{2Np^{k/2}}{K'\log(1/p)}
    - 
    k^2 N p^{k-1} - kK'Np^{k/2+1}\log(1/p).
  \]
  Since $k\geq 3$ and $p\leq C^{-2/k}$ for a large enough $C$, we deduce the 
  assertion of the claim.
\end{proof}

For the remainder of the proof, fix some integer $m$ satisfying $1\leq m \leq 
K \cdot \Phi_X(\delta+\eps)$ and let $K''$ be a sufficiently large positive constant depending on
$K'$ and $k$ (but not on $C$). For a subset $I'\subseteq \br N$
and an integer $i \in \br{N}\setminus I'$, we shall say that $i$ is \emph{rich}
with respect to $I'$ if
\begin{equation}
  \label{eq:i-rich}
  A_k(I'; i) \ge \frac{Np^{k/2}}{K''\log (1/p)} \cdot
  \left(\frac{|I'|}{m}\right)^{k-1}.
\end{equation}
Moreover, given a sequence $(i_1, \dotsc, i_m)$ of $m$ distinct elements of
$\br{N}$, we shall say that an index $m' \in \br m$ is \emph{rich} if $i_{m'}$
is rich with respect to the set $\{i_1, \dotsc, i_{m'-1}\}$. 

We first observe that for every $I' \subseteq \br{N}$, there are relatively few
integers $i \in \br{N}\setminus I'$ that are rich with respect
to $I'$. 
Indeed, there are at most $k|I'|^2$ arithmetic progressions $P$
of length $k$ in $\br{N}$ for which $|I'\cap P| = k-1$, because any such progression
is determined by its minimal and maximal element in $I'$ and the position in the progression
of the element in $P\setminus I'$. Then
\[
  \left|\big\{ i \in \br{N}\setminus I' : \text{$i$ is rich w.r.t.\
  $I'$}\big\}\right| \cdot \frac{Np^{k/2}}{K''\log(1/p)} \cdot
  \left(\frac{|I'|}{m}\right)^{k-1} \le \sum_{i\in \br N  \setminus I'} A_k(I';i) \le
  k|I'|^2.
\]
Consequently, as $m \leq K' \cdot Np^{k/2}\log(1/p)$ by \eqref{eq:kap-phibound},
\begin{equation}\label{eq:kap-number-rich}
  \left|\big\{ i \in \br{N}\setminus I' : \text{$i$ is rich w.r.t.\
  $I'$}\big\}\right|
  \leq 
  k K'K''\cdot m \big(\log (1/p)\big)^2 \cdot
  \left(\frac{m}{|I'|}\right)^{k-3}
  .
\end{equation}

The key property that allows us to control the number of cores $I$ of size $m$
is that, in a large proportion of orderings of the members of $I$, almost all
indices are rich.
This property implies that, if one builds an (ordered) core element by element,
then, very often, one must choose the next element from the small set of integers
that are rich with respect to the previously chosen ones. From this, it
will be easy to obtain an upper bound on the number of cores of a given size.

\begin{claim}
  \label{claim:rich-orderings}
  Suppose that $I$ is a core of size $m$. Then there are at
  least $m!/2$ orderings $(i_1, \dotsc, i_m)$ of the elements of $I$ such that
  all but at most
  \begin{equation}
    \label{eq:bad-indices}
    \left(\frac{K''\log(1/p)}{Np^{k/2}}\right)^{\frac{1}{k-1}} \cdot m
  \end{equation}
  indices $m'\in \br m$ are rich.
\end{claim}
\begin{proof}
  Let $(i_1, \dotsc, i_m)$ be a uniformly chosen random ordering of the
  elements of $I$. Fix integers $m'\in \br m$ and $i \in I$ and condition on
  the event $\{i_{m'} = i\}$. Under this conditioning, the set $\{i_1, \dotsc,
  i_{m'-1}\}$ is a uniformly random $(m'-1)$-element subset of $I \setminus
  \{i\}$. Therefore, we may use Janson's inequality for the hypergeometric distribution
  (Lemma~\ref{lemma:hyper-Janson}) to get an upper bound for the probability
  that the given $m'$ is not rich. It follows from the definition
  that $m'=1$ is trivially rich, so assume $m'\geq 2$. Let $\cB_i$ be the
  collection of all $(k-1)$-element subsets of $I \setminus \{i\}$ that form a
  $k$-AP with $i$. Define
  \[
    \mu_{m'}(i) = \sum_{B \in \cB_i} \left(\frac{m'-1}{m-1}\right)^{|B|} = A_k(I;
    i) \cdot \left(\frac{m'-1}{m-1}\right)^{k-1}
  \]
  and, writing $B\sim B'$ to mean that $B\neq B'$ and $B\cap B' \neq
  \emptyset$,
  \[
    \Delta_{m'}(i) = \sum_{\substack{B, B' \in \cB_i \\ B \sim B'}}
    \left(\frac{m'-1}{m-1}\right)^{|B \cup B'|}.
  \]
  Since for a given $B \in \cB_i$, there are fewer than $k^3$ sets $B' \in \cB_i$ such that $B \cap B' \neq \emptyset$, we have $\Delta_{m'}(i)  \le \mu_{m'}(i) \cdot k^3$.
  It follows from Claim~\ref{claim:kap-core-property} that
  \[
    \mu_{m'}(i) \geq \frac{Np^{k/2}}{K'
    \log(1/p)} \cdot \left(\frac{m'-1}{m}\right)^{k-1},
  \]
  which, provided that $K''$ is sufficiently large, is at least twice as large as the right-hand side of~\eqref{eq:i-rich} with $|I'| = m'-1$.
  Hence, by Lemma~\ref{lemma:hyper-Janson} with $\eps = 1/2$,
  \[
    \begin{split}
      \Pr\left(\text{$m'$ is not rich} \mid i_{m'} = i\right) & \le
      2\exp\left(-\frac{\mu_{m'}(i)^2}{8\big(\mu_{m'}(i) +
      \Delta_{m'}(i)\big)}\right)\\ &\le
      2\exp\left(-\frac{\mu_{m'}(i)}{9k^3}\right) \\
      & \le 2\exp\left(-\frac{Np^{k/2}}{9k^3K'\log(1/p)} \cdot
      \left(\frac{m'-1}{m}\right)^{k-1}\right).
    \end{split}
  \]
  Since this upper bound is independent of $i$, one may replace the conditional probability above with the unconditional one.
  Letting $\cX \subseteq \br m$ denote the (random) set of non-rich indices, we then find that
  \[
    \Ex\big[|\cX|\big] \le 2 \sum_{m'=2}^{m}
    \exp\left(-\frac{Np^{k/2}}{9k^3 K'\log(1/p)} \cdot
    \left(\frac{m'-1}{m}\right)^{k-1}\right).
  \]
  Since for every $\alpha>0$, we have
  \[
    \sum_{m'=2}^{m} \exp\left(-\left(\alpha\cdot
    \frac{m'-1}{m}\right)^{k-1}\right)
    \leq
    \int_0^\infty e^{-(\alpha x/m)^{k-1}}\, dx
    = \frac{m}{\alpha} \int_0^\infty e^{-y^{k-1}}\, dy
    \leq \frac{2m}{\alpha}
  \]
  we obtain
  \[ \e\big[|\cX|\big] \leq 4\cdot \left(\frac{9k^3 K'
  \log(1/p)}{Np^{k/2}}\right)^{\frac1{k-1}}\cdot m.\]
  The assertion of the claim now follows from Markov's inequality, provided that $K''$ is sufficiently large.
\end{proof}

Equipped with the above facts, we can now prove the desired upper bound on the
number of cores of size $m$. For a set $\cX \subseteq \br m$, let $\cS_m(\cX)$
denote the family of all sequences of $m$ distinct elements of $\br{N}$ such
that every index $m' \notin \cX$ is rich. 
To control the number of sequences in $\cS_m(\cX)$, note that we can pick the
first element of the sequence arbitrarily and, for every subsequent index $m'$,
bound the number of possible values for the $m'$th element of the sequence
either by appealing to \eqref{eq:kap-number-rich}, if $m' \notin \cX$, or simply by $N$, otherwise.
Thus,
\[
  \begin{split}
    \frac{\left| \cS_m(\cX) \right|}{m!}
    &\leq \frac1{m!}\cdot N\cdot N^{|\cX|}
    \cdot
    \prod_{m'=2}^m 
    \left(k K'K''\cdot m \big(\log(1/p)\big)^2 \cdot
    \left(\frac{m}{m'-1}\right)^{k-3} \right)\\
    &\leq N\cdot N^{|\cX|}
    \cdot
    \left(k K'K''\cdot \big(\log(1/p)\big)^2\right)^m
    \cdot
    \prod_{m'=1}^m 
    \left(\frac{m}{m'}\right)^{k-2}.
  \end{split}
\]
Since $\prod_{m'=1}^m \left(\frac{m}{m'}\right)^{k-2} \leq
e^{(k-2)m}$, we find that, whenever $C$ is sufficiently large,
\[
 \frac{\left| \cS_m(\cX) \right|}{m!}
  \leq e^{(|\cX|+1)\log N}\cdot e^{3m \log\log(1/p)}.
\]
Finally, denote by $\core_m$ the set of all cores of size $m$.
Claim~\ref{claim:rich-orderings} implies that
\[
  |\core_m| \le \frac{2}{m!}\sum_{\cX} \big|\cS_m(\cX)\big| \le
  2^{m+1} \cdot \max_\cX \frac{\big|\cS_m(\cX)\big|}{m!},
\]
where the sum and the maximum range over all $\cX \subseteq \br N$ 
of size at most $\left(\frac{K''\log(1/p)}{Np^{k/2}}\right)
^{\frac1{k-1}}m$. Hence,
\[
  |\core_m|
  \leq 2^{m+1} \cdot \exp\left(\left(\frac{K''\log
  (1/p)}{Np^{k/2}}\right)^\frac1{k-1}\cdot m\cdot \log N
  + \log N\right) \cdot e^{3m\log\log(1/p)}.
\]
Since we have assumed that $Np^{k/2} \geq C\log N$
and $p\leq C^{-2/k}$, then, whenever $C$ is sufficiently large, the above inequality implies that
\[
  |\core_m|
  \leq (1/p)^{\eps m/4} \cdot \exp\left(\left(\frac{K''\log
  (1/p)}{Np^{k/2}}\right)^\frac1{k-1}\cdot m\cdot \log N\right)
  \leq (1/p)^{\eps m /2},
\]
where the last inequality can be seen, for example, by distinguishing between the cases $p\leq N^{-1/k}$ (in which case $\log{N} = \Theta(\log(1/p))$) and $p > N^{-1/k}$ (where $Np^{k/2}> N^{1/k}$).
 This completes
the proof of Proposition~\ref{prop:kap-ldp}.

\section{Counting small subgraphs---a graph-theoretic interlude}
\label{sec:graph-theory-preliminaries}

As mentioned in the introduction, this section will collect some graph-theoretical results which will be required to analyze the localized regime of $X_{n,p}^{K_r}$ and $X_{n,p}^H$ for connected, regular graphs $H$ in Sections~\ref{sec:cliques} and~\ref{sec:H}, respectively.

The main goal is to bound the maximum number of embeddings of a given graph $J$ into a larger graph $G$ in terms of the number of vertices and edges in $G$, where
we are interested both in bounding the number of such embeddings
\emph{globally} (i.e., without additional restrictions on the image) and \emph{locally}
(where we require that the image contain a particular edge of $G$). These estimates will play a crucial role in translating conditions \ref{item:thmcore-bias}--\ref{item:thmcore-mindeg} from Theorem~\ref{thm:genpackaged} into structural restrictions on core graphs.

In the first two subsections, we collect results related to bounding the number
of embeddings globally; 
the fractional independence number of a graph (defined below)
plays an important role in this part. The next subsection contains bounds on the number of local embeddings, where the image of the embedding is required to contain a particular edge. In the final subsection, we establish several stability results (in the sense of extremal combinatorics) on graphs allowing a nearly maximal number of embeddings of stars and cliques; these results will be of use when establishing Theorem~\ref{thm:krstructure}.

Recall that
$\Emb(J,G)$ denotes the set of embeddings of $J$ into $G$ and,
for every edge $uv$ of $G$, $\Emb(J,G;uv)$ denotes the subset of $\Emb(J,G)$
containing all embeddings that map an edge of $J$ to $uv$. 

\subsection{Fractional graph theory}

 A \emph{fractional independent set} in a graph $J$ is an assignment $\alpha
  \colon V(J) \to [0,1]$ that satisfies $\alpha_u + \alpha_v \le 1$ for every
  edge $uv$ of $J$. The \emph{fractional independence number} of $J$, denoted
  by $\alpha_J^*$, is the largest value of $\sum_{v\in V(J)} \alpha_v$ among
  all fractional independent sets $\alpha$ in $J$. The following result is folklore; we include a proof for completeness.

\begin{lemma}
  \label{lemma:fractional-duality}
  Every graph $J$ admits a fractional independent set $\alpha$ with $\sum_{v
  \in V(J)} \alpha_v = \alpha_J^*$ such that $\alpha_v \in \{0, \frac12, 1\}$ for
  every $v \in V(J)$. Moreover, there is a partition $V(J) = V_1 \cup V_2$ with
  $|V_1|/2 + |V_2| = \alpha_J^*$ such that $V_1$ can be covered by a collection
  of vertex-disjoint edges and cycles of $J$.
\end{lemma}
\begin{proof}
  Let $J'$ be the bipartite double cover of $J$, that is, the graph with
  vertex set $V(J) \times \{1, 2\}$ whose edges are all pairs $\{(u, 1),
  (v,2)\}$ such that $uv \in E(J)$. Moreover, let $\pi \colon V(J) \times \{1, 2\} \to V(J)$ be the projection
  onto the first coordinate. The K\H{o}nig--Egerv\'ary theorem
  (see, e.g., \cite[Theorem~8.32]{BonMur08} or \cite[Theorem~2.1.1]{Die17})
  implies that $J'$ contains a matching $M'$ and an independent set $I'$ such
  that $|I'| + |M'| = v_{J'}$. Define $\alpha \colon V(J) \to \{0, \frac12, 1\}$
  by letting $\alpha_v = |\pi^{-1}(v) \cap I'| / 2$ for every $v\in V(J)$.
  Since $I'$ is an independent set in $J'$, one can see that
  $\alpha$ is a fractional independent set
  with $\sum_{v \in V(J)} \alpha_v = |I'|/2$. In particular,
  we have
  \begin{equation}\label{eq:fracdual1} 
    v_{J'} - |M'| = 2\sum_{v\in V(J)}\alpha_v\leq 2 \alpha_J^*.
  \end{equation}
  Since $\pi$ induces a projection of $J'$ onto $J$, we can define $M = \pi(M')$ to be the image of the matching
  $M'$. Since $M\subseteq J$, we have $\alpha_J^* \leq \alpha_M^*$. Moreover, as $M'$ is a matching in $J'$,
  we see that $M$ has maximum degree at most two and thus each nontrivial connected component of $M$
  is either a cycle or a path. Let $V_2 \subseteq V(J)$ comprise all isolated vertices of $M$
  and one arbitrarily chosen endpoint of each path of even length; let $V_1 = V(J) \setminus V_2$.
  By construction, each connected component of $M[V_1]$ is either a cycle or a path of odd length.
  Since the fractional independence number of every cycle and every path of odd length is exactly half
  its number of vertices, it follows that
  $\alpha_M^*\leq \alpha^*_{M[V_1]} + |V_2| = |V_1|/2 + |V_2|$.
  It is clear that $V_1$ can be covered with vertex-disjoint edges and cycles of $M$ and thus also of $J$.
  We now claim that $|V_1| \ge |M'|$. To see this, fix a connected component $L$ of $M$ and observe
  that $\pi^{-1}(L)$ has at most $e_L$ edges unless $L$ is a single edge, in which case $\pi^{-1}(L)$ has at most two
  edges. Therefore,
  \[
    e_{\pi^{-1}(L)} \le
    \begin{cases}
      v_L - 1 & \text{if $L$ is a path of length at least two}, \\
      v_L & \text{otherwise}.
    \end{cases}
  \]
  Let $\cC(M)$ denote the nontrivial connected components of $M$. We have
  \[
    \begin{split}
      |M'| & = \sum_{L \in \cC(M)} e_{\pi^{-1}(L)} \le \sum_{L \in \cC(M)} v_L - \1[\text{$L$ is a path of length at least two}] \\
      & \le \sum_{L \in \cC(M)} v_L - \1[\text{$L$ is a path of even length}] = |V_1|.
    \end{split}
  \]
  Consequently, \eqref{eq:fracdual1} shows that
  \[
    |V_1|+ 2|V_2|  = 2v_J - |V_1| \leq 2v_J - |M'| = v_{J'} - |M'|   =
    2\sum_{v\in
  V(J)}\alpha_v
    \leq 2\alpha_J^* \leq 2\alpha_M^* \leq |V_1|+2|V_2|,
  \]
  and so 
  $\sum_{v\in V(J)}\alpha_v =\alpha_J^* = |V_1|/2 + |V_2|$.
\end{proof}

The following lemma is implicit in~\cite[Appendix A]{janson2004upper}.
\begin{lemma}
  \label{lemma:eJ-alphaJ-clique}
  Suppose that $J$ is a nonempty subgraph of a connected, $\Delta$-regular
  graph $H$. Then
  \[ e_J \leq 
  \Delta\cdot(v_J-\alpha_J^*)
  \leq
  \Delta\cdot\alpha_J^*
  . \]
  If the first inequality is tight, then
  \begin{enumerate}[label=(Q\arabic*)]
  \item
    \label{item:QH-regular}
    $J = H$ or
  \item
    \label{item:QH-bipartite}
    $J$ admits a bipartition $V(J) = A \cup B$ such that $\deg_Ja = \Delta$ for all $a \in A$.
  \end{enumerate}
  If both inequalities are tight, then $J = H$.
\end{lemma}

\begin{remark}
  \label{remark:eJ-alphaJ-clique}
  Since every graph $J$ is a subgraph of the complete graph of $v_J$ vertices,
  Lemma~\ref{lemma:eJ-alphaJ-clique} implies that
  \[
    e_J \le (v_J-1) \cdot (v_J - \alpha_J^*).
  \]
  Moreover, equality holds if and only if $J$ is complete, $J$ is empty, or $J = K_{1,v_J-1}$.
\end{remark}

\begin{proof}[Proof of Lemma~\ref{lemma:eJ-alphaJ-clique}] By Lemma~\ref{lemma:fractional-duality}, $J$ has a
  fractional independent set
  $\alpha$ such that $\alpha_v \in \{0,\frac12,1\}$.
  Then
  \begin{equation}
    \label{eq:optimal-alpha}
    e_J \leq \sum_{uv\in E(J)}(2-\alpha_u-\alpha_v) =\sum_{v\in V(J)} (1-\alpha_v) \deg_Jv \le \Delta \cdot \sum_{v\in V(J)} (1-\alpha_v)  = \Delta \cdot (v_J - \alpha_J^*),
  \end{equation}
  which is the first inequality. For the second inequality, note that
  the function $\alpha\colon V(J)\to [0,1]$ defined by
  $\alpha_v = 1/2$ is a fractional independent set, so
  $\alpha^*_J \geq v_J/2$.

  Assume now that $e_J = \Delta\cdot (v_J-\alpha^*_J)$. Then both
  inequalities in~\eqref{eq:optimal-alpha} are equalities;
  this implies
  $\alpha_u + \alpha_v = 1$ for every edge $uv
  \in E(J)$ and $\deg_J(v) = \Delta$ whenever
  $\alpha_v \neq 1$.
  Let $A$, $B$, and $C$ denote the sets of vertices that
  $\alpha$ maps to $0$, $1$, and $1/2$, respectively.
  Each vertex in $A \cup C$ has degree $\Delta$ and each edge of $J$ has
  either both endpoints in $C$ or one endpoint in each of $A$ and $B$. In
  particular, if $C$ is not empty, then it induces a $\Delta$-regular graph
  and hence $C = V(J)$ and $J = H$, as $H$ is connected and $\Delta$-regular.
  Otherwise,
  if $C$ is empty, then $A \cup B$ is a bipartition of $J$ and all vertices
  of $A$ have degree $\Delta$.
  
  Lastly, suppose that $e_J = \Delta\cdot \alpha_J^*$,
  which implies $e_J = \Delta\cdot (v_J-\alpha^*_J)$.
  Let $A,B,C$ be the same partition as above.
  If $C$ is nonempty, then $J$ is $\Delta$-regular, and we are done.
  Otherwise, 
  \[ |A| = e_J/\Delta = 
  \alpha_J^* = v_J - e_J/\Delta = v_J-|A|= |B|. \]
  Therefore, every vertex of $B$ has degree
  $\Delta$ and $J = H$.
\end{proof}

\subsection{Global embedding bounds}\label{sec:GlobalEmbedding}

The main result of this section is the following theorem of Janson, Oleszkiewicz, and
Ruci\'nski~\cite{janson2004upper}. A closely related bound
that does not depend on the number of vertices in $G$ was obtained earlier by
Alon~\cite{alon1981number} (see also~\cite{FriKah98} for a short proof). The dependence on the number of vertices will be essential in the case where $J$ is a (double) star, see Figure~\ref{fig:doublestars}.

\begin{theorem}[{\cite{janson2004upper}}]
  \label{thm:max-copies}
  For every nonempty graph $J$ without isolated vertices and every graph $G$ with $n$ vertices,
  \[
    |\Emb(J, G)| \le (2e_G)^{v_J - \alpha_J^*} \cdot \min\{2e_G, n\}^{2\alpha_J^* - v_J}
  \]
\end{theorem}

We derive Theorem~\ref{thm:max-copies} from Lemma~\ref{lemma:fractional-duality}  and the following result due to Alon~\cite{alon1981number}, which establishes the theorem for the case where $J$ is a cycle.
\begin{lemma}
  \label{lemma:NCell}
  Let $C_\ell$ denote the cycle of length $\ell$.
  For every $\ell \ge 3$ and every graph $G$,
  \[
    |\Emb(C_\ell,G)| \le (2e_G)^{\ell/2}.
  \]
\end{lemma}

\begin{remark}
  If $\ell$ is even, this follows immediately from the
  fact that $C_\ell$ contains a perfect matching of $\ell/2$ edges.
  If $\ell$ is prime, there is also a very short and pretty proof using
  the monotonicity of $L^p$ norms; see~\cite{rivin2002counting} for this proof and
  more precise estimates.
  The proof presented below works for all
  $\ell\geq 3$.
\end{remark}

\begin{proof}[Proof of Lemma~\ref{lemma:NCell}]
  For each edge $e \in E(G)$, denote by $c_e$ the number of copies of $C_\ell$ in $G$ that contain the edge $e$. Since $\sum_{e \in E(G)} c_e = \ell \cdot N(C_\ell, G)$, 
  where $N(C_\ell,G)$ is the number of copies of $C_\ell$ in $G$, it follows from the Cauchy--Schwarz inequality that
  \[
    |\Emb(C_\ell, G)|^2 = \big(2\ell \cdot N(C_\ell,G)\big)^2 \le 2e_G \cdot \sum_{e \in E(G)} 2c_e^2.
  \]
  Let $C_\ell^*$ be the graph obtained from gluing two copies of $C_\ell$ along an edge. In other words, $C_\ell^*$ is obtained from the cycle of length $2\ell-2$ by adding to it one longest chord. Observe that if $(L_1, L_2)$ is an ordered pair of copies of $C_\ell$ in $G$, both containing $e$, then there are at exactly two homomorphisms $\varphi \colon V(C_\ell^*) \to V(G)$ that map the two vertices of degree three in $C_\ell^*$ onto the endpoints of $e$ and the two copies of $C_\ell$ in $C_\ell^*$ onto $L_1$ and $L_2$, respectively. Letting $\Hom(C_\ell^*, G)$ be the collection of all homomorphisms from $C_\ell^*$ to~$G$, we may conclude that
  \[
    \sum_{e \in E(G)} 2c_e^2 \le |\Hom(C_\ell^*, G)| \le |\Hom\big((\ell-1) \cdot K_2, G\big)| \le (2e_G)^{\ell-1},
  \]
  where the second inequality holds because $C_\ell^*$ contains a perfect matching of $\ell-1$ edges.
\end{proof}

\begin{proof}[Proof of Theorem~\ref{thm:max-copies}]
  By Lemma~\ref{lemma:fractional-duality}, there is a partition of $V(J)$ into $V_1$ and $V_2$ such that $|V_1|/2+|V_2| = \alpha_J^*$ and $V_1$ can be covered by a collection $\cC$ of vertex-disjoint edges and cycles of $J$. Let $J'$ be the spanning subgraph comprising the edges and cycles of $\cC$ and one edge incident to every vertex in $V_2$. We claim that
  \[
    |\Emb(J', G)| \le \prod_{C \in \cC} |\Emb(C, G)| \cdot \min\{2e_G, n\}^{|V_2|}.
  \]
  Indeed, every embedding of $J'[V_1]$ decomposes into embeddings of the graphs in $\cC$, and there are at most $\min\{2e_G, n\}$ possible images for every vertex of $V_2$. By Lemma~\ref{lemma:NCell}, for every cycle $C \in \cC$,
  \[
    |\Emb(C,G)| \le (2e_G)^{v_C/2};
  \]
  the same inequality holds when $C$ is a single edge. Since every embedding of $J$ into $G$ is also an embedding of $J'$, we deduce that
  \[
    |\Emb(J, G)| \le \prod_{C \in \cC} (2e_G)^{v_C/2} \cdot \min\{2e_G, n\}^{|V_2|} = (2e_G)^{|V_1|/2} \cdot \min\{2e_G, n\}^{|V_2|}.
  \]
  Since $|V_1|/2 = v_J - \alpha_J^*$ and $|V_2| = 2\alpha_J^* - v_J $, this completes the proof.
\end{proof}

\subsection{Local embedding bounds}\label{section:LocalEmbedding}

We now state three lemmas that bound $|\Emb(J,G;uv)|$ from above.

\begin{lemma}
  \label{lemma:core-edge-regular}
  Suppose that $H$ is a $\Delta$-regular graph. For every graph $G$ and
  each $uv \in E(G)$,
  \[
    |\Emb(H,G; uv)| \le 4e_H \cdot (2e_G)^{\frac{v_H}{2} -
    \frac{2\Delta-1}{\Delta}} \cdot (4\deg_Gu \cdot \deg_G
    v)^{\frac{\Delta-1}{\Delta}}.
  \]
\end{lemma}

\begin{lemma}
  \label{lemma:core-edge-bipartite}
  Suppose that $J$ is a nonempty, connected graph with maximum degree $\Delta$ that admits a bipartition $V(J) = A \cup B$ such that $|A| < |B|$ and $\deg_J a = \Delta$ for every $a \in A$. For every graph $G$ and every $uv \in E(G)$,
  \[
    |\Emb(J,G; uv)| \le e_J \cdot (\deg_G u + \deg_G v) \cdot (2e_G)^{|A|-1} \cdot \big(\min\{e_G, v_G\}\big)^{|B|-|A|-1}.
  \]
\end{lemma}

\begin{lemma}
  \label{lemma:max-copies-bad-edges}
  Suppose that $H$ is a $\Delta$-regular graph. For every graph $G$ and every $G' \subseteq G$,
  \[
    \sum_{uv \in E(G')} |\Emb(H,G; uv)| \le e_H \cdot (2e_G)^{v_H/2} \cdot
    \left(\frac{e_{G'}}{e_G}\right)^{1/\Delta}.
  \]
\end{lemma}

Our proofs of Lemmas~\ref{lemma:core-edge-regular} and~\ref{lemma:max-copies-bad-edges} are relatively straightforward adaptations of the elegant entropy argument of Friedgut and Kahn~\cite{FriKah98} (see also the excellent survey of Galvin~\cite{Gal}). They will be derived from the following somewhat abstract form of the main result of~\cite{FriKah98}. The proof of Lemma~\ref{lemma:core-edge-bipartite} is elementary.

\begin{lemma}
  \label{lemma:Shearer-embeddings}
  Suppose that $H$ is a $\Delta$-regular graph. Let $\cE$ be a family of
  embeddings of $H$ into a graph $G$ and, for every edge $ab$
  of $H$, let
  \[
    \cE_{ab} = \left\{\varphi(ab) : \varphi \in \cE\right\}.
  \]
  Then
  \[
    |\cE| \le \prod_{ab \in E(H)} \big(2|\cE_{ab}|\big)^{1/\Delta}.
  \]
\end{lemma}
\begin{proof}
  Let $\bar\varphi \colon V(H) \to V(G)$ be a uniformly chosen random element
  of $\cE$. Write $H(Z)$ for the entropy of a discrete random variable $Z$
  and observe that $H(\bar\varphi) = \log |\cE|$. Since $H$ is $\Delta$-regular, Shearer's
  inequality~\cite{ChuGraFraShe86} implies that
  \[ 
    H(\bar\varphi) \le \frac{1}{\Delta} \cdot \sum_{ab \in E(H)}
    H\big(\bar\varphi(a), \bar\varphi(b)\big).
  \]
  The random variable $(\bar\varphi(a), \bar\varphi(b))$ can take on
  at most $2|\cE_{ab}|$ values, as it an ordered pair of vertices
  that make up the edge $\varphi(ab)$.
  Using the fact that the entropy of any
  distribution on
  a set is at most that of the uniform distribution on that set,
  it follows that $H\big(\bar\varphi(a), \bar\varphi(b)\big) \le \log
  (2|\cE_{ab}|)$. This implies the assertion of the lemma.
\end{proof}

\begin{proof}[{Proof of Lemma~\ref{lemma:core-edge-regular}}]
  Given an ordered pair $(c,d)$ of adjacent vertices of $H$, let
  $\cE^{(c,d)}$ be
  the family of embeddings $\varphi$ of $H$ into $G$ such that
  $\varphi(c) = u$ and $\varphi(d) = v$.
  For a given edge $ab$ of $H$,
  define $\cE_{ab}^{(c,d)} = \{ \varphi(ab) : \varphi \in \cE^{(c,d)}\}$ as in the
  statement of Lemma~\ref{lemma:Shearer-embeddings}. Observe that
  \[
    \big|\cE^{(c,d)}_{ab}\big| \le
    \begin{cases}
      e_G & \text{if } \{a, b\} \cap \{c, d\} = \emptyset, \\
      \deg_G u & \text{if } \{a, b\} \cap \{c, d\} = \{c\},\\
      \deg_G v & \text{if } \{a, b\} \cap \{c, d\} = \{d\},\\
      1 & \text{if } \{a, b\} = \{c, d\}. \\
    \end{cases}
  \]
  Invoking Lemma~\ref{lemma:Shearer-embeddings} to bound $|\cE^{(c,d)}|$ from
  above and summing over all $2e_H$ pairs $(c,d)$ of adjacent vertices of $H$, the
  claimed upper bound on the number of embeddings of $H$ into $G$ that
  use the edge $uv$ follows.
\end{proof}

\begin{proof}[{Proof of Lemma~\ref{lemma:core-edge-bipartite}}]
  We fix an edge $ab$ of $J$, where $a\in A$ and $b\in B$, and count the
  embeddings $\varphi$ of $J$ into $G$ such that $\varphi(ab) = uv$.
  To this end, we first
  show that $J - b$ contains a matching $M$ that saturates $A$. To see this,
  note that for every nonempty $S \subseteq A$,
  \[
    |S| \cdot \Delta = \sum_{c \in S} \deg_J c \le \sum_{d \in N_J(S)} \deg_J d \le |N_J(S)| \cdot \Delta,
  \]
  yielding $|N_J(S)| \ge |S|$. Moreover, this inequality is strict unless
  the subgraph of $J$ induced by $S \cup N_J(S)$ is
  $\Delta$-regular. However, the latter is impossible since, as $J$ is
  connected, the only $\Delta$-regular subgraph of $J$ could be $J$ itself, but
  our assumption $|A|<|B|$ implies that $J$ is not regular.
  Hence $|N_{J-b}(S)| \ge |N_J(S)| - 1 \ge |S|$, verifying Hall's condition.
  Now, given a matching $M \subseteq J-b$ that saturates $A$, we may bound the
  number of embeddings $\varphi$ as above in the following way. Let $c$ be the
  neighbour of $a$ in $M$. There are at most $\deg_G u + \deg_G v$ embeddings
  of $J[\{a, b, c\}]$ into $G$ that map $ab$ to $uv$. Each of them
  admits at most $(2e_G)^{|M|-1}$ extensions to an embedding of $J[V(M) \cup
  \{b\}]$. Each of those embeddings can be extended to an embedding of $J$ in at most
  $\min\{e_G, v_G\}^{|B \setminus (\{b\} \cup V(M))|}$ ways. Since $|M| = |A|$
  and $|B \setminus (\{b\} \cup V(M))| = |B|-|A|-1$, summing over all
  $ab \in E(J)$ gives the claimed bound on $|\Emb(J, G; uv)|$.
\end{proof}

\begin{proof}[{Proof of Lemma~\ref{lemma:max-copies-bad-edges}}]
  Given an edge $cd$ of $H$, let $\cE^{cd}$ be the family of embeddings of
  $H$ into $G$ that map $cd$ onto an edge of $G'$. Define $\cE_{ab}^{cd} =
  \{ \varphi(ab) : \varphi \in \cE^{cd}\}$ as
  in the
  statement of Lemma~\ref{lemma:Shearer-embeddings} and observe that
  \[
    \big|\cE_{ab}^{cd}\big| \le
    \begin{cases}
      e_{G'} & \text{if } \{a, b\} = \{c,d\}, \\
      e_{G} & \text{otherwise}.
    \end{cases}
  \]
  Invoking Lemma~\ref{lemma:Shearer-embeddings} to bound $|\cE^{cd}|$ from
  above and summing over all edges $cd$ of $H$ gives the claimed upper bound.
\end{proof}

\subsection{Stability results}

Observe that, when $J$ is the complete graph, then Theorem~\ref{thm:max-copies}
yields the upper bound $|\Emb(K_r,G)|\leq (2e_G)^{r/2}$. This is a weak
version of a more precise result due to Erd\H{o}s--Hanani~\cite{Erd62} and also
follows from the Kruskal--Katona theorem
\cite{katona1966theorem, kruskal1963number}.
One can see that the upper bound $(2e_G)^{r/2}$
is asymptotically optimal if $G$ contains a clique comprising all of its edges.
Our next theorem states that, when the upper bound
given in Theorem~\ref{thm:max-copies} is nearly tight, then $G$ resembles such
a  graph, in the sense that it must contain a subgraph of density $1-o(1)$
covering nearly all of its edges. This could be proved by appealing to a
stability version of the Kruskal--Katona theorem due to
Keevash~\cite{keevash2008shadows}. The proof we present below is elementary.

\begin{theorem}
  \label{thm:Kr-stability}
  Suppose that $r \ge 3$. If a graph $G$ satisfies
  \[
    |\Emb(K_r, G)| \ge (1-\eps) \cdot (2e_G)^{r/2},
  \]
  for some $\eps \ge e_G^{-1/2}$, then $G$ has a subgraph $G'$ with minimum degree at least $(1-4\eps^{1/2}) \cdot (2e_G)^{1/2}$.
\end{theorem}
\begin{proof}
  The assertion of the theorem follows once we establish the case $r=3$
  and an analogous property for the path with four vertices (and three edges), which we denote by $P_4$. Indeed, if $r$ is odd,
  then $K_r$ contains a spanning subgraph that is the disjoint union of $K_3$ and a matching of size $(r-3)/2$. Thus,
  \[
    |\Emb(K_r, G)| \le |\Emb(K_3, G)| \cdot |\Emb(K_2,G)|^{(r-3)/2} = |\Emb(K_3, G)| \cdot (2e_G)^{(r-3)/2}
  \]
  and hence $|\Emb(K_3, G)| \ge (1-\eps) \cdot (2e_G)^{3/2}$. Analogously, if $r$ is even, then $K_r$ contains a subgraph
  that is the disjoint union of $P_4$ and a matching of size $(r-4)/2$. Thus,
  \[
    |\Emb(K_r, G)| \le |\Emb(P_4, G)| \cdot |\Emb(K_2,G)|^{(r-4)/2} = |\Emb(P_4, G)| \cdot (2e_G)^{(r-4)/2},
  \]
  which implies that $|\Emb(P_4, G)| \ge (1-\eps) \cdot (2e_G)^2$. Therefore, it suffices to prove the following two claims.
  \renewcommand{\qedsymbol}{}
\end{proof}

\begin{claim}
  \label{claim:stability-P4}
  If $|\Emb(P_4,G)| \ge (1-\eps) \cdot  (2e_G)^2$, for some positive $\eps$, then $G$ has a subgraph $G'$ with minimum degree at least $(1-2\eps^{1/2})(2e_G)^{1/2}$.
\end{claim}

\begin{claim}
  \label{claim:stability-K3}
  If $|\Emb(K_3,G)| \ge (1-\eps) \cdot  (2e_G)^{3/2}$, for some $\eps \ge e_G^{-1/2}$, then $G$ has a subgraph $G'$ with minimum degree at least $(1-4\eps^{1/2})(2e_G)^{1/2}$.
\end{claim}

\begin{proof}[Proof of Claim~\ref{claim:stability-P4}]
  We may assume that $\eps < 1/4$, as otherwise the assertion of the claim is trivially satisfied. Let $F$ be the graph with vertex set $E(G)$
  whose edges are all pairs $\{uv, xy\}$ such that the set $\{u, v, x, y\}$ induces a $K_4$ in $G$. Let $\1_G$ be the indicator function of
  the edge set of $G$ and note that
  \[
    \begin{split}
      |\Emb(P_4,G)| & = \sum_{\substack{\{uv, xy\} \subseteq E(G) \\ \{u, v\} \cap \{x, y\}
      = \emptyset}} 2 \cdot \big(\1_G(ux) + \1_G(uy) + \1_G(vx) +
      \1_G(vy) \big) \\
      & \le 8e_F + 6\left(\binom{e_G}{2} - e_F\right) \le 3e_G^2 + 2e_F.
    \end{split}
  \]
  In particular, our assumption implies that
  $e_F \ge (1/2-2\eps) \cdot e_G^2$.

  Let $F'$ be the subgraph obtained from $F$ by iteratively removing vertices whose degree is smaller than $(1-2\eps^{1/2}) \cdot e_G$.
  We claim that fewer than $2\eps^{1/2} \cdot e_G$ vertices are removed this way and, consequently, the
  graph $F'$ is nonempty and its minimum degree is at least $(1-2\eps^{1/2}) \cdot e_G$. Suppose that this
  were not true. We would then have
  \[
    e_F \le \binom{(1-2\eps^{1/2}) \cdot e_G}{2} + 2\eps^{1/2} \cdot e_G \cdot (1-2\eps^{1/2})\cdot e_G < \left( \frac{1}{2} - 2\eps\right) \cdot e_G^2 \le e_F,
  \]
  a contradiction.

  Finally, let $G'$ be the subgraph of $G$ induced by the set of endpoints of the edges from $V(F')$.
  Let $u$ be an arbitrary vertex of $G'$. There must
  be another vertex $v$ of $G'$ such that $uv \in V(F')$. Since $\deg_{F'} uv
  \geq (1-2\eps^{1/2})\cdot e_G$, the
  common neighbourhood of $u$ and $v$ in $G'$ induces a subgraph with at
  least $(1-2\eps^{1/2}) \cdot e_G$ edges in $G$. In particular,
  \[
    \deg_{G'}u \ge \deg_{G'}(u,v) \ge \sqrt{2 \cdot (1-2\eps^{1/2}) \cdot e_G} \ge (1-2\eps^{1/2}) \cdot (2e_G)^{1/2}.
  \]
  Since $u$ was arbitrary, we obtain the desired lower bound on the minimum degree of $G'$.
\end{proof}

\begin{proof}[Proof of Claim~\ref{claim:stability-K3}]
  We may assume that $\eps < 1/16$, as otherwise the assertion of the claim is trivially satisfied.
  For every edge $e \in E(G)$, let $t_e$ denote  the number of copies of $K_3$ in $G$ that contain the edge $e$.
  Observe that, for each $e \in E(G)$, there are at least $2t_e(t_e-1)$ embeddings of $P_4$
  into $G$ that map the middle edge of $P_4$ onto $e$. Since $\sum_{e \in
  E(G)} t_e = |\Emb(K_3,G)|/2$ and the function $t \mapsto 2t(t-1)$ is convex, we
  conclude that
  \[
    |\Emb(P_4, G)| \ge \sum_{e \in E(G)} 2t_e(t_e-1) \ge e_G \cdot
    \frac{|\Emb(K_3,G)|}{e_G} \cdot \left(\frac{|\Emb(K_3,G)|}{2e_G} - 1\right).
  \]
  Our assumptions imply that
  \[
    1 \le \eps \cdot e_G^{1/2} \le \eps \cdot \frac{(1-\eps) \cdot (2e_G)^{3/2}}{2e_G} \le \eps \cdot \frac{|\Emb(K_3, G)|}{2e_G}
  \]
  and consequently,
  \[
    |\Emb(P_4, G)| \ge (1-\eps) \cdot \frac{|\Emb(K_3, G)|^2}{2e_G} \ge (1-\eps)^3 \cdot
    (2e_G)^2 \ge (1-3\eps) \cdot (2e_G)^2,
  \]
  It now follows from Claim~\ref{claim:stability-P4} that $G$ contains a subgraph $G'$ with minimum degree at least
  \[
    \big(1-2 \cdot (3\eps)^{1/2}\big) \cdot (2e_G)^{1/2} \ge (1-4\eps^{1/2}) \cdot (2e_G)^{1/2},
  \]
  as claimed.
  \qedhere\,\qedsymbol
\end{proof}

Our next lemma gives a tight upper bound on the number of stars $K_{1,s}$ in a
given bipartite graph, as well as a structural characterisation of
the bipartite  graphs that
are close to achieving this bound. This lemma and Theorem~\ref{thm:Kr-stability} above
constitute the main combinatorial ingredient in the proof of Theorem~\ref{thm:krstructure}.
Given a graph $G$ and a set $U$ of vertices of $G$, we let $\Emb_U(K_{1,s}, G)$ denote the set
of embeddings of $K_{1,s}$ into $G$ that map the centre vertex to a vertex of $U$.

\begin{lemma}\label{lem:stars}
  Let $s \geq 2$ be an integer and suppose that $G$ is a bipartite graph with parts $U$ and $V$ and at most $q|V|$ edges, for some $q \in (0, |U|]$.
  Then the following holds:
  \begin{enumerate}[label=(\roman*)]
  \item
    \[
      |\Emb_U(K_{1,s}, G)| \leq \big(\floor{q}+\{q\}^s\big)|V|^s.
    \]
  \item
    For every positive $\eps$, there exists a positive $\eta$ such that, if
    \[
      |\Emb_U(K_{1,s}, G)| \geq (1-\eta)\cdot \big(\floor{q}+\{q\}^s\big)|V|^s,
    \]
    then there is a subset $W\subseteq U$ of size $\ceil{q}$ such that
      $e_G(W,V) \geq (1-\eps)q|V|$, and a further subset $W' \subseteq W$ of
      size at least $\lfloor (1 - \eps) |W| \rfloor$ such that $\deg_Gu \ge (1 -
      \eps)|V|$ for every $u \in W'$.
  \end{enumerate}
\end{lemma}

\begin{proof}
  We will use the following inequality, valid for any two numbers $x$ and $y$ with $x \ge y \ge 1$:
  \begin{equation}
    \label{eq:xs-convexity}
    (x+1)^s + (y-1)^s - x^s - y^s \ge (x+1-y) \cdot (x+1)^{s-2}.
  \end{equation}
  It is clear that
  \begin{equation}
    \label{eq:NUK1s-upper}
    |\Emb_U(K_{1,s}, G)|
    = \sum_{u \in U} (\deg_Gu) (\deg_Gu-1)\dotsb (\deg_Gu - s+1)\leq \sum_{u \in U} (\deg_Gu)^s.
  \end{equation}
  Let $m = |U|$ and, given a sequence $\bd = (d_1, \dotsc, d_m)$, define
  \[
    S(\bd) = \sum_{i=1}^m d_i^s.
  \]
  By the degree sequence of a
  bipartite graph with parts $U$ and $V$, we will mean the sequence of degrees
  of the vertices in $U$, listed in a nonincreasing order.
  Thus~\eqref{eq:NUK1s-upper} implies that $|\Emb_U(K_{1,s}, G)| \le S(\bd_G)$,
  where $\bd_G$ is the degree sequence of $G$.   Let $m_0 = \floor{e_G/|V|}$ and define $\bdmax = (d_1^*, \dotsc, d_m^*)$ by
  \[
    d_i^* =
    \begin{cases}
      |V| & \text{if $i \le m_0$}, \\
      \{e_G / |V|\} \cdot |V| & \text{if $i = m_0 + 1$}, \\
      0 & \text{otherwise}.
    \end{cases}
  \]
  Note that $\sum_{i=1}^m d_i^* = e_G$; in particular, $\{e_G/|V|\}\cdot |V|$
  is an integer. We claim that $\bdmax$ maximises $S$ over all degree sequences
  whose sum is $e_G$.
  Indeed, for any other such degree sequence $\bd' = (d_1', \dotsc, d_m')$,
  there must be two distinct indices $i$ and $j$ such that $0 < d_i' \le d_j'
  < |V|$. Let $\bd''$ be the degree sequence obtained from $\bd'$ by
  decreasing $d_i'$ by one and increasing $d_j'$ by one (and reordering the
  degrees, if necessary). It follows from~\eqref{eq:xs-convexity} that
  \[
    S(\bd'') - S(\bd') = (d_j'+1)^s + (d_i'-1)^s - (d_j')^s - (d_i')^s \ge (d_j'-d_i'+1) \cdot (d_j'+1)^{s-2} \geq 1.
  \]
  Therefore,
  \[
    |\Emb_U(K_{1,s}, G)|\leq S(\bd_G) \le S(\bdmax) = \big(\floor{e_G/|V|} + \{e_G/|V|\}^s\big)|V|^s .
  \]
  Since $e_G \le q |V|$, this completes the proof of the first part of the lemma.

  \begin{figure}
    \usetikzlibrary{decorations.pathreplacing}
    \begin{tikzpicture}[scale=0.5]
      \fill[black!20] (0,0) rectangle +(1,4);
      \fill[black!10] (1,0) rectangle +(1,4);
      \fill[black!20] (1,0) rectangle +(1,3);
      \fill[black!10] (2,0) rectangle +(1,4);
      \fill[black!20] (2,0) rectangle +(1,2);
      \fill[black!20] (3,0) rectangle +(1,1);
      \foreach \i in {0,...,8} {
        \draw (\i,0) rectangle ++(1,1)
         rectangle ++(-1,1)
         rectangle ++(1,1)
         rectangle ++(-1,1);
      }
      \node[anchor=south] at (0.5,-1) {$1$};
      \node[anchor=south] at (2.5,-1) {$\smash{m_0}$};
      \node[anchor=south] at (8.5,-1) {$m$};
      \draw [decorate,decoration={brace,amplitude=5pt,mirror,raise=5pt}]
        (9,0) -- (9,4) node [black,midway,xshift=0.7cm] {$|V|$};
      \node at (-1.5,2) {$\longrightarrow$};
      \begin{scope}[xshift=-12cm]
        \fill[black!20] (0,0) rectangle +(1,4);
        \fill[black!20] (1,0) rectangle +(1,3);
        \fill[black!20] (2,0) rectangle +(1,2);
        \fill[black!10] (3,0) rectangle +(1,2);
        \fill[black!20] (3,0) rectangle +(1,1);
        \fill[black!10] (4,0) rectangle +(1,1);
        \fill[black!10] (5,0) rectangle +(1,1);
        \fill[black!10] (6,0) rectangle +(1,0);
        \fill[black!10] (7,0) rectangle +(1,0);
        \foreach \i in {0,...,8} {
          \draw (\i,0) rectangle ++(1,1)
          rectangle ++(-1,1)
          rectangle ++(1,1)
          rectangle ++(-1,1);
        }
        \node[anchor=south] at (0.5,-1) {$1$};
        \node[anchor=south] at (2.5,-1) {$\smash{m_0}$};
        \node[anchor=south] at (8.5,-1) {$m$};
        \draw [decorate,decoration={brace,amplitude=5pt,mirror,raise=5pt}]
        (0,4) -- (0,-0) node [black,midway,xshift=-0.7cm] {$|V|$};
      \end{scope}
    \end{tikzpicture}
    \caption{The degree sequence on the left can be turned into
    the degree sequence on the right by moving mass from
    columns $i\geq m_0+1$ to columns
    $j\leq m_0+1$.
    }\label{fig:dmax}
  \end{figure}
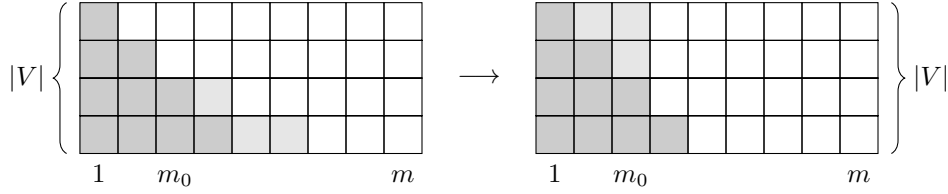

  For the second stipulation of the lemma, fix a positive $\eps$.
  Let $\Omega$ be the set of degree sequences of bipartite graphs $G$
  with parts $U$ and $V$ for which there is a subset $W \subseteq U$ of size
  $\ceil{e_G/|V|}$ such that $e_G(W, V) \ge (1-\eps) e_G$ and, additionally,
  a further
  subset $W' \subseteq W$ of size $\lfloor (1-\eps)|W| \rfloor$ such that
  $\deg_G u \ge (1-\eps)|V|$ for each $u \in W'$.

  It suffices to show that any degree
  sequence $\bd \notin \Omega$ whose sum is $e_G$ satisfies $S(\bd) < (1-\eta)
  S(\bdmax)$, for some positive $\eta$ that depends only on $\eps$. Let
  $\bd = (d_1, \dotsc, d_m)$. The crucial observation 
  is that we may obtain $\bdmax$ from $\bd$ by successively
  increasing some $d_i$ with $i \le m_0+1$ by one and, simultaneously,
  decreasing some $d_j$ with $j \ge m_0+1$ by one (see Figure~\ref{fig:dmax}).
  Note that, when doing so, we perform at least
  \[
    \sum_{i=1}^{m_0} (|V| - d_i) + \max\big\{d_{m_0+1}^* - d_{m_0+1}, 0\big\}
  \]
  such operations.
  We split further analysis into two cases.

  First, assume that $\sum_{i=1}^{m_0} (|V| - d_i) \ge \eps^2 e_G$; this implies
  that $m_0\geq 1$. In this case, in at
  least $\eps^2 e_G / 2$ steps
  of the above procedure, we will be increasing some $d_i$ with $i \le m_0$
  which is, at this time, already at least $\floor{(|V| + d_i) / 2}
  \geq 
  (|V| + d_{m_0}) / 2 - 1
  $, while decreasing
  some $d_j$ with $j \ge m_0+1$ which is at most $d_{m_0}$.
  Inequality~\eqref{eq:xs-convexity} implies that
  \[
    S(\bdmax) - S(\bd) \ge \frac{\eps^2 e_G}{2} \cdot \left(\frac{|V| + d_{m_0}}{2} - d_{m_0}\right) \cdot \left(\frac{|V| + d_{m_0}}{2}\right)^{s-2}.
  \]
  Since our assumption implies that $|V| - d_{m_0} \ge \eps^2 e_G / m_0$, it
  follows that
  \[
    S(\bdmax) - S(\bd) \ge \frac{\eps^2 e_G}{2} \cdot \frac{\eps^2 e_G}{2m_0} \cdot \left(\frac{|V|}{2}\right)^{s-2} = \frac{\eps^4}{2^{s+1}} \cdot \left(\frac{e_G}{m_0|V|}\right)^2 \cdot 2m_0|V|^s.
  \]
  Finally, we note that $e_G \ge m_0|V|$ and that $S(\bdmax) \le (m_0+1) |V|^s
   \le 2m_0 |V|^s$, and thus
  \[
    S(\bdmax) - S(\bd) \ge \frac{\eps^4}{2^{s+1}} \cdot S(\bdmax),
  \]
  proving $S(\bd) < (1-\eta)S(\bdmax)$ for some positive $\eta$.

  Assume now that $\sum_{i=1}^{m_0} (|V| - d_i) < \eps^2 e_G$. Let $W$ be the
  set comprising the $\lceil e_G / |V| \rceil$ vertices with largest degrees in
  $U$. Suppose first that $e_G(W,V) \ge (1-\eps)e_G$, but every set of
  $\floor{(1-\eps)|W|}$ vertices of $W$ contains a vertex of degree smaller
  than $(1-\eps)|V|$. In this case, $(1-\eps)|W| \ge 1$, as otherwise the
  latter condition is vacuously false, and $d_i < (1-\eps)|V|$ whenever $i \ge
  \floor{(1-\eps)|W|}$. Then
  \[
    \sum_{i=1}^{m_0} (|V|-d_i) > \big(m_0+1-\floor{(1-\eps)|W|}\big) \cdot
    \eps |V| \ge \big( |W| - \floor{(1-\eps)|W|}\big) \cdot \eps
    |V| \ge \eps^2 e_G,
  \]
  contradicting our assumption. Thus, since
  $\bd \notin \Omega$,
  we may assume that $e_G(W, V) < (1-\eps)e_G$. Therefore,
  \[
    (1-\eps)e_G > \sum_{i=1}^{|W|} d_i \ge \sum_{i=1}^{m_0} d_i = m_0 |V| -
    \sum_{i=1}^{m_0} (|V|-d_i) > m_0|V| - \eps^2 e_G.
  \]
  This means, in particular, that $m_0 < e_G/|V|$ and hence $|W| = m_0+1$. Moreover,
  \[
    m_0 |V| + d_{m_0+1}^* = e_G > \sum_{i=1}^{|W|}
    d_i + \eps e_G  > m_0|V| + d_{m_0+1} + (\eps-\eps^2) e_G,
  \]
  which implies that $d_{m_0+1}^* - d_{m_0+1} > (\eps-\eps^2) e_G$. Therefore,
  there exist
  at least $(\eps-\eps^2) e_G/2$ steps in which 
  we increase $d_{m_0+1}$ at a time where
  it is already at least $\floor{(\eps-\eps^2) e_G /2}$.
  Inequality~\eqref{eq:xs-convexity} implies that
  \[
    S(\bdmax) - S(\bd) \ge \left(\frac{(\eps-\eps^2) e_G}{2}\right)^s.
  \]
  However, we trivially have $S(\bdmax) \le e_G^s$ and thus
  \[
    S(\bdmax) - S(\bd) \ge \frac{(\eps-\eps^2)^s }{2^s} \cdot S(\bdmax),
  \]
  completing the proof.
\end{proof}

We remark that the extremal structures given by 
Theorem~\ref{thm:Kr-stability} and Lemma~\ref{lem:stars}(ii) are quite
different and, in a sense, incompatible. This has the following technically important consequence: if a graph
$G$ simultaneously
contains many copies of $K_r$ and many copies
of $K_{1,r-1}$, then it can be split into two edge-disjoint graphs, one
containing nearly all the copies of $K_r$ and the other containing nearly all
the copies of $K_{1,r-1}$. The following lemma formalises this statement;
its proof is similar to an argument of Lubetzky and Zhao~\cite{lubetzky2017variational}. We write $G[U,V]$ for the bipartite induced subgraph of $G$ with parts $U$ and $V$.

\begin{lemma}\label{lem:lz}
  For every integer $r\geq 3$ and positive real number $\eps$, there is a
  positive $\eta$ such that the following holds. Let $G$ be a graph on $n$
  vertices with $e_G \leq \eta n^2$. Then there exists a partition $V(G) =
  U\cup V$ satisfying $|U|\leq \eps n$,
  \[ |\Emb(K_r,G[V])|
  \geq 
  |\Emb(K_r,G)|
  - \eps e_G^{r/2},\]
  and\[
    |\Emb_U(K_{1,r-1},G[U,V])| \geq 
    |\Emb(K_{1,r-1},G)| 
    - \eps e_G n^{r-2}. \]
\end{lemma}
\begin{proof}
  Assume that $\eta$ is sufficiently small,  let $U$ be the set of vertices in
  $G$ with degree at least $\eta^{1/3} n$, and let $V = V(G)\setminus U$.
  Note that $|U|\leq 2 e_G/(\eta^{1/3} n) \le 2\eta^{2/3}n \le \eps n$.

  Every embedding of $K_r$ into
  $G$ that maps a vertex of $K_r$ to a vertex of $U$ can be specified
  by first choosing a vertex $v$ of $K_r$, then a vertex of $U$
  that $v$ will be mapped to, and finally an embedding of $K_{r-1}$
  into $G$. 
  Using Theorem~\ref{thm:max-copies}, we thus obtain
  \[ |\Emb(K_r,G)|-
  |\Emb(K_r,G[V])|\leq
  r\cdot |U| \cdot |\Emb(K_{r-1},G)|
  \leq r\cdot \frac{2e_G}{\eta^{1/3} n} \cdot (2e_G)^{(r-1)/2}. \]
  Since $e_G^{1/2}\leq \eta^{1/2} n$,
  this implies the first assertion of the lemma.

  Next, note that
  \begin{equation}\label{eq:ek1}
    |\Emb(K_{1,r-1},G)| 
    =
    |\Emb_U(K_{1,r-1},G[U,V])|  +
    t_1 + t_2,
  \end{equation}
  where $t_1$ is the number of embeddings of $K_{1,r-1}$ into $G$ that map the
  centre vertex and at least one leaf of $K_{1,r-1}$ to $U$ and $t_2$ is the
  number of embeddings that map the centre vertex of $K_{1,r-1}$ to $V$. We have
  \[ t_1 \leq (r-1)\cdot |U|^2\cdot n^{r-2}
  \leq (r-1)\left(\frac{2e_G}{\eta^{1/3} n} \right)^2n^{r-2}\leq \eps
  e_Gn^{r-2}/2, \] as $e_G\leq \eta n^2$.
  Finally, in order to bound $t_2$, observe that
  every embedding counted by $t_2$
  can be specified by first choosing
  a leaf $a$ of $K_{1,r-1}$,
  then choosing the image $e$ of the edge of $K_{1,r-1}$ incident with $a$,
  then choosing the endpoint $v \in V$ of $e$ that is the image of the centre
  vertex of $K_{1,r-1}$, and finally choosing the images of the remaining $r-2$ leaves
  of $K_{1,r-1}$ among the neighbours of $v$ in $G$.
  Since every vertex $v \in V$ has degree at most $\eta^{1/3} n$, it follows that
  \[t_2\leq (r-1)\cdot e_G\cdot 2 \cdot (\eta^{1/3} n)^{r-2}
  \leq \eps e_Gn^{r-2}/2.
  \] 
  Together with \eqref{eq:ek1}, these bounds on $t_1$ and $t_2$ imply the
  second assertion of the lemma.
\end{proof}

\section{Cliques in random graphs}
\label{sec:cliques}

Fix an integer $r \ge 3$ and let $X = X_{n,p}^{K_r}$ be the number of
$r$-vertex cliques in the random graph $G_{n,p}$. In this section, we
shall use Theorem~\ref{thm:packaged} not only to determine the
logarithmic upper tail probability of $X$ but also to provide a
detailed description of the upper tail event. Before we restate the two
theorems that will be proved in this section, we discuss the combinatorial
constructions that are responsible for the localisation phenomenon in more
detail.

As was shown in~\cite{lubetzky2017variational}, when $n^{-1} \ll p^{(r-1)/2} \ll 1$, there
are essentially two optimal strategies for planting
a subgraph inside $G_{n,p}$ that increases the expected number of copies of $K_r$ by
the required $\delta \Ex[X]$.
The first, and most
straightforward, involves planting a clique with $\delta^{1/r} n p^{(r-1)/2}$ vertices; note that our
assumption on $p$ implies that this expression is unbounded and thus we may
implicitly assume that it is an integer. 
Note that such a clique
has close to
\[ \frac{\delta^{2/r} n^2 p^{r-1}}2 \]
edges
and contains approximately $\delta\binom{n}{r}p^{\binom{r}{2}} = \delta \e[X]$
copies of $K_r$. If $np^{r-1}$ is bounded from below, then there is an alternative
strategy that
competes with planting a clique. By a \emph{hub}
of order $\floor{\delta np^{r-1}/r}+1$, 
we mean a subgraph of
$G_{n,p}$ constructed as follows. Let $U$ be a set of $\lfloor \delta
np^{r-1}/r \rfloor$ vertices of $G_{n,p}$ and let $u$ be another vertex that
lies outside of $U$. Connect every vertex in $U$ to every vertex outside of $U
\cup \{u\}$ and connect $u$ to some $\{ \delta n p^{r-1} /
r\}^{1/(r-1)} \cdot n$ such vertices.
Note that every hub
has close to 
\[
  \Big(\lfloor \delta n p^{r-1} / r \rfloor + \{ \delta n p^{r-1}/r
  \}^{1/(r-1)}\Big) \cdot n
\]
edges, which is $\Theta(n^2p^{r-1})$, as $np^{r-1}$ is bounded from
below. Unlike in the
previous construction, the hub itself 
contains \emph{no} copies of $K_r$.
However, as $p \ll 1$, planting a hub creates approximately
$\lfloor \delta n p^{r-1}/r \rfloor \cdot
\binom{n}{r-1}$ copies of the star graph $K_{1,r-1}$ whose centre
vertex lies in $U$ and approximately $\{\delta n p^{r-1} / r\} \cdot
\binom{n}{r-1}$ copies of $K_{1,r-1}$ whose centre vertex is $u$. 
The total number of planted copies of $K_{1,r-1}$ is thus
approximately $\delta p^{r-1}\cdot \binom{n}{r} = \delta
\Ex[X] \cdot p^{-\binom{r-1}{2}}$.
Since each of the
planted copies of $K_{1,r-1}$ lies in a copy of $K_r$ that now
appears in $G_{n,p}$ with probability $p^{\binom{r-1}{2}}$, one
expects to see approximately $\delta \Ex[X]$ such extra copies of
$K_r$. We remark that if $np^{r-1}$ is large,
then the contribution of the single vertex $u$ becomes negligible and
the hub construction can be described more concisely as connecting
some $\delta n p^{r-1}/r$ vertices to all
the others, using $\delta n^2 p^{r-1}/r$ edges.

We prove that, for a vast majority of values of $p$ in the range of interest, the logarithmic upper tail probability of $X$ corresponds to one of the two strategies described above. In particular, we show that
\begin{enumerate}[label={(\roman*)}]
\item
  If $np^{r-1} \to 0$, then the logarithmic upper tail probability is
  asymptotically equal to the `cost' of planting the smallest
  clique that has $\delta \Ex[X]$ copies of $K_r$.
\item
  If $np^{r-1} \to \infty$, then the logarithmic upper tail
  probability is asymptotically equal to the `cost' of planting
  either a clique as above (when $\delta^{2/r} \le \delta/r$) or the
  smallest hub that has $\delta \Ex[X] \cdot p^{-\binom{r-1}{2}}$
  copies of $K_{1,r-1}$ (when $\delta^{2/r} \ge \delta/r$).
\end{enumerate}
Note that in the regime $np^{r-1} \to \infty$, we may approximate the number of edges
planted in the hub construction by $\delta n^2p^{r-1}/r$. However, when
$np^{r-1} \to c$ for some constant $c \in (0, \infty)$, this
approximation is no longer valid and we are forced to account for the
lack of smoothness that stems from the integral and fractional parts of
$\delta c / r$. As a result, we find that, for certain values of the
parameters $\delta$ and $c$, the logarithmic upper tail probability
corresponds to a mixture of the first and second strategies: it is equal
to the cost of planting a graph comprising both a hub and a clique, each
contributing a nonnegligible proportion of the (expected) extra $\delta
\Ex[X]$ copies of $K_r$.

Finally, suppose that one conditions $G_{n,p}$ on the upper tail event $\{X \ge (1+\delta) \Ex[X]\}$. We prove that, with probability close to one, the conditioned random graph contains a subgraph that very closely resembles the graph described by the optimal strategy (for the particular values of $n$, $p$, and $\delta$). For example, in cases where the logarithmic upper tail probability corresponds to planting a clique, we show that $G_{n,p}$ conditioned on the event $\{X \ge (1+\delta)\Ex[X]\}$ contains a set of $\delta^{1/r}np^{(r-1)/2}$ vertices that induces an `almost-clique', that is, a subgraph of density $1-o(1)$.

We now turn to the details.
As in the introduction, we define
continuous functions $\psi_r\colon (0,\infty)^2\times
[0,1]\to (0,\infty)$ and
$\varphi_r\colon (0,\infty)^2
\to (0,\infty)$ by
\[ \psi_r(\delta,c,x) = \frac{\big(\delta(1-x)\big)^{2/r}}2 +
\frac{\floor{x\delta  c/r} + \{x\delta  c/r\}^\frac1{r-1}}{c} \quad\text{and}
\quad \varphi_r(\delta,c) = \min_{x\in [0,1]} \psi_r(\delta,c,x).
\]
Note that $\psi_r(\delta,np^{r-1},x)\cdot n^2p^{r-1}$ is approximately the
number of edges in the disjoint union of a clique with $\delta (1-x)
np^{(r-1)/2}$ vertices and a hub of order $\floor{\delta x np^{r-1}/r}+1$. Recalling
the discussion above, $\varphi_r(\delta,np^{r-1})\cdot n^2p^{r-1}$
then represents the smallest number of edges among all combinations of
clique and hub that yield an expected $\delta \e[X]$ copies of $K_{r-1}$.

In order to handle the three cases $np^{r-1}\to 0$, $np^{r-1}\to c\in
(0,\infty)$, and $np^{r-1}\to \infty$ in a unified manner, it will be
convenient to extend $\varphi_r$ to a continuous function $\varphi_r
\colon (0,\infty)\times [0,\infty] \to (0,\infty)$. This
extension may be defined by noting that
\[
  \lim_{c \to 0} \varphi_r(\delta, c) = \delta^{2/r}/2 \qquad \text{and}
  \qquad  \lim_{c \to \infty} \varphi_r(\delta, c) = \min\{\delta^{2/r}/2,
  \delta/r\}
\]
uniformly as functions of $\delta$.
For every $\delta>0$ and $c\in [0,\infty]$, we then define the set
\[ \barX(\delta,c) = \{ x \in [0,1] :
\varphi_r(\delta,c) = \lim_{c'\to c} \psi_r(\delta,c',x) \} \]
of (asymptotic) minimisers to $x\mapsto \psi_r(\delta,c,x)$. One can check
that this set is nonempty for any $\delta$ and $c$, though it might contain
more than one element. The following lemma describes the set of possible
minimisers in more
detail.

\begin{lemma}\label{lem:grdcx}
  Let $r\geq 3$ be an integer and let $\delta$ and $c$ be positive real numbers. 
  Let $x^* = r\floor{\delta c/r}/(\delta c)$.
  Then
  \[ \barX(\delta,0) =
  \big\{0\big\},
  \quad
   \barX(\delta,c) \subseteq
  \big\{0,x^*,1\big\},
  \quad\text{and}\quad
  \barX(\delta,\infty) \subseteq
  \big\{0,1\big\}.\]
\end{lemma}

\begin{proof}
  The statement for $\barX(\delta,0)$ holds because $\lim_{c\to 0}
  \psi_r(\delta,c,x)=\infty$ whenever $x>0$.
  The statement for $\barX(\delta,\infty)$ follows
  because
  \[
    \lim_{c\to \infty}
    \psi_r(\delta,c,x) = (\delta(1-x))^{2/r}/2 + x\delta/r
  \]
  and the right-hand side is strictly concave in $x \in [0,1]$. In particular, since
  \[
    \varphi_r(\delta,\infty) = \min{\{\delta^{2/r},\delta/r\}} = \min_{x \in [0,1]}(\delta(1-x))^{2/r}/2 + x\delta/r,
  \]
  the equality $\varphi_r(\delta, \infty) = \lim_{c\to \infty} \psi_r(\delta,c,x)$ implies that $x$ is one of the endpoints of the interval~$[0,1]$.
  Assume now that $c\in (0,\infty)$.
  Treating $r$, $\delta$, and $c$ as fixed, consider the function $f
  \colon [0,1] \to (0, \infty)$ defined by $f(x) =
  \psi_r(\delta,c,x)$. Since $\psi_r$ is continuous,
  we have
  $x\in \barX(\delta,c)$ if and only if $x$ is a minimiser of $f(x)$
  in $[0,1]$.
  We cover the domain of $f$ with essentially
  disjoint intervals as follows:
  \[
    [0,1] = \frac{r}{\delta c} [0,1] \cup \frac{r}{\delta c}
    [1,2] \cup \dotsb \cup
    {\frac{r}{\delta c} [\floor{\delta c/r}-1, \floor{\delta c/r}]} \cup
    \frac{r}{\delta c}
    [\floor{\delta c/r},\delta c/r].
  \]
  Observe that $f$ is strictly concave on each of these intervals and thus
  $f$ can achieve its minimum value only at an endpoint of one of the intervals,
  that is, either at some $x$ for which $\delta xc/r \in \ZZ$ or at $x=1$.
  Define the function $h \colon [0,1] \to (0, \infty)$ by
  \[
    h(x) = \frac{\big(\delta(1-x)\big)^{2/r}}2+ \frac{x\delta}{r}
  \]
  and observe that $h$ is strictly concave and that $h(x) = f(x)$ whenever $\delta
  xc/r\in \ZZ$. It follows that the minimum of $f(x)$ is achieved either
  at $x=1$ or at the smallest or largest $x$ satisfying $\delta x c/r \in \ZZ$.
  Since the latter two points are $x=0$ and $x=r\floor{\delta c/r}/(\delta
  c)=x^*$, respectively, we may conclude that the minimiser is in the set
  $\{0,x^*,1\}$, as desired.
\end{proof}

Let us now state the two main results of this section. The following is a
straightforward reformulation of Theorem~\ref{thm:krrate}.

\begin{theorem}\label{thm:krrate2}
  Let $r \ge 3$ be an integer and let $X = X_{n,p}^{K_r}$ denote the number of
  $r$-vertex cliques in the random graph $G_{n,p}$.
  Then, for every fixed positive constant
  $\delta$ and all $p = p(n)$ such that $n^{-1} (\log
  n)^{\frac1{r-2}} \ll p^{\frac{r-1}{2}} \ll 1$ and $\lim_{n\to \infty} np^{r-1} = c \in [0,\infty]$,
  \[
    \lim_{n \to \infty} \frac{-\log \Pr\big(X \ge
    (1+\delta)\Ex[X]\big)}{n^2p^{r-1} \log(1/p)} =
      \varphi_r(\delta, c) .
  \]
\end{theorem}

Recall the following three events:
  \begin{enumerate}[{label=(\roman*)}]
  \item
    Let $\UT(\delta)$ be the upper tail event $\{X \ge  (1+\delta)\e[X]\}$.
  \item
    Let $\Clique_\eps(x)$ be the event that $G_{n,p}$ contains a set $U\subseteq \br n$ of 
    size at
    least $(1-\eps)x^{1/r}np^{(r-1)/2}$ such that
    $G_{n,p}[U]$ has minimum degree
    at least $(1-\eps)|U|$.
  \item
    Let $\Hub_\eps(x)$ be the event that $G_{n,p}$ contains a set $U\subseteq \br n$
    such that at least $\floor{(1 - \eps)|U|}$ vertices in $U$ have degree at least $(1-
    \eps) n$ and 
    \[ e(U,\br n\setminus U) \geq (1-\eps)n
      \big( \floor{xnp^{r-1}/r}
      + \{xnp^{r-1}/r\}^{\frac{1}{r-1}}\big).\]
  \end{enumerate}

Together with Lemma~\ref{lem:grdcx}, the following theorem directly implies
Theorem~\ref{thm:krstructure}.

\begin{theorem}\label{thm:krstructure2}
  Let $r \ge 3$ be an integer and let $X = X_{n,p}^{K_r}$ denote the number of
  $r$-vertex cliques in the random graph $G_{n,p}$. Then, for every fixed positive
  constant
  $\delta$ and all $p = p(n)$ such that $n^{-1} (\log
  n)^{\frac1{r-2}} \ll p^{\frac{r-1}{2}} \ll 1$ and $\lim_{n\to \infty} np^{r-1} = c \in [0,\infty]$,
  \[
    \lim_{n\to\infty} \Pr\Big(\bigcup_{x\in \bar X(\delta,c)}
    \Clique_\eps(\delta(1-x)) \cap \Hub_\eps(\delta x)\mid \UT(\delta)\Big)= 1.
  \]
\end{theorem}

In order to prove Theorems~\ref{thm:krrate2} and~\ref{thm:krstructure2}, we will first relate $-\log
\Pr\big(X \ge (1+\delta)\Ex[X]\big)$ to the solutions of the optimisation
problem
\[
  \Phi_X(\delta) = \min\big\{e_G \log(1/p) : G \subseteq K_n \text{ and } \Ex_G[X] \ge (1+\delta)\Ex[X]\big\},
\]
where $\Ex_G[X] = \Ex[X \mid G \subseteq G_{n,p}]$.
For every $\eps>0$, we  define the event
\[ \begin{split}\Gmin(\eps) = \{ G_{n,p}\text{ contains a subgraph $G$ such that }
  & e_G \log(1/p) \leq
  (1+\eps)\Phi_X(\delta+\eps) \text{ and}\\
  &\e_G[X]\geq (1+\delta-\eps)\e[X]\}.
\end{split}\]
Using
Theorem~\ref{thm:packaged}, we shall prove the following result.

\begin{proposition}
  \label{prop:kr-ldp}
  For every integer $r \ge 3$ and all positive reals $\eps$ and $\delta$, there exists a positive constant $C$ such that the following holds. Suppose that an integer $n$ and $p \in (0,1)$ satisfy $Cn^{-1} (\log n)^{1/(r-2)} \le p^{(r-1)/2} \le 1/C$. Then $X = X_{n,p}^{K_r}$ satisfies
  \[
    (1-\eps) \Phi_X(\delta - \eps) \le - \log \Pr\big(X \ge (1+\delta)\Ex[X]\big) \le (1+\eps)\Phi_X(\delta+\eps).
  \]
  Furthermore,
  \[ \Pr\big(\Gmin(\eps) \mid X\geq (1+\delta)\e[X]\big) \geq 1-\eps. \]
\end{proposition}

In order to complete the proof of Theorem~\ref{thm:krrate2}, we must also evaluate the asymptotic value of the function $\Phi_X$. This is the content of our second proposition.

\begin{proposition}
  \label{prop:kr-Phi}
  Let $r \ge 3$ be an integer and let $X = X_{n,p}^{K_r}$. Then, for every fixed
  positive real $\delta$ and all $p = p(n)$ such that $n^{-1} \ll p^{(r-1)/2}
  \ll 1$ and $\lim_{n\to\infty} np^{r-1} = c \in [0,\infty]$,
  \[
    \lim_{n \to \infty} \frac{\Phi_X(\delta)}{n^2p^{r-1} \log(1/p)} =
      \varphi_r(\delta, c).
  \]
\end{proposition}
We recall that, when $c \in \{0, \infty\}$, this result was already established in~\cite{lubetzky2017variational}. The proof we provide is not unsimilar to that work, although it is slightly easier, as $\Phi_X$ is a discrete restriction of the variational problem in their work (see the discussion following Proposition~\ref{prop:kap-rate}, which discusses the same issue in the case of arithmetic progressions). However, the extension to the case $c \in (0,\infty)$ is delicate, and requires the use of the more precise result of Lemma~\ref{lem:stars}.

We would like to emphasise that, while the proofs of Propositions~\ref{prop:kr-ldp} and~\ref{prop:kr-Phi} have certain similarities, they are conceptually very different. The analysis of the variational problem in the proof of Proposition~\ref{prop:kr-Phi} relies on bounding the maximal number of embeddings of $K_r$ and $K_{1,r-1}$ in a graph with given numbers of vertices and edges, using the results of Section~\ref{sec:GlobalEmbedding}. In contrast, the proof of Proposition~\ref{prop:kr-ldp} exploits the bounds on the number of embeddings of $K_r$ and $K_{1,r-1}$ (as well as several other key graphs) containing a given edge of the core, proved in Section~\ref{section:LocalEmbedding}. The defining properties of a core allow us to translate these upper bounds on the number of local embeddings into lower bounds on the degrees of the endpoints of edges of a core. These degree restrictions are sufficient to prove that the family of cores is entropically stable.

Finally, to prove Theorem~\ref{thm:krstructure2}, we characterise the
near-minimisers of the optimisation problem for $\Phi_X(\delta)$. This is what we do in the final
proposition of this section.

\begin{proposition}
  \label{prop:kr-gmin}
  Let $r \ge 3$ be an integer and let $X = X_{n,p}^{K_r}$. 
  For all fixed
  $\eps,\delta>0$ and $c\in [0,\infty]$, there
  exists some positive constant $\eta$ such that the following holds.
  Assume $p = p(n)$ is such that $n^{-1} \ll p^{(r-1)/2}
  \ll 1$ and $\lim_{n\to\infty}
   np^{r-1} = c$. Then
  \[ \Gmin(\eta) \subseteq \bigcup_{x\in \bar X(\delta,c)} \Clique_\eps(\delta(1-x))
  \cap \Hub_\eps(\delta x)  \]
  whenever $n$ is sufficiently large.
\end{proposition}

The propositions above readily imply Theorems~\ref{thm:krrate2} and
\ref{thm:krstructure2}.

\subsection{{Proof of Proposition~\ref{prop:kr-ldp}}}

We may assume without loss of generality that $\eps$ is sufficiently small, say
$\eps < \min{\{1/2,\delta/2\}}$. Note also that the case $n < r$ is trivial;
indeed, in that case $X$ is identically zero and thus $- \log \Pr\big(X\geq
(1+\delta)\e[X]\big) = 0 = \Phi_X(\delta)$ for every $\delta\in \RR$,
and $\Gmin(\eps)$ holds vacuously. We may
therefore assume that $n \geq r \ge 3$, which, in turn, implies that $n \ge C$.

Set $N = \binom{n}{2}$ and let $Y = (Y_1, \dotsc, Y_N)$ be the sequence of
indicator random variables of the events that $e \in E(G_{n,p})$, where $e$
ranges over $\binom{\br{n}}{2}$ in some arbitrary order. Observe that $Y$ is a
vector of independent $\Ber(p)$ random variables and that $X$ is a nonzero
polynomial with nonnegative coefficients and degree at most
$\binom{r}{2}$ in the coordinates of $Y$. Let $K = K(r, \eps, \delta)$ be the
constant given by Theorem~\ref{thm:packaged}. We shall show that $X$ satisfies
the various assumptions of the theorem; the theorem then implies
both assertions of the proposition.

First, our assumption on $p$ implies that $p \le 1-\eps$, provided that $C$ is sufficiently large.

Recall that $N(J,G)$ denotes the number of copies of $J$ in $G$.
Note that for all $J \subseteq K_r$ without isolated vertices and all $G \subseteq K_n$, we can trivially bound $N(J,G)$ from above by $e_G^{e_J}$. It follows that
\[
  \begin{split}
    \Ex_G[X] - \Ex[X] & \le \sum_{\emptyset \neq J \subseteq K_r} N(J,G) \cdot n^{r-v_J}p^{\binom{r}{2}-e_J} \le 2^{\binom{r}{2}} \cdot \max_{\emptyset \neq J \subseteq K_r} e_G^{e_J} \cdot n^{r-v_J} p^{\binom{r}{2} - \binom{v_J}{2}} \\
    & \le (2e_G)^{\binom{r}{2}} \cdot \frac{n^r p^{\binom{r}{2}}}{\min_{2 \le k \le r} n^k p^{\binom{k}{2}}},
  \end{split}
\]
where the sum ranges over all nonempty subgraphs $J \subseteq K_r$ without isolated vertices.
Since $\Ex[X] = \Theta\big(n^rp^{\binom{r}{2}}\big)$ and our assumption on $p$ implies that $n^kp^{\binom{k}{2}} \ge C^{2/(r-1)}$ for each $k \in \{2, \dotsc, r\}$, the right-hand side above is at most $(\delta/2)\Ex[X]$ unless $e_G \ge K$, provided that $C$ is sufficiently large. Therefore, $\Phi_X(\delta-\eps) \ge \Phi_X(\delta/2) \ge K\log(1/p)$. Furthermore, since a clique with $\lceil (1+2\delta)^{1/r} n p^{(r-1)/2} \rceil$ vertices contains at least $(1+\delta+\eps) \Ex[X]$ copies of $K_r$ and has fewer than $K' n^2 p^{r-1}$ edges, for some constant $K' = K'(\delta)$, we deduce that $\Phi_X(\delta+\eps) \le K' n^2p^{r-1} \log(1/p)$.

Recall that a graph $\Gcore \subseteq K_n$ is a \emph{core} if it satisfies the following three conditions:
\begin{enumerate}[label=(C\arabic*)]
\item\label{item:krcore-bias}
  $\Ex_{\Gcore}[X] \geq (1+\delta-\eps)\Ex[X]$,
\item\label{item:krcore-size}
  $e_{\Gcore} \le K \cdot \Phi_X(\delta+\eps)$, and
\item\label{item:krcore-mindeg}
  $\min_{e \in E(\Gcore)}\left(\Ex_{\Gcore}[X]-\Ex_{\Gcore \setminus e}[X]\right)
  \ge \Ex[X]/\big(K \cdot \Phi_X(\delta+\eps)\big)$.
\end{enumerate}
Our goal is to show that, for every integer $m$, there are at most $(1/p)^{\eps m/2}$ cores with $m$ edges. The key observation is that, for any core $\Gcore$ and any edge $uv$ of $\Gcore$, either $u$ and $v$ have many common neighbours or the sum of the degrees of $u$ and $v$ is large. More precisely, letting $\deg_{\Gcore}(u,v)$ denote the number of common neighbours of $u$ and $v$ in $\Gcore$, we shall establish the following statement.

\begin{claim}
  \label{claim:core-degree}
  There exists a positive constant $\eta = \eta(\delta, r, K)$ such that,
  for every core $\Gcore$ and each edge $uv \in E(\Gcore)$, either
  \[
    \deg_{\Gcore}(u,v) \ge \frac{\eta np^{(r-1)/2}}{\big(\log(1/p)\big)^{1/(r-2)}} \qquad \text{or} \qquad \deg_{\Gcore}u + \deg_{\Gcore} v \ge \frac{\eta n}{\log(1/p)}.
  \]
\end{claim}

The above claim readily implies the desired bound on the number of cores with $m$ edges.
Note first that this number is zero whenever $m > KK'n^2p^{r-1}\log(1/p)$, see~\ref{item:krcore-size}, so we may assume that $m\leq KK'n^2p^{r-1}\log(1/p)$.
Given a core $\Gcore$, we denote by $A_\Gcore$ the set of vertices of $\Gcore$ with degree at least $\eta np^{(r-1)/2}/\big(\log(1/p)\big)^{1/(r-2)}$ and by
$B_\Gcore\subseteq A_\Gcore$ the set of vertices of $\Gcore$ with degree at least $\eta n/\\big(2\log(1/p)\big)$. Since $G^*$ has $m$ edges,
\[
  |A_{G^*}|\leq a := \frac{2m\big(\log(1/p)\big)^{1/(r-2)}}{\eta np^{(r-1)/2}}
  \quad\text{and}\quad
  |B_{G^*}| \leq b := \frac{4m\log(1/p)}{\eta n}
\]
Claim~\ref{claim:core-degree} states that every edge of $\Gcore$ is either fully contained in $A_{\Gcore}$ or has at least one endpoint in $B_{\Gcore}$. In particular, for fixed sets $B \subseteq A \subseteq \br{n}$ with $|A| = a$ and $|B| = b$, there are at most $\binom{a^2/2 + bn}{m}$ cores $\Gcore$ with $m$ edges that satisfy $A_{\Gcore} \subseteq A$ and $B_{\Gcore} \subseteq B$. We can thus (generously) upper bound the number of cores with~$m$ edges by
\[
  \binom{n}{a}\binom{n}{b}\binom{a^2/2 + bn}{m}.
\]
Recalling the inequality $\binom{x}{y} \le (ex/y)^y$, valid for all nonnegative integers $x$ and $y$, we may conclude that the number of cores with $m$ edges is at most
\[
  n^{\frac{6m(\log(1/p))^{1/(r-2)}}{\eta np^{(r-1)/2}}} \cdot \left(
    \frac{2em\big(\log(1/p)\big)^{2/(r-2)}}{\eta^2 n^2p^{r-1}} +\frac{4e\log(1/p)}{\eta} \right)^{m}.
\]
Since $p^{(r-1)/2} \ge Cn^{-1}(\log n)^{1/(r-2)}$, the first factor is at most $e^{\eps m\log(1/p)/4}$. Since we have assumed that $m \leq KK'n^2p^{r-1}\log(1/p)$ and $1/p \ge C^{2/(r-1)}$, the second factor is at most $e^{O(m\log\log(1/p))} \le e^{\eps m\log(1/p)/4}$. This shows that the number of cores with $m$ edges is indeed at most $(1/p)^{\eps m/2}$, as claimed.

\begin{proof}[{Proof of Claim~\ref{claim:core-degree}}]
  For any nonempty $J \subseteq K_r$, we shall let $N(J, \Gcore; uv)$ denote the
  number of copies of $J$ in $\Gcore$ that contain the edge $uv$. For the sake of
  brevity, we set $\mmax = KK' \cdot n^2p^{r-1} \log(1/p)$. Observe that
  \[
    \Ex_{\Gcore}[X]-\Ex_{\Gcore \setminus uv}[X] \le \sum_{\emptyset \neq J \subseteq K_r} N(J, \Gcore; uv) \cdot n^{r-v_J} p^{\binom{r}{2} - e_J},
  \]
  where $J$ ranges over all nonempty subgraphs of $K_r$ that have no isolated
  vertices. Since $\Ex[X] = \binom{n}{r}p^{\binom{r}{2}}$
  and $\Phi_X(\delta + \eps) \le K' n^2p^{r-1}\log(1/p)=\mmax/K$, it follows
  from~\ref{item:krcore-mindeg} in the definition of the core
  that
  \[
    \sum_{\emptyset \neq J \subseteq K_r} \frac{N(J,\Gcore;
    uv)}{n^{v_J}p^{e_J}}
    \geq \frac{\Ex[X]}{K\cdot \Phi_X(\delta+\eps)}
    \ge \frac{\gamma}{\mmax},
  \]
  where $\gamma = \gamma(r)$ is a constant that depends only on $r$.
  
  For every edge $ab\in E(J)$, let $\Emb(J,\Gcore;ab,uv)$ denote the
  set of embeddings of $J$ into $\Gcore$ that map $ab$ to $uv$.
  Then the above inequality implies that there is a nonempty $J
  \subseteq K_r$ with no isolated vertices, an edge $ab \in E(J)$, and a
  constant $\gamma' = \gamma'(r)$ such that \begin{equation}
    \label{eq:core-J-extremal}
    \frac{|\Emb(J, \Gcore; ab, uv)|}{n^{v_J}p^{e_J}} \ge \frac{\gamma'}{\mmax}.
  \end{equation}

  \newcommand{\doublestar}[2]{
    \draw (0,0) node (a) {} (1,0) node (b) {};
    \pgfmathsetmacro{\maxi}{#1}
    \pgfmathsetmacro{\maxj}{#2}
    \pgfmathparse{ifthenelse(\maxi>0,1,0)} 
    \ifnum\pgfmathresult=1
    \foreach \i in {1,...,\maxi} {
      \draw[xscale=0.3] (-\maxi/2+\i-1/2,1) node {} -- (a);
    }
    \fi
    \pgfmathparse{ifthenelse(\maxj>0,1,0)} 
    \ifnum\pgfmathresult=1
    \foreach \j in {1,...,\maxj} {
      \draw[xshift=1cm,xscale=0.3] (-\maxj/2+\j-1/2,1) node {} -- (b);
    }
    \fi
    \draw (a) -- (b);
    \node[draw=none,fill=none,rectangle] at (0.5,-0.5) {$S_{#1,#2}$};
  }
  \begin{figure}
    \begin{tikzpicture}
    \tikzset{every node/.style={draw,circle,fill=black,inner sep=0pt,minimum
    size=3pt}}
      \doublestar{0}{4}
      \begin{scope}[xshift=3cm]\doublestar{1}{3}\end{scope}
      \begin{scope}[xshift=6cm]\doublestar{2}{2}\end{scope}
    \end{tikzpicture}
    \caption{The double stars with six vertices.}\label{fig:doublestars}
  \end{figure}
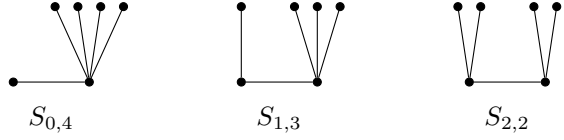

  For nonnegative integers $i$ and $j$, let $S_{i,j}$ denote the graph
  obtained from a copy of $K_{1,i}$ and a copy of $K_{1,j}$ by joining their
  centres (vertices of degrees $i$ and $j$, respectively) by an edge; see
  Figure~\ref{fig:doublestars} for an illustration.
  As the graphs $K_{1,i}$
  are often called \emph{stars}, we shall refer to the $S_{i,j}$ as
  \emph{double stars}. Moreover, we shall call an edge of $S_{i,j}$ whose
  endpoints have degrees $i+1$ and $j+1$ a \emph{centre edge}. Note that if
  $i, j > 0$, then $S_{i,j}$ has only one centre edge; otherwise, $S_{i,j}$
  is just a star graph and each of its edges is a centre edge.

  We shall now show that, unless $\deg_\Gcore(u,v) \ge np^{(r-1)/2}$, any graph $J$ that satisfies~\eqref{eq:core-J-extremal} for some $ab \in E(J)$ must be either $K_r$ or a double star with $r$ vertices. We first show how this fact implies the assertion of the claim. Assume first that 
 \[
   |\Emb(K_r, \Gcore; ab, uv)| \ge \frac{\gamma' n^r p^{\binom{r}{2}}}{\mmax} = \frac{\gamma'}{KK'} \cdot \frac{n^{r-2} p^{\binom{r-1}{2}}}{\log(1/p)}.
 \]
 Since $|\Emb(K_r, \Gcore; ab, uv)|
 \leq 2\deg_{\Gcore}(u,v)^{r-2}$, we conclude that
 \[
   \deg_{\Gcore}(u,v) \ge \left(\frac{\gamma'}{2KK'}\right)^{1/(r-2)} \cdot \frac{n p^{(r-1)/2}}{\big(\log(1/p)\big)^{1/(r-2)}},
 \]
 as claimed. Next, assume that, for some $i$ and $j$ with $i+j=r-2$,
 \[
   |\Emb(S_{i,j}, \Gcore; ab, uv)| \ge \frac{\gamma' n^r p^{i+j+1}}{\mmax} = \frac{\gamma'}{KK'} \cdot \frac{n^{r-2}}{\log(1/p)}.
 \]
 If $ab$ is a centre edge of $S_{i,j}$, then $|\Emb(S_{i,j}, \Gcore; ab, uv)| \le
 (\deg_{\Gcore}u + \deg_{\Gcore}v)^{i+j}
 \leq
 (\deg_{\Gcore}u + \deg_{\Gcore}v)\cdot (2n)^{r-3}
 $. Otherwise,
 we have $i, j > 0$ and so
 $|\Emb(S_{i,j}, \Gcore; ab, uv)| \le (\deg_{\Gcore}u +
 \deg_{\Gcore}v)^{\min\{i,j\}} \cdot n^{\max\{i,j\}} \le (\deg_{\Gcore}u +
 \deg_{\Gcore}v) \cdot
 (2n)^{r-3}$ as well. Thus, in both cases, \[
   \deg_{\Gcore}u + \deg_{\Gcore}v \ge \frac{\gamma'}{2^{r-3} KK'} \cdot \frac{n}{\log(1/p)},
 \]
 which completes the proof of the claim.

 It remains to prove our assertion. We first consider the special case $r =
 3$. The only nonempty subgraph of $K_3$ with no isolated vertices that is
 not isomorphic to a double star with three vertices is $K_2$. However,
 as $p\leq C^{-2/(r-1)} = C^{-1}$, we have
 \[
   \frac{|\Emb(K_2, \Gcore; ab, uv)|}{n^{2} p} = \frac{2}{n^2p} \le
   \frac{2\log C}{C n^2p^2 \log(1/p)} < \frac{\gamma'}{\mmax}
 \]
 whenever $C$ is sufficiently large, which contradicts~\eqref{eq:core-J-extremal}. We henceforth assume that $r \ge 4$.
 
 By way of contradiction, suppose that $J$ is neither $K_r$ nor a double star
 with $r$ vertices and that $\deg_{\Gcore}(u,v) < np^{(r-1)/2}$. Let $ab$ be
 an arbitrary edge of $J$, let $J_{ab}$ be the subgraph of $J$ induced by
 $V(J) \setminus \{a,b\}$, and let $\alpha_{ab}^*$ be the fractional
 independence number of $J_{ab}$. By Lemma~\ref{lemma:fractional-duality},
 there is a partition of
 $V(J_{ab})$ into $V_1$ and $V_2$ such that
 \begin{enumerate}[label=(P\arabic*)]
 \item
   \label{item:Jab-partition-alpha}
   $|V_1|/2+|V_2| = \alpha_{ab}^*$,
 \item
   \label{item:Jab-partition-V1}
   $V_1$ can be covered by a collection of vertex-disjoint edges and cycles of $J_{ab}$.
 \end{enumerate}
 Among all partitions satisfying~\ref{item:Jab-partition-alpha} and~\ref{item:Jab-partition-V1}, choose one that maximises the number of common neighbours of $a$ and $b$ in $V_2$, that is, the cardinality of the set
 \[
   X = \{c \in V_2 : ac, bc \in E(J)\}.
 \]
 Finally, let $v_1 = |V_1|$, $v_2 = |V_2|$, and $x = |X|$.

 We now observe that for every partition satisfying~\ref{item:Jab-partition-alpha} and~\ref{item:Jab-partition-V1}, we have
 \begin{equation}
   \label{eq:NJGcore-upper-bound}
  |\Emb(J, \Gcore; ab, uv)| \le 2\cdot \deg_{\Gcore}(u,v)^x \cdot
   (2e_{\Gcore})^{v_1/2} \cdot \min\{2e_{\Gcore}, n\}^{v_2-x}.
\end{equation}
To see this, note first that there are two embeddings of $ab$ onto $uv$. Since
$X$ lies in the common neighbourhood of $a$ and $b$, each such embedding can be
extended to an embedding of $J[\{a,b\}\cup X]$ in at most $\deg_\Gcore(u,v)^x$ ways.
Next, since $V_1$ can be covered by cycles and edges of $J$ (by
property~\ref{item:Jab-partition-V1}),
Lemma~\ref{lemma:NCell} implies that every embedding of 
$J[\{a,b\}\cup X]$ can be extended in at most
$(2e_{\Gcore})^{v_1/2}$ ways to an embedding of $J[\{a,b\}\cup X\cup V_1]$.
Finally, since
$J$ contains no isolated vertices, any embedding of
$J[\{a,b\}\cup X\cup V_1]$
can be extended in at most
$\min\{2e_{\Gcore}, n\}^{v_2-x}$ ways to an embedding of $J$.

Combining~\eqref{eq:core-J-extremal}, \eqref{eq:NJGcore-upper-bound}, the inequality $e_{\Gcore} \le \mmax$, which follows from~\ref{item:krcore-size}, and the assumption $\deg_{\Gcore}(u,v) < np^{(r-1)/2}$, we deduce that
\[
  \begin{split}
   \gamma' n^{v_J} p^{e_J} & \le \left(np^{(r-1)/2}\right)^x \cdot
    (2\mmax)^{v_1/2+1} \cdot \min\{2\mmax, n\}^{v_2-x} \\
    & = \frac{1}{\big(KK' \log(1/p)\big)^{x/2}} \cdot (2\mmax)^{v_1/2+x/2+1}
    \cdot \min\{2\mmax, n\}^{v_2-x}.
  \end{split}
\]
On the other hand,
\[
  n^{v_J}p^{e_J} = \left(n^2p^{r-1}\right)^{e_J/(r-1)} \cdot n^{v_J - 2e_J/(r-1)} = \left(\frac{\mmax}{KK'\log(1/p)}\right)^{e_J/(r-1)} \cdot n^{v_J - 2e_J/(r-1)}.
\]
Therefore, for some constant $K'' = K''(K, K', r)$, we must have
\begin{equation}
  \label{eq:mmax-lower}
  (\mmax)^{v_1/2+x/2-e_J/(r-1)+1} \cdot n^{2e_J/(r-1) - v_J} \cdot \min\{\mmax,
  n\}^{v_2-x} \ge \big( K'' \log(1/p)\big)^{-K''}.
\end{equation}
In order to reach the desired contradiction, it suffices to prove that
there is some positive $\sigma=\sigma(r)$ such that the left-hand side of
\eqref{eq:mmax-lower} is bounded from above by $\max{\{n^{-1},KK'\cdot
p^{r-1}\log(1/p)\}}^\sigma$.
We first observe that such an upper bound is implied by the following inequality:
\begin{equation}
  \label{eq:eJ-bound-sparse}
  \max\big\{v_1/2+v_2 -x/2 -e_J/(r-1) + 1, \, 0\big\} < v_J - 2e_J/(r-1).
\end{equation}
Indeed, if $\mmax\leq n$, then the left-hand side of \eqref{eq:mmax-lower} is upper bounded
by
\begin{multline*}
  (\mmax)^{v_1/2+v_2-x/2-e_J/(r-1)+1} \cdot n^{2e_J/(r-1) - v_J} \\
  \leq \max{\{n^{2e_J/(r-1)-v_J}, n^{v_1/2+v_2-x/2 +1 - (v_J-e_J/(r-1))}\}}.
\end{multline*}
On the other hand, if $n< \mmax = KK'\cdot n^2p^{r-1}\log(1/p)$, then the
left-hand side of \eqref{eq:mmax-lower} becomes
\begin{multline*}
  \left(KK'\cdot n^2p^{r-1}\log(1/p)\right)^{v_1/2+x/2-e_J/(r-1)+1} \cdot n^{2e_J/(r-1) - v_J + v_2-x}\\
  = \left(KK'\cdot p^{r-1}\log(1/p)\right)^{v_J - e_J/(r-1) - (v_1/2 +v_2 -x/2+1)},
\end{multline*}
using $v_1+v_2= v_{J_{ab}} = v_J-2$. Note that \eqref{eq:eJ-bound-sparse} guarantees
that, in both cases, the left-hand side of
\eqref{eq:mmax-lower} is at most
$\max{\{n^{-1},KK'\cdot
p^{r-1}\log(1/p)\}}^\sigma$, for a suitable positive $\sigma=\sigma(r)$.

In order to complete the proof of Claim~\ref{claim:core-degree}, we now prove inequality~\eqref{eq:eJ-bound-sparse}. Recall that $V(J_{ab}) = V_1 \cup V_2$ is a partition that satisfies~\ref{item:Jab-partition-alpha} and~\ref{item:Jab-partition-V1} that maximises the cardinality of the set $X = \{c \in V_2 : ac, bc \in E(J)\}$. Since, by the definition of $X$, each vertex in $V_2 \setminus X$ has at most one neighbour in $\{a, b\}$,
 \begin{equation}
   \label{eq:eJ-eJab}
   e_J \le e_{J_{ab}} + 2v_1 + v_2 + x + 1 = e_{J_{ab}} + 3v_J - 2\alpha_{ab}^* + x - 5
 \end{equation}
and equality holds only if every vertex in $V_1$ is adjacent to both $a$ and $b$
and every vertex in $V_2\setminus X$ is adjacent to exactly one of $a$ and $b$. Moreover, Lemma~\ref{lemma:eJ-alphaJ-clique} and Remark~\ref{remark:eJ-alphaJ-clique} give
\begin{equation}
  \label{eq:eJab-alphaab}
  e_{J_{ab}} \le (v_{J_{ab}}-1)(v_{J_{ab}}-\alpha_{ab}^*) = (v_J-3)(v_J-\alpha_{ab}^*-2),
\end{equation}
where equality holds only if $J_{ab}$ is complete, empty, or isomorphic to $K_{1,v_J-3}$. Putting~\eqref{eq:eJ-eJab} and~\eqref{eq:eJab-alphaab} together yields the inequality
\begin{equation}
  \label{eq:eJ-final}
  e_J \le (v_J-1)(v_J-\alpha_{ab}^*-1) + x.
\end{equation}
Moreover, inequality~\eqref{eq:eJ-final} is strict unless both~\eqref{eq:eJ-eJab} and~\eqref{eq:eJab-alphaab} hold with equality. Rearranging~\eqref{eq:eJ-final} gives the inequality
\begin{equation}
  \label{eq:eJ-bound}
  \frac{e_J}{r-1} \le \frac{v_J-1}{r-1} \cdot \left(v_J-\alpha_{ab}^*-1\right) + \frac{x}{r-1} \le v_J-\alpha_{ab}^*-1+\frac{x}{2}.
\end{equation}
Since $\alpha_{ab}^* \le v_{J_{ab}} = v_J-2$ and $r \ge 4$, the second inequality in~\eqref{eq:eJ-bound} is strict unless $v_J = r$ and $x = 0$. Consequently, the left-hand and the right-hand sides of~\eqref{eq:eJ-bound} can be equal only if the following conditions are satisfied for every partition $V(J_{ab}) = V_1 \cup V_2$ with properties~\ref{item:Jab-partition-alpha} and~\ref{item:Jab-partition-V1}:
\begin{enumerate}[label=(D\arabic*)]
\item
  \label{item:D1}
  $J_{ab}$ is either $K_{r-2}$, $E_{r-2}$ (the empty graph with $r-2$ vertices), or $K_{1, r-3}$;
\item
  \label{item:D2}
  every vertex in $V_1$ is adjacent to both $a$ and $b$;
\item
  \label{item:D3}
  every vertex of $V_2$ is adjacent to exactly one of $a$ and $b$.
\end{enumerate}
We now show that our assumptions preclude \ref{item:D1}--\ref{item:D3} holding
simultaneously. Indeed, note first that, if $J_{ab} = K_{r-2}$ (which also
includes the case $J_{ab} = K_{1,r-3}$ and $r = 4$), then $\alpha_{ab}^* =
v_{J_{ab}}/2$. Since $v_1/2+v_2=\alpha_{ab}^*$ and $v_1+v_2 = v_{J_{ab}}$, this
implies $V_1 = V(J_{ab})$. Then \ref{item:D2} shows that $J = K_r$, a contradiction.
Second, if $J_{ab} = E_{r-2}$, then $\alpha^*_{ab} = r-2$. Since
$v_1/2+v_2=\alpha_{ab}^*$ and $v_1+v_2 = r-2$, this implies
$V_2 = V(J_{ab})$ and it follows from \ref{item:D3} that $J$ is a double star
whose centre edge is $ab$, another
contradiction. Finally, suppose that $J_{ab} = K_{1,r-3}$ and $r > 4$. Since
$\alpha^*_{ab} = v_1/2+v_2= r - 3$ and $v_1+v_2=r-2$, we see that
$v_1 = 2$ and $v_2 = r-4$. Property~\ref{item:Jab-partition-V1} implies that $V_1 = \{c,d\}$,
where $c$ is the vertex of degree $r-3$ in $J_{ab}$ and $d$ is one of its
neighbours. Since $d \in V_1$, it must be adjacent to both $a$ and $b$ (see Figure~\ref{fig:starproof}
for an illustration). Let $e$
be an arbitrary vertex of $V_2$; there is at least one such vertex as $v_2 =
r-4 \ge 1$. The partition of $V(J_{ab})$ into $V_1' = \{c, e\}$ and $V_2' =
(V_2 \setminus \{e\}) \cup \{d\}$ satisfies both
conditions~\ref{item:Jab-partition-alpha} and~\ref{item:Jab-partition-V1} but
the set $X' = \{c' \in V_2' : ac', vc' \in E(J_{ab})\}$ is nonempty (as it
contains the vertex $d$); this contradicts property~\ref{item:D3} for the
partition $V(J) = V_1' \cup V_2'$.

\begin{figure}
  \centering
  \begin{tikzpicture}
    \tikzset{every node/.style={draw,circle,fill=black,inner sep=0pt,minimum size=3pt}}
    \path[draw=none,fill=black!10] (1.5,0) ellipse (0.5cm and 0.8cm)
    ++(-0.75,0) node[draw=none,rectangle,fill=none] {$V_2$};
    \path[draw=none,fill=black!10] (3.5,0) ellipse (1cm and 0.5cm) ++(1.3,0) node[draw=none,rectangle,fill=none] {$V_1$};
    \node[label=left:$a$] (a) at (0,-1) {};
    \node[label=left:$b$] (b) at (0,1) {};
    \node[label=45:$c$] (c) at (3,0) {};
    \node[label=45:$d$] (d) at (4,0) {};
    \draw (a)--(b) (c)--(d);
    \draw (a) to[bend right] (c) (a) to[bend right] (d);
    \draw (b) to[bend left] (c) (b) to[bend left] (d);
    \draw (1.5,-0.5) node (x1) {}
    ++(0,0.25) node (x2) {}
    ++(0,0.25) node (x3) {}
    ++(0,0.25) node (x4) {}
    ++(0,0.25) node (x5) {};
    \foreach \i in {1,...,5} {\draw (c) -- (x\i);}
    \foreach \i in {1,...,3} {\draw (a) -- (x\i);}
    \foreach \i in {4,5} {\draw (b) -- (x\i);}
  \end{tikzpicture}
  \caption{Illustration for the case $J_{ab} = K_{1,r-3}$ and $r>4$. Any partition 
  $V(J_{ab}) = V_1'\cup V_2'$ obtained by exchanging $d$ with a vertex from $V_2$ 
  satisfies~\ref{item:Jab-partition-alpha} and~\ref{item:Jab-partition-V1} but violates~\ref{item:D3}.
  }\label{fig:starproof}
\end{figure}
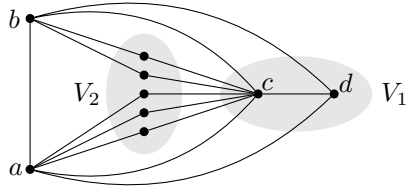

To summarise, at least one of the inequalities in~\eqref{eq:eJ-bound} is strict. Since $J \neq K_r$, not every vertex of $J$ has degree $r-1$ and thus $2e_J < (r-1) \cdot v_J$. We conclude that
\[
  \max\big\{\alpha_{ab}^* -e_J/(r-1) - x/2 +  1,\, 0\big\} < v_J - 2e_J/(r-1).
\]
Since $\alpha_{ab}^* = v_1/2 + v_2$, this is exactly~\eqref{eq:eJ-bound-sparse}.
\end{proof}

\subsection{Proof of Proposition~\ref{prop:kr-Phi}}
We begin by showing that
\begin{equation}\label{eq:kr-phi-upper}
    \limsup_{n \to \infty} \frac{\Phi_X(\delta)}{n^2p^{r-1} \log(1/p)} \le
    \varphi_r(\delta, c).
\end{equation}
For every small $\eps>0$ and sufficiently large $n$, we shall construct a
graph $G$ with vertex set $\br{n}$ and at most $(\varphi_r(\delta+\eps,
c)+\eps) \cdot n^2p^{r-1}$ edges that satisfies $\e_G[X]\geq
(1+\delta)\e[X]$. 
The existence of such a graph
and the continuity of $\varphi_r$ will imply that
\[
  \limsup_{n \to \infty} \frac{\Phi_X(\delta)}{n^2p^{r-1} \log(1/p)}
  \leq 
  \lim_{\eps \to 0} \varphi_r(\delta+\eps,c)
  +\eps
  = \varphi_r(\delta,c),
\]
as required.

Let $x \in \bar X(\delta+\eps,c)$ and note that it follows from Lemma~\ref{lem:grdcx} that
$x = 0$ if $c=0$. Set
\[
  \ell_1 = \big((\delta+\eps)(1-x)\big)^{1/r} np^{(r-1)/2}
  \qquad\text{and}\qquad
  \ell_2 = x(\delta+\eps) np^{r-1}/r.
\]
and fix an arbitrary partition  $\br n = \{u\} \cup U_1\cup U_2\cup U_3$, where $|U_1| = \floor{\ell_1}$ and $|U_2| = \floor{\ell_2}$.
Let $G$ be the union of the clique on $U_1$, the complete bipartite graph between $U_2$ and $U_3$, and an arbitrary star, centred at $u$, with $\floor{\{\ell_2\}^{1/(r-1)}|U_3|}$ edges whose non-$u$ endpoints are in $U_3$. We have
\[
  \begin{split}
    e_G & \leq \binom{\ell_1}{2} + \floor{\ell_2}\cdot |U_3| + \{\ell_2\}^{1/(r-1)}\cdot |U_3|\\
    & \le \frac{\big((\delta+\eps)(1-x)\big)^{2/r}n^2p^{r-1}}{2} + \big(\floor{\ell_2} + \{\ell_2\}^{1/(r-1)}\big)\cdot n \\
    &\le \left(\frac{\big((\delta+\eps)(1-x)\big)^{2/r}}{2} + \frac{\floor{\ell_2} + \{\ell_2\}^{1/(r-1)}}{np^{r-1}}\right) \cdot n^2p^{r-1} \\
    & =
    \psi_r(\delta+\eps,np^{r-1},x) \cdot n^2p^{r-1},
  \end{split}
\]
and our choice of $x$ ensures that
$\psi_r(\delta+\eps,np^{r-1},x) \leq
\varphi_r(\delta+\eps,c)+\eps$ for all large enough $n$.

It remains to show that $\e_G[X] \geq (1+\delta)\e[X]$. To this end, observe first that:
\begin{enumerate}[label={(\roman*)}]
  \item
    The complete graph $G[U_1]$ contains $\binom{\floor{\ell_1}}{r}$ copies of $K_r$.
  \item
    The complete bipartite graph $G[U_2,U_3]$ contains $\floor{\ell_2}\cdot \binom{|U_3|}{r-1}$ copies of $K_{1,r-1}$ whose centre vertex lies in $U_2$.
  \item
    The star $G[u ,U_3]$ contains $\binom{\floor{\{\ell_2\}^{1/(r-1)}|U_3|}}{r-1}$ copies of $K_{1,r-1}$.
\end{enumerate}
In particular,
\[
  \begin{split}
    \e_G[X] - \e[X]  & =
    \sum_{\emptyset \neq J \subseteq K_r} N(J, G) \cdot 
    \binom{n-v_J}{r-v_J}
    \cdot
    p^{\binom{r}{2}} ( p^{-e_J} - 1) \\
    & \ge (1-p) \cdot \left[\binom{\floor{\ell_1}}{r} + \left[\floor{\ell_2}\cdot \binom{|U_3|}{r-1} + \binom{\floor{\{\ell_2\}^{1/(r-1)}|U_3|}}{r-1} \right] \cdot p^{\binom{r}{2} - (r-1)} \right].
  \end{split}
\]
We now estimate the right-hand side of the above inequality. Since $np^{(r-1)/2} \to \infty$, we have
\[
  \binom{\floor{\ell_1}}{r} \ge \frac{(\ell_1 - r)^r}{r!} \ge \big((\delta+\eps)(1-x) - \eps/2\big) \cdot \frac{n^rp^{\binom{r}{2}}}{r!},
\]
provided that $n$ is sufficiently large. In the case $c = 0$, 
this is already sufficient, as $x=0$ and
\[
  \e_G[X] - \e[X]  \ge  (1-p) \cdot \binom{\floor{\ell_1}}{r} \ge (1-p) \cdot (\delta + \eps/2) \cdot \frac{n^rp^{\binom{r}{2}}}{r!} \ge \delta \Ex[X].
\]
We may therefore assume that $c \in (0, \infty]$. Since $p \to 0$, and thus $|U_3| / n \to 1$, we have
\[
  \begin{split}
    \floor{\ell_2}\cdot \binom{|U_3|}{r-1} + \binom{\floor{\{\ell_2\}^{1/(r-1)}|U_3|}}{r-1} & \ge
    \floor{\ell_2}\cdot \frac{(1-\eps^2)\cdot n^{r-1}}{(r-1)!} + \frac{\big(\{\ell_2\}^{1/(r-1)}|U_3| - r\big)^{r-1}}{(r-1)!} \\
    & \ge \frac{(1-\eps^2) \cdot \ell_2 \cdot  n^{r-1}}{(r-1)!} - \frac{\eps^2 n^{r-1}}{(r-1)!} \\
    & = \left((1-\eps^2) x (\delta+\eps)- \frac{\eps^2r}{np^{r-1}}\right) \cdot \frac{n^rp^{r-1}}{r!}.
  \end{split}
\]
Consequently,
\[
  \begin{split}
    \e_G[X] - \e[X] & \ge (1-p) \cdot \left((\delta+\eps)(1-x) - \eps/2 + (1-\eps^2) x (\delta+\eps)- \frac{\eps^2r}{np^{r-1}}\right) \cdot \frac{n^rp^{\binom{r}{2}}}{r!}\\
    & \ge (1-p) \cdot \left(\delta + \eps/2 - \eps^2(\delta+\eps) - \frac{\eps^2r}{np^{r-1}}\right) \cdot \Ex[X] \ge \delta \Ex[X],
  \end{split}
\]
where the last inequality holds for all sufficiently small $\eps$ since $np^{r-1} \to c > 0$.
This completes the proof of \eqref{eq:kr-phi-upper}.

It remains to prove the matching lower bound
\begin{equation}
  \liminf_{n \to \infty} \frac{\Phi_X(\delta)}{n^2p^{r-1} \log(1/p)} \ge
  \varphi_r(\delta,c).
\end{equation}
Fix $\eps>0$ small enough and suppose that $n$ is sufficiently large.
By the continuity of $\varphi_r$,
it is enough to show that any graph $G$ on $n$ vertices satisfying
\[\e_G[X] \geq (1+\delta)\e[X] \]
has at least $(1-\eps)\cdot \varphi_r(\delta-\eps,c)\cdot n^2p^{r-1}$
edges. By way of contradiction,
assume that $e_G < (1-\eps)\cdot \varphi_r(\delta-\eps,c)\cdot n^2p^{r-1}$.
Note that, for all large enough $n$,
\[
  (\delta-\eps/4)\cdot \frac{n^rp^{\binom{r}{2}}}{r!}
  \leq \delta \e[X]\leq  \Ex_G[X] - \e[X] \leq
  \sum_{\emptyset \neq J \subseteq K_r} N(J, G) \cdot
  n^{r-v_J}\cdot p^{\binom{r}{2}-e_J},
\]
where the sum ranges over the nonempty subgraphs $J$ of $K_r$ without
isolated vertices, so
\begin{equation}\label{eq:qwe}
  \sum_{\emptyset \neq J \subseteq K_r} 
  \frac{N(J,G)}{n^{v_J}p^{e_J}}\geq \frac{\delta-\eps/3}{r!}.
\end{equation}
Using our assumed upper bound on $e_G$,
Theorem~\ref{thm:max-copies} implies that
\[ \begin{split}
  \frac{N(J,G)}{\delta n^{v_J}p^{e_J}}
  &\leq 
  \frac{(2e_G)^{v_J-\alpha_J^*}\cdot\min{\{2e_G,n\}}^{2\alpha_J^*-v_J}}{n^{v_J}p^{e_J}}\\
  &\leq C\cdot 
  \frac{
    \min{\{
      (n^2p^{r-1})^{\alpha_J^*}, n^{v_J}p^{(r-1)(v_J-\alpha_J^*)} \}}
  }{n^{v_J}p^{e_J}
} \end{split}\]
for a suitable constant $C$.
If the minimum is achieved by
the first term, then $np^{r-1} \leq 1$, and thus
\[
  \frac{(n^2p^{r-1})^{\alpha_J^*}}{n^{v_J}p^{e_J}} = n^{\alpha_J^*-v_J+e_J/(r-1)} \cdot \big(np^{r-1}\big)^{\alpha_J^*-e_J/(r-1)} \leq  n^{\alpha^*_J - v_J + e_J/(r-1)},
\]
as $\alpha_J^* \ge v_J/2 \ge e_J / (r-1)$ for every graph $J$ with maximum degree at most $r-1$.
A straightforward algebraic manipulation in the case where the minimum is
achieved by the second term then shows that, in both cases,
\[
  \frac{N(J,G)}{n^{v_J}p^{e_J}}
  \leq C \cdot \min{\{n^{-1},p^{r-1}\}}^{v_J-\alpha^*_J - e_J/(r-1)}.
\]
If $J\subseteq K_r$ is not equal to either $K_r$ or $K_{1,r-1}$, then
Lemma~\ref{lemma:eJ-alphaJ-clique} and the remark that follows it imply that
$v_J-\alpha^*_J - e_J/(r-1)>0$, and 
then the right-hand side goes to zero as $n\to\infty$.
It thus follows from \eqref{eq:qwe} that
\[ \frac{N(K_r,G)}{n^rp^{\binom{r}{2}}}
+ 
\frac{N(K_{1,r-1},G)}{n^r p^{r-1}}
\geq \frac{\delta-\eps/2}{r!},
\]
or, equivalently, since $|\Aut(K_r)| = r!$ and $|\Aut(K_{1,r-1})| = (r-1)!$,
\begin{equation}\label{eq:embemb} |\Emb(K_r,G)|
  + 
  r\cdot |\Emb(K_{1,r-1},G)|\cdot p^{\binom{r}{2}-r+1}
  \geq (\delta-\eps/2)\cdot n^rp^{\binom{r}{2}}.
\end{equation}

Recalling our assumption $e_G < (1-\eps)\cdot \varphi_r(\delta-\eps,c)\cdot
n^2p^{r-1}\ll n^2$, it follows from Lemma~\ref{lem:lz} that
there is a partition $V(G) = U\cup V$ such that
$|U|\leq \eps^2 n$,
\[ |\Emb(K_r,G[V])|
\geq 
|\Emb(K_r,G)|
- \eps^2 e_G^{r/2}
\geq
|\Emb(K_r,G)| - \eps n^rp^{\binom{r}{2}}/4,
\]
and, recalling that $\Emb_U(K_{1,r-1}, G)$ denotes the set of embeddings of $K_{1,r-1}$ into $G$ that map the centre vertex of $K_{1,r-1}$ to a vertex of $U$,
\[
  \begin{split}
    r\cdot |\Emb_U(K_{1,r-1},G[U,V])|
    \cdot p^{\binom{r}{2}-r+1} & \geq 
    r\cdot |\Emb(K_{1,r-1},G)| 
    \cdot
    p^{\binom{r}{2}-r+1}
    - \eps^2 e_G n^{r-2}\cdot 
    p^{\binom{r}{2}-r+1}\\
    &\geq 
    r\cdot |\Emb(K_{1,r-1},G)| 
    \cdot
    p^{\binom{r}{2}-r+1}
    - \eps n^rp^{\binom{r}{2}}/4
  \end{split},
\]
where the stated inequalities are valid if $\eps$ is sufficiently small.
From this and \eqref{eq:embemb}, we obtain
\[ |\Emb(K_r,G[V])|
+ 
r\cdot |\Emb_U(K_{1,r-1},G[U,V])|\cdot p^{\binom{r}{2}-r+1}
\geq (\delta-\eps)\cdot n^rp^{\binom{r}{2}};
\]
consequently, there exists an
$x\in [0,1]$ such that
\[ |\Emb(K_r,G[V])|
\geq (1-x)\cdot
(\delta-\eps)\cdot n^rp^{\binom{r}{2}}  \]
and
\[ r\cdot |\Emb_U(K_{1,r-1},G[U,V])|\cdot p^{\binom{r}{2}-r+1} \geq x\cdot
(\delta-\eps)\cdot n^rp^{\binom{r}{2}}. \]
By Theorem~\ref{thm:max-copies} and Lemma~\ref{lem:stars}, we thus obtain
the bounds
\begin{align*}
  &\begin{cases}
    (2e_{G[V]})^{r/2} \geq (1-x)\cdot (\delta-\eps)\cdot
    n^rp^{\binom{r}{2}}\\
    (\floor{e_{G[U,V]}/|V|} + \{e_{G[U,V]}/|V|\}^{r-1})\cdot n^{r-1} \geq x\cdot (\delta-\eps)\cdot
  n^rp^{r-1}/r, \end{cases}\\
  \intertext{and solving for $e_{G[V]}$ and $e_{G[U,V]}$, we get}
  &\begin{cases}
    e_{G[V]} \geq \big((1-x)\cdot (\delta-\eps) \big)^{2/r}\cdot
    \frac{n^2p^{r-1}}{2}\\
    e_{G[U,V]} \geq |V|\cdot \big(\floor{x\cdot (\delta-\eps) np^{r-1}/r}
    + \{x\cdot (\delta-\eps) np^{r-1}/r\}^{1/(r-1)}\big).
  \end{cases}
\end{align*}
Finally, since $|V| = n-|U|\geq (1-\eps^2)n$, the definition
of $\psi_r$ shows that
\[ e_G \geq e_{G[V]} + e_{G[U,V]}
\geq (1-\eps^2)\cdot \psi_r(\delta-\eps,np^{r-1},x)\cdot n^2p^{r-1}
. \]
As $\psi_r(\delta-\eps,np^{r-1},x)\geq 
\varphi_r(\delta-\eps,np^{r-1}) \to 
\varphi_r(\delta-\eps,c)$, this contradicts our assumption that $e_G
<
(1-\eps)\cdot \varphi_r(\delta-\eps,c)\cdot n^2p^{r-1},
$
provided that $\eps$ is sufficiently small and $n$ is large enough.

\subsection{Proof of Proposition \ref{prop:kr-gmin}}
Fix $\eps,\delta>0$ and $c\in [0,\infty]$ and assume that
$np^{r-1}\to c$. We fix three additional
positive constants $\eta$, $\eta'$, and $\gamma$, where
$\gamma$ is sufficiently small given the parameters
of the proposition,
$\eta'$ is sufficiently small given $\gamma$, and finally $\eta$ is
sufficiently small given both $\eta'$ and $\gamma$.

If $\Gmin(\eta)$ occurs, then $G_{n,p}$ contains a subgraph $G$ such
that $e_G \log(1/p)\leq (1+\eta) \Phi_X(\delta+\eta)$ and $\e_G[X]\geq
(1+\delta-\eta)\e[X]$. We claim that, if $n$ is sufficiently large,
then every such graph admits a partition $V(G)
= U\cup V$ such that, for some $x\in \bar X(\delta,c)$,
\begin{enumerate}[label=(\roman*)]
  \item\label{it:clique}
    $V$ contains a subset $V'$  of size at least $(1-\eps)
    \big(\delta(1-x)\big)^{1/r}np^{(r-1)/2}$ such that
    $G[V']$ has minimum degree at least $(1-\eps)|V'|$.
  \item \label{it:hub}
    $U$ contains a set $W\subseteq U$ 
    such that at least 
    $\floor{(1-\eps)|W|}$ vertices in $W$ have
    degree at least $(1-\eps) n$ and
    \[ e(W,V)\geq (1-\eps) n \big( \floor{\delta xnp^{r-1}/r} +
    \{\delta xnp^{r-1}/r\}^{1/(r-1)}\big). \]
\end{enumerate}
Note that these properties imply $\Clique_\eps (\delta (1-x))\cap
\Hub_\eps(\delta x)$.

By repeating the argument found in the proof of the lower bound in Proposition~\ref{prop:kr-Phi},
we can find a partition $V(G) = U\cup V$ and some $x'\in [0,1]$ such that
\begin{align}
  &\begin{cases}
    |\Emb(K_r,G[V])| \geq (1-x')\cdot (\delta-2\eta)\cdot n^rp^{\binom{r}{2}}\\
    r\cdot |\Emb_U(K_{1,r-1},G[U,V])|\cdot p^{\binom{r}{2}-r+1} \geq x'\cdot
    (\delta-2\eta)\cdot n^rp^{\binom{r}{2}}
  \end{cases}
  \label{eq:emblower} 
  \\
  \intertext{and}
  &\begin{cases}
  e_{G[V]} \geq \big((1-x')\cdot (\delta-2\eta) \big)^{2/r}\cdot
\frac{n^2p^{r-1}}{2}\\
  e_{G[U,V]} \geq
  \big(\floor{x'(\delta-2\eta) np^{r-1}/r}
  + \{x'(\delta-2\eta) np^{r-1}/r\}^{1/(r-1)}\big)\cdot n.
\end{cases}
\label{eq:eglower} 
\end{align}
It follows that $e_G\geq 
e_{G[V]}+
e_{G[U,V]}
\geq \psi_r(\delta-2\eta,np^{r-1},x')\cdot n^2p^{r-1}
$. Thus, by
Proposition~\ref{prop:kr-Phi},
\begin{equation}\label{eq:phiphi}
  \psi_r(\delta-2\eta,np^{r-1},x') \leq\frac{e_G}{n^2p^{r-1}}\leq
  \frac{(1+\eta)\Phi_X(\delta+\eta)}{n^2p^{r-1}\log(1/p)}\leq
 (1+2\eta)\varphi_r(\delta+\eta,c)\cdot
\end{equation}
Our next claim tells us how to choose the constant $\eta$.
Given a set $S\subseteq \RR$, we write
$B_{\eta'}(S) = \{ x\in \RR : \inf_{s\in S} |x-s|< \eps\}$ for  the
$\eta'$-neighbourhood of $S$.

\begin{claim}
  We may choose $\eta = \eta(\eta')>0$ such that $x'\in
  B_{\eta'}(\barX(\delta,c))$ whenever $n$ is sufficiently large.
\end{claim}
\begin{proof}
  Suppose there is no such choice.
  Then, by invoking \eqref{eq:phiphi} with
  successively smaller values of $\eta$, we find that there is a
  subsequence $(\eta_m,c_m,x_m)$
  of points in $(0,\infty)^2\times [0,1]$ such that
  $\eta_m\to 0$ and $c_m\to c$ as $m \to \infty$, and, for all $m$,
  we have $x_m \notin B_{\eta'}(\barX(\delta,c))$ and
  \[ \varphi_r(\delta-2\eta_m,c_m)\leq
    \psi_r(\delta-2\eta_m,c_m,x_m)\leq
    (1+2\eta_m)\varphi_r(\delta+\eta_m,c).
  \]
  Since both the left-hand and the right-hand sides above converge to $\varphi_r(\delta, c)$ as $m \to \infty$,
  so must $\psi_r(\delta-2\eta_m,c_m,x_m)$.
  Denote by $(\eta_\ell,c_\ell,x_\ell)$ a subsequence on which
  $x_\ell$ converges to some $x_\infty\in [0,1]$.
  We claim that $x_\infty\in \barX(\delta,c)$, which contradicts the fact
  that $x_\ell \notin B_{\eta'}(\barX(\delta,c))$ for all $\ell$.

  If $c = 0$, then we have $x_\infty = 0\in \barX(\delta,c)$, since otherwise
  the definition of $\psi_r$ would imply that
  $\psi_r(\delta-2\eta_\ell,c_\ell,x_\ell)\to \infty$. If $c\in (0,\infty)$, then
  continuity of $\psi_r$ implies that 
  $
  \varphi_r(\delta,c)=\lim_{\ell\to\infty}
  \psi_r(\delta-2\eta_\ell,c_\ell,x_\ell) = \psi_r(\delta,c,x_\infty)$ and thus $x_\infty \in \barX(\delta,c)$.
  Finally,
  if $c=\infty$, then $(\delta,x)\mapsto \psi_r(\delta,c_m,x)$ converges
  uniformly to the continuous function
  $ (\delta,x)\mapsto (\delta
  x)^{2/r}/2 + \delta x/r$,
  which implies that
  \[ \lim_{\ell\to\infty}\psi_r(\delta-2\eta_\ell,c_\ell,x_\ell)
  = 
  \frac{(\delta
  x_\infty)^{2/r}}2 + \frac{\delta x_\infty}r = \lim_{c'\to\infty}
  \psi(\delta,c',x_\infty),
  \]
  so $x_\infty\in \barX(\delta,c)$ in this case as well.
\end{proof}

Suppose now that  $x'\in B_{\eta'}(\bar X(\delta,c))$.
Since the right-hand sides of~\eqref{eq:emblower} and~\eqref{eq:eglower} are continuous in $x'$,
we may choose $\eta$ and $\eta'$ sufficiently small, as a function of $\gamma$, so that there
is some $x\in \barX(\delta,c)$ such that
\begin{align*}
  &\begin{cases}
    |\Emb(K_r,G[V])| \geq (1-x)\cdot (\delta-\gamma)\cdot n^rp^{\binom{r}{2}}\\
    r\cdot |\Emb_U(K_{1,r-1},G[U,V])|\cdot p^{\binom{r}{2}-r+1} \geq x\cdot
    (\delta-\gamma)\cdot n^rp^{\binom{r}{2}},
  \end{cases}\\
  &\begin{cases}
    e_{G[V]} \geq \big((1-x)\cdot (\delta-\gamma) \big)^{2/r}\cdot
    \frac{n^2p^{r-1}}{2}\\
    e_{G[U,V]} \geq
    \big(\floor{x(\delta-\gamma) np^{r-1}/r}
    + \{x(\delta-\gamma) np^{r-1}/r\}^{1/(r-1)}\big)\cdot n.
  \end{cases}
  \\
  \intertext{
  Since $e_{G[V]} + e_{G[U,V]} \le e_G \le (1+2\eta)\psi_r(\delta+\eta,c,x')\cdot n^2p^{r-1}$ and since
  the lower bounds on $e_{G[V]}$ and $e_{G[U,V]}$ in~\eqref{eq:eglower} sum to
  $\psi_r(\delta-2\eta, np^{r-1}, x') \cdot n^2p^{r-1}$, the two lower bounds on $e_{G[V]}$ and $e_{G[U,V]}$ stated above
  must be nearly tight. More precisely, the continuity of $\psi_r$ implies that we may choose $\eta$
  and $\eta'$ sufficiently small so that}
  &\begin{cases}
    e_{G[V]} \leq \big((1-x)\cdot (\delta+\gamma) \big)^{2/r}\cdot
    \frac{n^2p^{r-1}}{2}\\
    e_{G[U,V]} \leq
    \big(\floor{x(\delta+\gamma) np^{r-1}/r}
    + \{x(\delta+\gamma) np^{r-1}/r\}^{1/(r-1)}\big)\cdot n.
  \end{cases}
\end{align*}

Finally, if we choose $\gamma$ sufficiently small, then the above statements 
for $G[V]$ and Theorem~\ref{thm:Kr-stability} yield a set $V'\subseteq V$
satisfying the conditions in \ref{it:clique}.
Similarly, the statements for $G[U,V]$ and Lemma~\ref{lem:stars}(ii) yield
a subset $W\subseteq U$ satisfying \ref{it:hub}.

\section{Extensions to regular graphs}\label{sec:H}

Fix a connected and $\Delta$-regular graph $H$. In this section, we apply
Theorem~\ref{thm:packaged} to study the upper tail of the
random variable $X =
X_{n,p}^H$.
In this setting,
\eqref{eq:phi} may be rewritten as
\[
  \Phi_X(\delta) = \min\big\{e_G\log(1/p): G \subseteq K_n \text{ and }
  \e_G[X]\geq  (1+\delta)\e\left[X\right]\big\},
\]
where we use the notation $\e_G[X] = \e[X \mid G\subseteq G_{n,p}]$.
Our main result in this section is the following:

\begin{proposition} \label{prop:ldp-H}
  For every $\Delta \ge 2$, every connected, nonbipartite, $\Delta$-regular
  graph $H$, and all positive real numbers $\eps$ and $\delta$, there
  exists a positive constant $C$ such that the following holds. Suppose that
  an integer $n$ and $p\in(0,1)$ satisfy 
  $Cn^{-1} (\log n)^{\Delta v_H^2} \le p^{\Delta/2} \le 1/C$.
  Then $X = X_{n,p}^H$ satisfies
  \[
    (1-\eps)\Phi_X(\delta-\eps) \leq -\log \Pr\big(X \geq (1+\delta)\e[X]\big)
    \leq (1+\eps) \Phi_X(\delta+\eps).
  \]
  Additionally, there is a positive constant $\xi = \xi(\Delta,\eps)$ such that the assumption
  that $H$ is nonbipartite is not necessary when $p^{\Delta/2} \geq n^{-1/2-\xi}$.
\end{proposition}
\begin{remark}
As was mentioned in the introduction, the lower bound on the density $p$ is suboptimal by a polylogarthmic factor. We require this slightly stronger bound on the density to prove an analogue of Claim~\ref{claim:core-degree} for $\Delta$-regular graphs $H$. In this more general setup, the lower bound on the number of common neighbours of the endpoints of core edges may not hold, and will be replaced by a lower bound on the product of degrees of the endpoints of core edges. We will prove this bound via an application of Lemma~\ref{lemma:core-edge-regular}; unfortunately, this will incur a polylogarthmic loss, and thus necessitates the suboptimal bound on the density $p$.
\end{remark}

The optimisation problem $\Phi_X(\delta)$ is a discretisation of the variational problem considered by  Bhattacharya, Ganguly, Lubetzky, and Zhao~\cite{bhattacharya2017upper}. As was the case in the context of arithmetic progressions and cliques, their variational problem was more general (being optimised over a larger set), but is asymptotically equivalent to the one considered in this paper. Their results imply
\[ \lim_{n\to \infty} \frac{\Phi_X(\delta)}{n^2p^{\Delta}\log(1/p)}
= 
  \begin{cases}
    \delta^{2/v_H}/2 & \text{if $np^\Delta\to 0$,}\\
    \min{\{\delta^{2/v_H}/2,\theta\}} & \text{if $np^\Delta\to \infty$},
\end{cases}\]
where $\theta$ is the unique positive solution to $P_H(\theta) = 1+\delta$ and $P_H$ is
the independence polynomial of $H$. Thus,
Proposition~\ref{prop:ldp-H} implies Theorem~\ref{thm:H}. The proof of
Proposition~\ref{prop:ldp-H} does not require these precise estimates, but only
that $\Phi_X(\delta)$ is of order $n^2p^\Delta\log(1/p)$, see Lemma~\ref{lemma:H-phi} below. We include a short proof of this weaker statement for the sake of completeness, and to emphasize it may be proved combinatorially; the original proof given in~\cite{bhattacharya2017upper} is conceptually equivalent but more analytically-flavoured.

\begin{lemma}
  \label{lemma:H-phi}
  For every $\Delta\geq 2$, every connected, $\Delta$-regular graph $H$, and
  every positive real number $\delta$, there exists a positive
  constant $C$ such that the following holds. Assume $n\in \NN$ and $p\in
  (0,1)$ are such that
  $Cn^{-1} \le p^{\Delta/2} \le 1/C$. Then $X = X_{n,p}^H$ satisfies
  \[
    1/C \le \frac{\Phi_X(\delta)}{n^2p^\Delta\log(1/p)} \le C.
  \]
\end{lemma}

\begin{proof}
  The upper bound follows by noting that, if $C$ is large enough, then a
  clique with $\lceil (1+2\delta)^{1/v_H} n p^{\Delta/2} \rceil$ vertices
  contains at least $(1+\delta) \Ex[X]$ copies of $H$ and has fewer than
  $C n^2 p^\Delta$ edges.
  For the lower bound, suppose that $G$ is a graph with $\e_G[X]\geq
  (1+\delta)\e[X]$ with fewer than $C^{-1}n^2p^{\Delta}$ edges.
  Then
  \[ \frac{\delta}{2} \cdot \frac{n^{v_H}p^{e_H}}{|\Aut(H)|}\leq \delta \e[X]
  \leq \e_G[X] - \e[X] \leq \sum_{\emptyset \neq J \subseteq H} N(J,G)\cdot
  n^{v_H-v_J}p^{e_H-e_J}, \]
  where the sum ranges over all nonempty subgraphs $J$ of $H$ without isolated
  vertices. This implies that there is a nonempty subgraph $J\subseteq H$
  without isolated
  vertices 
  and a positive constant $\gamma = \gamma(H,\delta)$
  such that
  \begin{equation}\label{eq:embjg3} |\Emb(J,G)| \geq \gamma\cdot n^{v_J}p^{e_J}. \end{equation}
  Theorem~\ref{thm:max-copies} implies that
  \[
    |\Emb(J,G)| \leq (2e_G)^{v_J-\alpha_J^*} \cdot \min{\{2e_G,n\}}^{2\alpha_J^*-v_J}
  \]
  and Lemma~\ref{lemma:eJ-alphaJ-clique} yields $\alpha_J^*\leq v_J-e_J/\Delta$.
  Therefore, if $2e_G\leq n$, then $|\Emb(J, G)|$ is bounded from above by
  \[ (2e_G)^{\alpha_J^*}
    \leq
    (2e_G)^{\alpha_J^*-v_J+2e_J/\Delta }\cdot n^{v_J-2e_J/\Delta}
    \leq
    \left(\frac{2n^2p^\Delta}{C}\right)^{e_J/\Delta }n^{v_J-2e_J/\Delta}
    = \frac{n^{v_J}p^{e_J}}{(C/2)^{e_J/\Delta}},
  \]
  where the first inequality holds since $2e_J/\Delta \le v_J$, as the maximum degree of $J$ is at most $\Delta$.
  If $n < 2e_G$, then $|\Emb(J,G)|$ is bounded from above by
  \[
    \left(\frac{2n^2p^\Delta}{C}\right)^{v_J-\alpha_J^*} \cdot n^{2\alpha_J^*-v_J}
    = \frac{n^{v_J}p^{e_J}}{(C/2)^{v_J-\alpha_J^*}}
    \cdot p^{\Delta(v_J-\alpha_J^*) - e_J}
    \leq
    \frac{n^{v_J}p^{e_J}}{(C/2)^{v_J-\alpha_J^*}}.
  \]
  In both cases, the obtained upper bound on $|\Emb(J,G)|$ contradicts \eqref{eq:embjg3} whenever $C$ is large;
  indeed $e_J/\Delta$ and $v_J-\alpha_J^*$ are both positive, as $J$ is nonempty. 
\end{proof}

\subsection{Proof of Proposition~\ref{prop:ldp-H}}\label{ssec:H-proof}

Fix $\Delta \ge 2$, a nonempty, connected, $\Delta$-regular graph $H$, and
positive reals $\eps$ and $\delta$. Without loss of generality, we may assume
that $\eps \le \min{\{1/3,\delta/2\}}$. Let $X = X_{n,p}^H$ and assume $Cn^{-1}
(\log n)^{\Delta v_H^2} \leq p^{\Delta/2} \leq 1/C$. 
Note that the case $n < v_H$ is trivial as then $X$ is identically
zero. We may therefore assume that $n\geq v_H \geq \Delta+1\geq 3$, which, in
turn, implies that $n\geq C$.

Set $N = \binom{n}{2}$ and let $Y = (Y_1,\dotsc,Y_N)$ be the sequence of indicator
random variables of the events that $e\in E(G_{n,p})$, where $e$ ranges over
$\binom{\br n}{2}$ in some arbitrary order. Then $Y$ is a vector of independent $\Ber(p)$
random variables and $X$ is a nonzero polynomial with nonnegative coefficients and degree at most $e_H$ in the coordinates of $Y$.
Let $K = K(e_H,\eps,\delta)$ be the constant whose existence is asserted by
Theorem~\ref{thm:packaged}. To prove the proposition, it suffices to verify the
assumptions of the theorem.

It follows from
Lemma~\ref{lemma:H-phi} and our assumptions on $p$ that $p\leq
1-\eps$ and $\Phi_X(\delta-\eps) \geq K \log(1/p)$ for a large enough choice
of $C$.
It thus only remains to bound the number of cores of a
given size. To this end, let $\core_m$ be the set of cores with $m$ edges,
that is, subgraphs $\Gcore\subseteq K_n$ such that
\begin{enumerate}[label=(C\arabic*)]
  \item\label{item:Hcore-bias}
    $\e_\Gcore[X] \geq (1+\delta-\eps)\e[X]$,
  \item\label{item:Hcore-size}
    $e_\Gcore = m\leq K \cdot \Phi_X(\delta+\eps)$, and
  \item\label{item:Hcore-mindeg}
    $\min_{e\in E(\Gcore)}\left(\e_\Gcore[X]-\e_{\Gcore\setminus e}[X]\right)
    \geq \e[X]/(K\cdot \Phi_X(\delta+\eps))$.
\end{enumerate}
Proposition~\ref{prop:ldp-H} will follow once we prove
\begin{equation}\label{eq:Hcore-count}
  |\core_m| \leq (1/p)^{\eps m/2}\quad \text{for all $m$}.
\end{equation}
Observe that due to \ref{item:Hcore-bias}, \ref{item:Hcore-size}, and the
definition of $\Phi_X$, it suffices to verify \eqref{eq:Hcore-count}
for integers $m$ such that  $\mmin \leq m \leq \mmax$ with
\begin{align*}
  &\mmin  := \Phi_X(\delta-\eps)/\log(1/p) 
  \ge n^2p^\Delta/K'\text{ and} \\
  &\mmax  := K\cdot  \Phi_X(\delta+\eps)
  \leq  K'\cdot n^2p^\Delta \log (1/p),
\end{align*}
where the stated inequalities follow,
for a suitable constant $K' = K'(H,\eps,\delta)$,
from Lemma~\ref{lemma:H-phi}, our
bounds on $p$, and the assumption that $C$ is sufficiently large.

The first step towards establishing~\eqref{eq:Hcore-count} is to understand the combinatorial meaning
of~\ref{item:Hcore-mindeg}. Suppose that $\mmin \le m \le \mmax$ and
$\Gcore \in \core_m$. Recall that $N(J,\Gcore;e)$ denotes the number of copies
of $J$ in $\Gcore$ that contain the edge $e$.
Note that \ref{item:Hcore-mindeg} implies that, for every $e \in E(\Gcore)$,
\begin{equation}
  \label{eq:H-core-meaning}
   \frac{\Ex[X]}{\mmax} \le \Ex_{\Gcore}[X] - \Ex_{\Gcore \setminus e}[X] \leq
  \sum_{\emptyset \neq J \subseteq H} N(J,\Gcore; e) \cdot n^{v_H-v_J}
  p^{e_H - e_J},
\end{equation}
where the sum ranges over the nonempty subgraphs $J$ of $H$ without isolated vertices.
Since $n\geq C$ for a large enough constant
$C$, we can bound
$\Ex[X] \ge
\binom{n}{v_H} p^{e_H} \ge \frac{n^{v_H}p^{e_H}}{2v_H!}$
and, consequently, \eqref{eq:H-core-meaning} implies that there is 
a positive constant $\gamma=\gamma(H)$ such that
\begin{equation}
  \label{eq:H-core-meaning-NGJe}
  \sum_{\emptyset \neq J \subseteq H} \frac{N(J,\Gcore; e)}{n^{v_J}p^{e_J}}
  \ge \frac{2\gamma}{\mmax} \qquad \text{for every $e \in E(\Gcore)$}.
\end{equation}

\begin{definition}\label{def:qh}
  Let $\QH$ denote the family of
  all nonempty subgraphs $J \subseteq H$
  without isolated vertices satisfying
  \begin{enumerate}[label=(Q\arabic*)]
  \item
    \label{item:QH-regular2}
    $J = H$ or
  \item
    \label{item:QH-bipartite2}
    $J$ admits a bipartition $V(J) = A \cup B$ such that $\deg_Ja = \Delta$ for all $a \in A$.
  \end{enumerate}
\end{definition}

Our first claim is that, for the vast majority of $e \in E(\Gcore)$, the sum on the left-hand side
of~\eqref{eq:H-core-meaning-NGJe} is dominated by subgraphs $J\in\QH$. Let
$\Gcoreb$ comprise all edges $e$ of $\Gcore$ such that
\[
  \sum_{J \in \QH} \frac{N(J,\Gcore; e)}{n^{v_J}p^{e_J}} <
  \frac{\gamma}{\mmax}.
\]

\begin{claim}
  \label{claim:mb-upper}
  There is a positive constant $\sigma=\sigma(H)$ such that $e_\Gcoreb\leq
  p^\sigma\cdot \mmin$.
\end{claim}
\begin{proof}
  Let $\mb$ denote the number of edges of $\Gcoreb$. The definition of
  $\Gcoreb$ and~\eqref{eq:H-core-meaning-NGJe} imply that
  \[
    \sum_{e \in E(\Gcoreb)} \sum_{\substack{\emptyset \neq J \subseteq H
    \\ J \notin \QH}} \frac{N(J,\Gcore;
    e)}{n^{v_J}p^{e_J}} \ge \frac{\gamma\cdot\mb}{\mmax}
    \geq \frac{\gamma\cdot\mb}{\mmin \cdot (K')^2\cdot \log(1/p)}
    .
  \]
  On the other hand, since
  $\sum_{e \in E(\Gcore)} N(J,\Gcore; e) \leq e_J\cdot |\Emb(J,\Gcore)|$ for every graph $J$,
  there must exist a nonempty subgraph $J\subseteq H$ without isolated vertices such that
  $J\notin \QH$ and
  \begin{equation}
    \label{eq:mb-upper}
    \frac{\mb}{\mmin}\leq
    \frac{e_J\cdot (K')^2\cdot \log(1/p)\cdot 2^{e_H+v_H}}{\gamma} \cdot 
    \frac{|\Emb(J,\Gcore)|}{n^{v_J} p^{e_J}}.
  \end{equation}

  We now show that the right-hand side of~\eqref{eq:mb-upper} is at most $p^\sigma$, for some positive constant $\sigma = \sigma(H)$.
  To this end, recall that Theorem~\ref{thm:max-copies} states that
  \begin{equation}
    \label{eq:mb-upper2}
    |\Emb(J,\Gcore)| \leq (2m)^{v_J-\alpha_J^*} \cdot \min{\{2m,n\}}^{2\alpha_J^*-v_J}.
  \end{equation}
  Since $J$ is a proper subgraph of a connected, $\Delta$-regular graph, it must contain
  a vertex of degree smaller than $\Delta$ and hence $v_J>2e_J/\Delta$. Moreover, since $J \notin \QH$, Lemma~\ref{lemma:eJ-alphaJ-clique}
  implies that $\alpha_J^* < v_J-e_J/\Delta$.
  Therefore, if $2m\leq n$, then there is a positive $\sigma = \sigma(H)$ such that
  the right-hand side of \eqref{eq:mb-upper2} can be bounded from above as follows:
  \[\begin{split} (2m)^{\alpha_J^*}
    & \le
    (2m)^{\alpha_J^*-v_J+2e_J/\Delta+2\sigma}\cdot n^{v_J-2e_J/\Delta-2\sigma}
    \\
    & \le (2m)^{e_J/\Delta} \cdot n^{v_J-2e_J/\Delta-2\sigma}
    \\
    &\leq
  \left(2 K' \cdot n^2p^\Delta \log(1/p)\right)^{e_J/\Delta}n^{v_J-2e_J/\Delta-2\sigma}\\
    &= n^{-2\sigma} \cdot \big(2K'\cdot \log(1/p)\big)^{e_J/\Delta}\cdot n^{v_J}p^{e_J}.
  \end{split}\]
  Similarly, if
  $n < 2m$, there is a positive $\sigma = \sigma(H)$ such that
  the right-hand side of \eqref{eq:mb-upper2} is bounded from above by
  \[
    \begin{split}
      (2m)^{v_J - \alpha_J^*} \cdot n^{2\alpha_J^* - v_J} & \le
    \big(2K'\cdot n^2p^\Delta\log(1/p)\big)^{v_J-\alpha_J^*} \cdot n^{2\alpha_J^*-v_J} \\
    &= p^{\Delta(v_J-\alpha_J^*) - e_J}
    \cdot \big(2K'\cdot \log(1/p)\big)^{v_J-\alpha_J^*}
    \cdot n^{v_J}p^{e_J}\\
    &\leq
    p^{2\sigma}\cdot
    \big(2K'\cdot \log(1/p)\big)^{v_J-\alpha_J^*}
    \cdot n^{v_J}p^{e_J}.
  \end{split}
  \]
  Since $p\leq C^{-2/\Delta}$ and $Cn^{-1}\leq p$, it follows that,
  in both cases,
  \[ 
  |\Emb(J,\Gcore)| \leq p^\sigma\cdot n^{v_J}p^{e_J}\cdot
  \frac{\gamma}{e_J\cdot (K')^2\cdot \log(1/p)\cdot 2^{e_H+v_H}}, \]
  provided that $C$ is sufficiently large. Substituting this inequality into \eqref{eq:mb-upper}
  proves the claim.
\end{proof}

The next claim shows that the endpoints of every edge of $\Gcore\setminus \Gcoreb$
satisfy a certain degree restriction.

\begin{claim}
  \label{claim:core-degrees}
  There is a positive $\gamma' = \gamma'(H,K')$
  such that, for every
  edge $uv$ of $\Gcore \setminus \Gcoreb$, either
  \[
    \deg_{\Gcore}u \cdot \deg_{\Gcore}v \ge \frac{\gamma'}{\big(\log
    (1/p)\big)^{v_H}} \cdot m \qquad \text{or} \qquad \deg_{\Gcore}u +
    \deg_{\Gcore}v \ge \frac{\gamma'}{\big(\log(1/p)\big)^{v_H/2}} \cdot n.
  \]
  Moreover, if the second inequality fails, then $uv$ is contained in at least
  $\gamma' n^{v_H}p^{e_H} / \mmax$ copies of $H$ in $\Gcore$.
\end{claim}
\begin{proof}
  Suppose that $uv$ is an edge of $\Gcore \setminus \Gcoreb$. It follows
  from~\eqref{eq:H-core-meaning-NGJe} and the definition of $\Gcoreb$ that
  there is some $J \in \QH$ such that
  \begin{equation}
    \label{eq:H-core-meaning-NGJuv}
    |\Emb(J,G; uv)| \ge N(J,G;uv) \ge \frac{\gamma}{|\QH|} \cdot \frac{n^{v_J}p^{e_J}}{\mmax}.
  \end{equation}
  Let $\mu = n^2p^\Delta/m$ and observe that $\mu \ge \big(
  K'\cdot \log(1/p)\big)^{-1} \cdot \mmax/m$. We split the remainder of the proof
  into two cases, depending on whether or not $J = H$.

  Assume first $J = H$. Then, since $n^{v_H}p^{e_H} =
  \big(n^2p^\Delta\big)^{v_H/2} = (\mu m)^{v_H/2}$,
  Lemma~\ref{lemma:core-edge-regular} and~\eqref{eq:H-core-meaning-NGJuv}
  imply that
  \[
    \begin{split}
      \big(4\deg_{\Gcore}u \cdot \deg_{\Gcore}v\big)^{\frac{\Delta-1}{\Delta}} &
      \ge |\Emb(J,G;uv)| \cdot \frac{ (2m)^{\frac{2\Delta-1}{\Delta} - \frac{v_H}{2}} }{4e_H}\\
      &\ge \frac{\gamma}{|\QH|} \cdot \frac{(\mu m)^{v_H/2}}{\mmax} \cdot
      \frac{(2m)^{\frac{2\Delta-1}{\Delta} - \frac{v_H}{2}} }{4e_H} \\
      & = \frac{2^{\frac{2\Delta-1}{\Delta} - \frac{v_H}{2}}\gamma}{4e_H|\QH|}
      \cdot \mu^{v_H/2} \cdot \frac{m}{\mmax} \cdot m^{\frac{\Delta-1}{\Delta}}
      \\
      & \ge (4\gamma')^{\frac{\Delta-1}\Delta} \cdot \big(\log(1/p)\big)^{-v_H/2} \cdot
      m^{\frac{\Delta-1}{\Delta}},
    \end{split}
  \]
  for a suitable positive constant $\gamma' = \gamma'(H,K')$.
  Since $\Delta \ge 2$, this implies the claimed lower bound on $\deg_{\Gcore} u \cdot \deg_{\Gcore} v$.

  Assume now that $J\neq H$. In this case, $J$ admits a bipartition $V(J) = A
  \cup B$ such that every vertex in $A$ has degree $\Delta$. In particular,
  $e_J = |A| \cdot \Delta$; moreover,
  as $J \neq H$ and $H$ is connected
  and $\Delta$-regular,
  we also have $|B| > |A|$. Since $n^{v_J} p^{e_J} = \big(n^2p^\Delta\big)^{|A|}
  \cdot n^{|B|-|A|} = (\mu m)^{|A|} \cdot n^{|B|-|A|}$, it follows
  from Lemma~\ref{lemma:core-edge-bipartite}
  and~\eqref{eq:H-core-meaning-NGJuv} that
  \[
    \begin{split}
      \deg_{\Gcore}u + \deg_{\Gcore}v 
      & \geq \frac{|\Emb(J,G;uv)|}{e_J\cdot (2m)^{|A|-1} \cdot \big(\min\{m,n\}
      \big)^{|B|-|A|-1}}\\
      & \ge \frac{\gamma}{|\QH|} \cdot \frac{(\mu m)^{|A|} \cdot
      n^{|B|-|A|}}{\mmax} \cdot \frac{ 1}{e_J\cdot (2m)^{|A|-1} \cdot
      \big(\min\{m, n\}\big)^{|B|-|A|-1} } \\
      & \ge \frac{\gamma}{|\QH|} \cdot \frac{\mu^{|A|}m
      }{\mmax} \cdot \frac{n}{e_J\cdot 2^{|A|-1}} \\
      & \ge \gamma' \cdot \big(\log(1/p)\big)^{-|A|} \cdot n
    \end{split}
  \]
  for a suitable positive constant $\gamma'$. Since $|A| < v_H/2$, this
  implies the claimed lower bound on $\deg_{\Gcore}u + \deg_{\Gcore}v$.
  In particular, if the second inequality in the statement of the claim
  fails, then the above shows that $J = H$, and thus the second assertion
  of the claim follows from \eqref{eq:H-core-meaning-NGJuv}.
\end{proof}

To prove \eqref{eq:Hcore-count}, we will further distinguish between the cases
$p^{\Delta/2} \ge n^{-1/2 - \xi}$ and $p^{\Delta/2} \le n^{-1/2 - \xi}$, for a
small constant $\xi$. The first case is easier and is handled by our next
claim.

\newcommand{\kmax}{{k_\text{max}}}

\begin{claim}
  There is a positive constant $\xi=\xi(\Delta,\eps)$ such
  that \eqref{eq:Hcore-count} holds if $p^{\Delta/2}\geq n^{-1/2-\xi}$.
\end{claim}
\begin{proof}
  To simplify the presentation, we shall prove \eqref{eq:Hcore-count} with
  $7\eps$ instead of $\eps/2$.
  Suppose that $\Gcore$ is a core with $m$ edges, where $\mmin\leq m
  \leq \mmax$. For each $k \in \ZZ$, we define the set
  \[
    B_k = \left\{v \in V(\Gcore) : \deg_\Gcore v \ge
    e^k\sqrt{m}p^{\eps}\right\}
  \]
  and let $\kmax$ be the largest integer such that $e^\kmax
  \sqrt{m} \le np^{2\eps}$. Note that 
  the bounds $n^2p^\Delta/K'\leq m \leq K'\cdot n^2p^\Delta\log(1/p)$ and 
  our assumption
  $p\leq C^{-2/\Delta}$ imply that
  $0\leq \kmax\leq (\Delta/2)\cdot \log (1/p)$
  whenever $C$ is large enough.

  We will first prove that Claim~\ref{claim:core-degrees} implies that, for
  each edge $uv$ of $\Gcore \setminus \Gcoreb$, either
  \begin{enumerate}[{label=(\roman*)}]
    \item\label{item:Huv1}
      $uv$ has an endpoint in $B_\kmax$
      or
    \item\label{item:Huv2}
      $u \in B_k$ and $v \in B_{-k}$ for some $k \in \{-\kmax, \dotsc,
      \kmax\}$.
  \end{enumerate}
  Indeed, if we suppose that
  \[ \deg_\Gcore u+ \deg_\Gcore v\geq
  \frac{\gamma' \cdot n}{(\log(1/p))^{v_H/2}} ,
   \]
  then
  \ref{item:Huv1} 
  follows immediately (with room to spare) because
  $
  \frac{\gamma' \cdot n}{(\log(1/p))^{v_H/2}}
 \geq 
  np^{3\eps}
  \geq 
  2e^\kmax \sqrt{m}p^{\eps}$.
  On the other hand, if \ref{item:Huv1} does not hold (i.e., neither $u$ nor $v$ is in $B_{\kmax}$), then the above lower bound on $\deg_{\Gcore} u + \deg_{\Gcore} v$ does not hold, and Claim~\ref{claim:core-degrees} guarantees
  \[ 
    \deg_\Gcore u\cdot \deg_\Gcore v\geq \frac{\gamma'\cdot m}{
    (\log(1/p))^{v_H}}
  \geq emp^{2\eps}.
\]
Observe that the lower bound above and the assumption on $u$ and $v$ imply that
\[
\min\{ \deg_{\Gcore} u, \deg_{\Gcore} v\} = \frac{ \deg_{\Gcore} u \cdot \deg_{\Gcore} v}{\max\{ \deg_{\Gcore} u, \deg_{\Gcore} v\}} > \frac{emp^{2\eps}}{e^{\kmax} \sqrt{m} p^\eps} = e^{-\kmax + 1}\sqrt{m} p^{\eps}.
\]
Let $k \in \{-\kmax+1, \dotsc, \kmax-1\}$ be the largest index such that $\{u,v\} \subseteq B_k$. Without loss of generality, $u \not \in B_{k+1}$, so 
\[
 emp^{2\eps} \leq \deg_{\Gcore} u \cdot \deg_{\Gcore} v < e^{k+1} \sqrt{m} p^{\eps} \cdot \deg_{\Gcore} v,
\]
which implies \ref{item:Huv2}.

  To make use of this property of the edges of $\Gcore \setminus \Gcoreb$, we also require upper bounds on the
  cardinalities of the sets $B_k$. To that end, observe that, for all $k\in \ZZ$,
  \[
    2m \ge \sum_{v \in B_k} \deg_\Gcore v \ge |B_k| \cdot e^k
    \sqrt{m}p^{\eps} \]
  and hence
  \begin{equation}
    \label{eq:Bk-upper}
    |B_k| \le 
    2e^{-k} \sqrt{m}p^{-\eps}.
  \end{equation}
  We claim that there is a positive constant $\xi= \xi(\Delta, \eps)$ such that $n^2p^\Delta \ge 2K'np^\eps \log n$ whenever $p^{\Delta/2} \ge n^{-1/2-\xi}$ and $C$ is sufficiently large.
  Indeed, if $p^{\Delta/2} \ge n^{-1/3}$, then $n^2p^\Delta \ge n^{4/3} \ge 2K'np^\eps \log n$, provided that $C$ is sufficiently large; otherwise, $p^\eps \le n^{-2\eps/(3\Delta)}$
  and our assumption implies that $n^2p^\Delta \ge n^{1-2\xi} \ge K'np^\eps\log n$, provided that $\xi$ is sufficiently small and $C$ is sufficiently large.
  Now inequality~\eqref{eq:Bk-upper} and the
  choice of $\kmax$ imply
  \begin{equation}
    \label{eq:B-K-upper}
    |B_{-\kmax}| \le
    2e^{\kmax} \sqrt{m}p^{-\eps}
    \leq 2np^\eps
    \leq \frac{m}{\log n}
  \end{equation}
  and (since $e^{\kmax+1}> np^{2\eps}$)
  \begin{equation}
    \label{eq:BK-upper}
    |B_\kmax|\cdot n \le 2e^{-\kmax} \sqrt{m}p^{-\eps}\cdot n
    < 2e mp^{-3\eps} .
  \end{equation}

  Recall from Claim~\ref{claim:mb-upper} that $e_{\Gcoreb}\leq p^\sigma\cdot \mmin$ for a positive constant
  $\sigma$ depending only on $H$. It follows that we may construct each
  $\Gcore \in \core_m$ as follows:
  \begin{enumerate}[{label=(\arabic*)}]
    \item
      Choose some $\mb \le p^\sigma\cdot \mmin$ and then choose $\mb$ edges of
      $K_n$ to form $\Gcoreb$.
    \item
      Choose the sets $B_{-\kmax}, \dotsc, B_\kmax$ and then
      choose $m-\mb$ edges from
      \[
        \cB = 
        \big\{uv \in E(K_n) : u \in B_{\kmax}\big\}\cup
        \bigcup_{k=0}^{\kmax} \big\{ uv \in E(K_n) : u \in B_k, v \in
        B_{-k}\big\} 
      \]
      to form $\Gcore\setminus \Gcoreb$.
  \end{enumerate}

  Since $B_{\kmax}\subseteq \dotsb \subseteq B_{-\kmax}$
  and $|B_{-\kmax}|\leq m/\log n$ by \eqref{eq:B-K-upper}, the number of
  ways to choose the sets $B_{-\kmax},\dotsc,B_\kmax$ is at most
  \[ \big((2\kmax+2)\cdot n\big)^{m/\log n}
  \leq e^{2m}, \]
  using the (very crude) bound $\kmax\leq n/4$.
  Moreover, inequalities~\eqref{eq:Bk-upper} and \eqref{eq:BK-upper}
  imply that
  \[
    |\cB| \le 
     |B_{\kmax}| \cdot n+
    \sum_{k=0}^\kmax |B_k| \cdot |B_{-k}| \le
    2emp^{-3\eps} + 
    (\kmax+1) \cdot
    4mp^{-2\eps} \le mp^{-4\eps},
  \]
  where we use $\kmax\leq (\Delta/2)\cdot \log(1/p)$ and $p^{\Delta/2}\leq 1/C$ for a large enough $C$.
  We conclude that
  \[
    |\core_m| \le 
    e^{2m}\cdot 
    \sum_{\mb=0}^{p^\sigma\cdot \mmin}
    \binom{n^2}{\mb} \cdot \binom{mp^{-4\eps}}{m-\mb}.
  \]
  In order to bound the right-hand side above, we note that
  \[
    \binom{mp^{-4\eps}}{m-\mb} \le \binom{mp^{-4\eps}}{m} \le
    \left(\frac{em p^{-4\eps}}{m}\right)^m \le p^{-5\eps m}.
  \]
  Moreover, using the inequalities $\sum_{i=0}^k \binom{n}{i}\leq (en/k)^{k}$
  and $m\geq \mmin \ge n^2p^\Delta / K'$,
  \[
    \sum_{\mb=0}^{p^\sigma\cdot \mmin} \binom{n^2}{\mb} \le
    \left(\frac{en^2}{p^\sigma\cdot \mmin}\right)^{p^\sigma\cdot \mmin } \le
    \left(\frac{eK'}{p^{\Delta + \sigma}}\right)^{mp^\sigma}  \le
    p^{-\eps m}.
  \]
  We may conclude that
  \[
    |\core_m| \le e^{2m} \cdot p^{-6\eps m}
    \leq p^{-7\eps m},
  \]
  which completes the proof of \eqref{eq:Hcore-count}
  (with $7\eps$ instead of $\eps/2$).
\end{proof}

The argument above shows that we do not need the assumption
that $H$ is nonbipartite when $p^{\Delta/2} \geq n^{-1/2-\xi}$.
In the following, we will assume that $p^{\Delta/2}\leq n^{-1/2-\xi}$ and that
$H$ is not bipartite.

Let $\gamma'$ be the constant from Claim~\ref{claim:core-degrees}. 
Our assumption on $p$ implies that
\[
  \mmax\leq K'\cdot n^2p^\Delta \log(1/p) \le K'\cdot n^{1-2\xi} \log(1/p) <
  \frac{\gamma'\cdot n}{\big(\log(1/p)\big)^{v_H/2}}-1,
\]
where the last inequality holds because $n \ge C$ and $C$ is large.
In
particular, 
if $\Gcore\in \core_m$
for some $\mmin\leq m\leq \mmax$, then
for any two vertices $u,v\in V(\Gcore)$, we have
$\deg_\Gcore(u)+\deg_\Gcore(v)\leq \mmax+1 < 
  \gamma'n/(\log(1/p))^{v_H/2}$,
so it follows from Claim~\ref{claim:core-degrees} that 
\begin{equation}
  \label{eq:degree-product-lower}
  \deg_{\Gcore}u \cdot \deg_{\Gcore}v \ge \frac{\gamma'}{\big(\log
  (1/p)\big)^{v_H}} \cdot m \qquad \text{for every $uv \in E(\Gcore
  \setminus \Gcoreb)$}
\end{equation}
and that every edge $uv$ of $\Gcore \setminus \Gcoreb$ belongs to at least $\gamma' n^{v_H}p^{e_H} / \mmax$ copies of $H$
in $\Gcore$. Set
\[
  \beta = \Delta(1+ v_H/2) + 1/2,
\]
let $\Gcoreh$ comprise all edges $uv$ of $\Gcore$ such that
\begin{equation}
  \label{eq:degree-product-upper}
  \deg_\Gcore u \cdot \deg_\Gcore v \geq \big(\log(1/p)\big)^{\beta} \cdot m,
\end{equation}
and denote by $\mh$ the number of edges in $\Gcoreh$.
We claim that
\begin{equation}
  \label{eq:mh-upper}
  \mh \le \frac{8m}{\big(\log(1/p)\big)^{\beta}}.
\end{equation}
In order to show it, we estimate the number of copies of $P_4$, the path with four vertices (and three edges), in $\Gcore$
in two different ways. On the one hand,
\[
  2N(P_4, \Gcore) = |\Emb(P_4, \Gcore)| \le (2m)^2,
\]
since every embedding of $P_4$ into $\Gcore$ is determined by the images of its two nonincident edges. On the other hand,
\[
  \begin{split}
    N(P_4, \Gcore) & \ge \sum_{uv \in E(\Gcore)} (\deg_{\Gcore} u-1) \cdot (\deg_{\Gcore}v - 2) \\
    & = \sum_{uv \in E(\Gcore)} \deg_{\Gcore}u \cdot \deg_{\Gcore}v - 3\sum_{v \in V(\Gcore)} (\deg_{\Gcore}v)^2 + 2m \\
    & \ge \mh \cdot \big(\log(1/p)\big)^{\beta} \cdot m - 3m \sum_{v \in V(\Gcore)} \deg_{\Gcore}v \\
    & = \mh \cdot \big(\log(1/p)\big)^{\beta} \cdot m - 6m^2.
  \end{split}
\]
These two lower and upper bounds on $N(P_4, \Gcore)$ imply~\eqref{eq:mh-upper}.

\begin{claim}\label{cl:Hexc}
  Suppose that $\varphi$ is an embedding of $H$ into $\Gcore \setminus (\Gcoreb \cup \Gcoreh)$. Then for every $a \in V(H)$,
  \[
    \deg_{\Gcore} \varphi(a) \ge \frac{m^{1/2}}{\big(\log(1/p)\big)^{\beta \cdot
    v_H}}.
  \]
\end{claim}
\begin{proof}
  Define $f \colon V(H) \to \RR$ by
  \[
    f(a) = \log \left(\frac{\deg_{\Gcore}\varphi(a)}{m^{1/2}}\right)
  \]
  and let
  \[
    f^* = \beta \log\log(1/p).
  \]
  It suffices to show that $f(a) \ge -v_H f^*$ for every $a \in V(H)$. To
  this end, note that~\eqref{eq:degree-product-lower} and our definition of
  $\Gcoreh$ (see~\eqref{eq:degree-product-upper}) imply that
  \begin{equation}
    \label{eq:h-sum}
    -f^* \le f(a) + f(b) \le f^* \qquad \text{for every $ab \in E(H)$},
  \end{equation}
  since $\beta \geq v_H + 1$ and $p \le C^{-2/\Delta}$ for a large constant $C$.
  Since $H$ is not bipartite, it contains an odd cycle. Let $Z$ be one such
  cycle and suppose that $a_0, \dotsc, a_{2\ell}$ are its vertices (listed in
  an arbitrarily chosen cyclic ordering). 
  It follows from~\eqref{eq:h-sum},
  applied to all $2\ell+1$ edges of $Z$, that
  \[
    2f(a_0) = f(a_0) + f(a_{2\ell}) + \sum_{i=0}^{2\ell-1}
    (-1)^i\big(f(a_i) + f(a_{i+1})\big)
    \in \big[-(2\ell+1)f^*, (2\ell+1)f^*\big].
  \]
  Since the particular choice of $a_0$ among all vertices of $Z$ was
  arbitrary, we may conclude that
  \[
    -(2\ell+1)f^* \le f(a) \le (2\ell+1)f^* \qquad\text{for every $a
    \in V(Z)$},
  \]
  with room to spare.
  Since $2\ell+1 = v_Z\leq v_H$, this proves the desired inequality for all
  $a\in V(Z)$.
  Suppose now that $b \in V(H) \setminus V(Z)$. Since $H$ is connected, it
  contains a path from $b$ to $Z$. Let $b_0, b_1, \dotsc, b_{\ell'}$, where
  $b_0 = b$ and $b_{\ell'} \in V(Z)$, be the vertices of a shortest such path
  (listed in their natural order) and note that $\ell' + 2\ell+1 \le v_H$. It
  follows from~\eqref{eq:h-sum}, applied to all $\ell'$ edges of the path,
  that
  \[
    f(b) + (-1)^{\ell'-1}f(b_{\ell'}) = \sum_{i=0}^{\ell'-1} (-1)^i
    \big(f(b_i) + f(b_{i+1})\big) \in \big[-\ell' f^*, \ell' f^*\big]
  \]
  and consequently, as $b_{\ell'} \in V(Z)$, that
  \[ f(b) \geq -\ell' f^* - |f(b_\ell')| \ge -(\ell'+2\ell+1)f^* \ge -v_Hf^*, \]
  as claimed.
\end{proof}

Let $\Gcoree$ comprise all edges of $\Gcore$ that do not belong to a copy of
$H$ in the graph $\Gcore \setminus (\Gcoreb \cup \Gcoreh)$ and let $\me$ be the
number of such edges; note that $\Gcoree \supseteq \Gcoreb \cup \Gcoreh$. Since
each edge of $\Gcoree \setminus \Gcoreb$ belongs to at least
$\gamma'n^{v_H}p^{e_H} / \mmax$ copies of $H$ in $\Gcore$, none of which are in
$\Gcore\setminus \Gcoreb\cup \Gcoreh$, we have
\[
  (\me - \mb) \cdot \frac{\gamma'n^{v_H}p^{e_H}}{\mmax} \le 
  \sum_{e\in 
  E(\Gcoreb \cup \Gcoreh)}
  |\Emb(H,\Gcore; e)|.
\]
It follows from Lemma~\ref{lemma:max-copies-bad-edges}, Claim~\ref{claim:mb-upper}, and inequality~\eqref{eq:mh-upper} that
\[
  \sum_{e\in E(\Gcoreb \cup \Gcoreh)}|\Emb(H,\Gcore;e)| \le e_H \cdot (2m)^{v_H/2}
  \cdot \left(\frac{\mb + \mh}{m}\right)^{1/\Delta} \le 8e_H \cdot
  \frac{(2m)^{v_H/2}}{\big(\log(1/p)\big)^{\beta/\Delta}}.
\]
Consequently,
\[
  \begin{split}
    \me & \le
    \mb + \frac{\mmax}{\gamma'n^{v_H}p^{e_H}} \cdot 
    8e_H \cdot
    \frac{(2m)^{v_H/2}}{\big(\log(1/p)\big)^{\beta/\Delta}}\\
    &\leq
    p^\sigma\cdot \mmin+
    \frac{8\cdot 2^{v_H/2}\cdot e_H}{\gamma'} \cdot
    \frac{\mmax^{v_H/2}}{n^{v_H}p^{e_H}\big(\log(1/p)\big)^{\beta/\Delta}}
    \cdot m \\
    & \le \left(p^\sigma + 
    \frac{8\cdot 2^{v_H/2}\cdot e_H}{\gamma'}\cdot \frac{\big(K'\cdot n^2p^\Delta
    \log(1/p)\big)^{v_H/2}}{n^{v_H}p^{e_H}\big(\log(1/p)\big)^{\beta/\Delta}}
    \right) \cdot m\\
    & = \left(p^\sigma + \frac{8\cdot e_H\cdot \big(2K'
    \big)^{v_H/2}}{\gamma'\cdot \big(\log(1/p)\big)^{\beta/\Delta-v_H/2}}
    \right) \cdot m
    \le \frac{m}{\log(1/p)},
  \end{split}
\]
since $\beta/\Delta = 1+ v_H/2 + 1/(2\Delta)$ and $p \le C^{-2/\Delta}$ for a large constant $C$.

Let $U$ be the set of nonisolated vertices of $\Gcore \setminus \Gcoree$. It follows from Claim~\ref{cl:Hexc} that
\[
  |U| \cdot \frac{m^{1/2}}{\big(\log(1/p)\big)^{\beta \cdot v_H}} \le \sum_{v \in U} \deg_{\Gcore}v \le 2m
\]
and thus
\[
  |U| \le 2m^{1/2} \cdot\big(\log(1/p)\big)^{\beta\cdot v_H} \le m/\log n,
\]
as $p^{\Delta/2}\geq Cn^{-1}(\log n)^{\Delta v_H^2}$ and $C$ is large enough.

To summarise, we may construct each $\Gcore \in \core_m$ as follows:
\begin{enumerate}[{label=(\arabic*)}]
\item
  Choose $\me \le m/\log(1/p)$ and the $\me$ edges of $K_n$ to form $\Gcoree$.
\item
  Choose the vertices of $U$ and the edges of $\Gcore \setminus \Gcoree$ from the set $\binom{U}{2}$.
\end{enumerate}
Using the above bounds on the size of $U$, we conclude that
\[
  |\core_m| \le  \sum_{\me=0}^{m/\log(1/p)} \binom{n^2}{\me}
  \cdot n^{m/\log n}
  \cdot
  \binom{4m\cdot\big(\log(1/p)\big)^{2\beta\cdot v_H}}{m-\me}
  .
\]
In order to bound the right-hand side from above, we note that,
for sufficiently large $C$,
\[
  \binom{4m\cdot\big(\log(1/p)\big)^{2\beta\cdot v_H}}{m-\me} \le \binom{4m\cdot\big(\log(1/p)\big)^{2\beta\cdot v_H}}{m} \le \left(\frac{4em\cdot\big(\log(1/p)\big)^{2\beta\cdot v_H}}{m}\right)^m \le p^{-\eps m/6}
\]
and, since $m \ge n^2p^\Delta / K'$,
\[
  \sum_{\me=0}^{m/\log(1/p)} \binom{n^2}{\me} \le
  \left(\frac{en^2\log(1/p)}{m}\right)^{m/\log(1/p)} \le
  \left(\frac{eK'\cdot\log(1/p)}{p^\Delta}\right)^{m/\log(1/p)}  \le p^{-\eps
  m /6}.
\]
Since $n^{m/\log n} = e^m \leq p^{-\eps m/6}$, we may conclude that
$|\core_m| \le p^{-\eps m/2}$, which completes the proof of
Proposition~\ref{prop:ldp-H}.

\section{The Poisson regime}\label{sec:poisson}

Given a nonnegative real $\mu$, we shall denote by $\Po(\mu)$ the Poisson distribution with mean $\mu$.
Suppose that $X \sim \Po(\mu)$. A classical result in large deviation theory is that, for every fixed $\delta > 0$,
\[
  - \log \Pr\big(X \ge (1+\delta) \mu \big) = \big((1+\delta)\log(1+\delta) - \delta\big) \mu + o(\mu),
\]
as $\mu \to \infty$. Motivated by this estimate, for any random variable $X$ with positive expectation and any $\delta>0$, we define
\[
  \Psi_X(\delta) = \big((1+\delta)\log(1+\delta)-\delta\big)\e[X].
\]
Theorems~\ref{thm:kappoisson} and \ref{thm:Hpoisson} follow immediately from the following two propositions.

\begin{proposition}\label{prop:kappoisson}
  For every integer $k\geq 3$ and all positive real numbers $\eps$ and
  $\delta$, there exists a positive constant $C$ such that the following
  holds. Suppose that $N\in \NN$ and $p\in
  (0,1)$ satisfy $C N^{-1} \leq p^{k/2} \leq C^{-1} N^{-1} \log N$.
  Then 
  $X = X_{N,p}^\kap$ satisfies
  \[ (1-\eps)\Psi_X(\delta) \leq -\log \Pr\big(X\geq (1+\delta)\e[X]\big) \leq (1+\eps)\Psi_X(\delta). \]
\end{proposition}

\begin{proposition}\label{prop:Hpoisson}
  For every $\Delta\geq 2$, every connected, $\Delta$-regular graph $H$, and all positive
  real numbers $\eps$ and $\delta$,
  there exists a positive constant $C$ such that the following holds.
  Suppose that $n\in \NN$ and $p\in (0,1)$ satisfy $C n^{-1} \leq p^{\Delta/2} \leq
  C^{-1} n^{-1} (\log n)^{\frac1{v_H-2}}$. Then $X = X_{n,p}^H$ satisfies
  \[ (1-\eps)\Psi_X(\delta) \leq -\log \Pr\big(X\geq (1+\delta)\e[X]\big) \leq (1+\eps)\Psi_X(\delta). \]
\end{proposition}

It is not difficult to show that the requirements on $p$ in these results
are optimal up to the choice of the constant $C$. 
Indeed, if $X = X_{N,p}^\kap$, then planting an interval of length
$C_\delta Np^{k/2}$ creates $(1+\delta)\e[X]$ $k$-term arithmetic progressions
(for a sufficiently large $C_\delta$),
which shows that
\[ - \log
 \Pr\big(X \geq (1+\delta)\e[X]\big) = O\big(
N p^{k/2} \log (1/p)\big). \] 
Similarly, if $X = X_{n,p}^{H}$, then
planting a clique of size $C_\delta np^{\Delta/2}$ 
results in $(1+\delta)\e[X]$ copies of~$H$, which proves
\[ -\log \Pr\big(X \geq (1+\delta)\e[X]\big) = O\big(
n^2 p^{\Delta} \log (1/p)\big). \]
The upper bounds are $o(\Psi_X(\delta))$, and therefore dominate the Poisson bounds, whenever 
$p^{k/2}\gg N^{-1}\log N$ or
$p^{\Delta/2}\gg n^{-1}(\log n)^{\frac1{v_H-2}}$, respectively.

\subsection{Poisson approximation via factorial moments}

For a real number $m$ and a nonnegative integer $t$, write $\ff{m}{t} = m (m-1) \dotsb (m-t+1)$ for the
$t$-th falling factorial of $m$. For a random variable $X$, let $M_t(X) = \e[\ff{X}{t}]$
be the $t$-th factorial moment of $X$. It is straightforward to verify that, if $X\sim \Po(\mu)$,
then $M_t(X) = \mu^t$ for all $t \geq 0$.

A classical application of the method of moments is that, if $(X_n)$ is a sequence of random variables
whose $t$-th factorial moments converge to $\mu^t$ for some fixed $\mu$, then $X_n$ converges
in distribution to $\Po(\mu)$. The lemma below can be viewed as an extension of this result to the case when
$\mu \to \infty$. It states that, if the $t$-th factorial moment of some random variable $X$ is approximately $\mu^t$,
for each $t$ around $\delta \mu$, then the logarithmic upper tail probability $- \log \Pr\big(X \ge (1+\delta)\Ex[X]\big)$
is well approximated by $\Psi_X(\delta)$.

\begin{proposition}\label{prop:poisson}
  For all positive real numbers $\eps$ and $\delta$, there exists a positive constant $\eta$ such
  that the following holds. Let $X$ be a nonnegative integer-valued random
  variable with mean $\mu \geq 1/\eta$ such that $|M_t(X) - \mu^t|\leq \eta \mu^t$
  for every integer $t$ satisfying $(\delta-\eps)\mu\leq t\leq (\delta+\eps)\mu$.
  Then
  \[ 
    (1-\eps) \Psi_X(\delta) \leq -\log \Pr\big(X\geq (1+\delta)\e[X]\big) \leq (1+\eps) \Psi_X(\delta).
  \]
\end{proposition}

Define the continuous function $I \colon [0,\infty) \to [0, \infty)$ by
\[
  I(\delta) = (1+\delta) \log (1+\delta) - \delta,
\]
so that $\Psi_X(\delta) = I(\delta) \cdot \mu$. Note that $I(\delta) > 0$ whenever $\delta > 0$.

\begin{lemma}
  \label{lemma:ff}
  For every nonnegative integer $t$ and every positive real $x$,
  \[
    \log \ff{(x+t)}{t} = I(t/x) \cdot x + t \log x + \lambda(x,t),
  \]
  where $0 \le \lambda(x,t) \le (t+1)/x$.
\end{lemma}
\begin{proof}
  Observe that $\log \ff{(x+t)}{t} = \sum_{s=1}^t \log(x+s)$. Since $\log$ is an increasing function,
  \[
    \int_{x}^{x+t} \log y \, dy \le \sum_{s=1}^{t} \log (x+s) \le \int_{x+1}^{x+t+1} \log y \, dy.
  \]
  Recalling that $\int \log y \, dy = y(\log y - 1) + C$, we have
  \[
    \int_x^{x+t} \log y \, dy = \big((1+t/x) \log (1+t/x) - t/x\big) \cdot x + t \log x = I(t/x) \cdot x+ t \log x.
  \]
  On the other hand,
  \[
    \int_{x+1}^{x+t+1} \log y \, dy - \int_x^{x+t} \log y \, dy \le \log (x+t+1) - \log x = \log\big(1+(t+1)/x\big) \le (t+1)/x.
  \]
  This proves the claimed estimate.  
\end{proof}

\begin{proof}[{Proof of Proposition~\ref{prop:poisson}}]
  Since, for every positive integer $t$, the function $x \mapsto \ff{x}{t}$ is increasing on $[t-1, \infty)$ and nonnegative on $\ZZ_{\ge 0}$, Markov's inequality implies that
  \[
    \Pr\big( X \ge (1+\delta) \Ex[X] \big) \le \Pr\left(\ff{X}{t} \ge \ff{\big((1+\delta)\Ex[X]\big)}{t}\right) \le \frac{M_t(X)}{\ff{\big((1+\delta)\Ex[X]\big)}{t}}
  \]
  for every positive integer $t \le (1+\delta)\Ex[X]$. This implies that, for every such $t$,
  \begin{equation}
    \label{eq:Poisson-Markov}
    -\log \Pr\big(X \geq (1+\delta)\Ex[X]\big) \geq \log \ff{\big((1+\delta)\mu\big)}{t} - \log M_t(X).
  \end{equation}
  Let $t = \floor{\delta \mu}$. Since $(1+\delta) \mu - t \ge \mu$, it follows from Lemma~\ref{lemma:ff} that
  \[
    \log \ff{\big((1+\delta)\mu\big)}{t} \ge \log \ff{(\mu + t)}{t} \ge I(t/\mu) \cdot \mu + t \log \mu.
  \]
  On the other hand, our assumption implies that
  \[
    \log M_t(X) \le \log \big((1+\eta) \mu^t\big) = t \log \mu + \log (1+\eta) \le t \log \mu + \eta.
  \]
  Finally, since $I$ is continuous, $I(\delta) > 0$, and $|t/\mu - \delta| \le 1/\mu \le \eta$, substituting the above two inequalities into~\eqref{eq:Poisson-Markov} yields
  \[
    -\log \Pr\big(X \geq (1+\delta)\Ex[X]\big) \geq I(\delta) \cdot \mu - \eps \Psi_X(\delta) = (1-\eps) \Psi_X(\delta),
  \]
  provided that $\eta$ is sufficiently small (as a function of $\eps$ and $\delta$).
  
  For the upper bound, we will use the tilting argument, which is a
  standard trick of large deviation theory.
  For the sake of brevity, let $m_t = M_t(X)$ for every nonnegative integer $t$.
  Let $t = (\delta + \gamma)\mu$, where $\gamma$ is a small positive constant that depends on $\eps$ and $\delta$ (but not on $\eta$).
  Since $\eta$ is allowed to depend on $\gamma$ and $\mu \ge 1/\eta$, we may assume that $t$ is an integer.
  The idea is to consider a `tilted' random variable $\tilde X$ defined by the relation
  \[
    \Pr(\tilde X = x) = \frac{\Pr(X = x) \cdot \ff{x}{t}}{m_t} \qquad \text{for every $x \in \ZZ$}.
  \]
  The definition of $m_t$ ensures that $\tilde X$ is a well-normalised random variable. In particular,
  \begin{equation}\label{eq:rnd}
    \e[g(\tilde X)] = \frac{\e[g(X) \cdot \ff{X}{t}]}{m_t} \qquad \text{for every $g \colon \ZZ \to \ZZ$}.
  \end{equation}

  Using the identities 
  \[
    x \cdot \ff{x}{t} = \ff{x}{t+1} + t  \cdot \ff{x}{t} \qquad \text{and} \qquad x^2 \cdot \ff{x}{t} = \ff{x}{t+2} + (2t+1)  \cdot \ff{x}{t+1} + t^2 \cdot \ff{x}{t},
  \]
  we have
  \[
    \e[\tilde X] = \frac{\e[X \cdot \ff{X}{t}]}{m_t} =  \frac{m_{t+1}}{m_t} + t
  \]
  and
  \[
    \e[\tilde X^2] = \frac{\e[X^2 \cdot \ff{X}{t}]}{m_t} = \frac{m_{t+2} + (2t+1)m_{t+1}}{m_t}  + t^2,
  \]
  and so
  \[
    \Var[\tilde X] = \e[\tilde X^2] - \e[\tilde X]^2 = \frac{m_{t+2}m_t - m_{t+1}^2 + m_{t+1}m_t}{m_t^2}.
  \]
  Since $t=(\delta+\gamma)\mu$, we have $(\delta-\eps)\mu \leq t\leq t+2\leq (\delta+\eps)\mu$,
  provided that $\gamma$ is sufficiently small as a function of $\delta$ and $\eps$. Since the assumptions of the proposition
  imply that $m_s$ is well approximated by $\mu^s$ for each $s \in \{t, t+1, t+2\}$, a straightforward computation yields
  \[ (1+\delta+\gamma/2)\mu\leq \e[\tilde X]
  \leq (1+\delta+3\gamma/2)\mu \qquad
  \text{and}
  \qquad\Var[\tilde X]
    \leq 10\eta\mu^2 + 2\mu,
  \]
  provided that $\eta$ is sufficiently small. Therefore, Chebyshev's inequality yields
  \begin{equation}
    \label{eq:chebyshev}
    \Pr\big( |\tilde X - (1 + \delta + \gamma) \mu| \geq
    \gamma\mu\big)
    \leq
    \Pr\big( |\tilde X - \e[\tilde X]| \geq \gamma\mu/2\big)
    \leq \frac{4 \cdot (10\eta\mu^2+2\mu)}{\gamma^2\mu^2}
    \leq \eps.
  \end{equation}

  Next, using \eqref{eq:rnd} with the function
  $g(x) = 
  \1\big[|x - (1+\delta+\gamma)\mu| < \gamma \mu\big] \cdot (m_t/\ff{x}{t})$, we see that
  \[
    \Pr\big(X \geq (1 + \delta) \mu\big) \geq \Pr\big( |X - (1+\delta+\gamma)\mu|\leq\gamma\mu\big) = \Ex\big[g(\tilde X)\big].
  \]
  When $g(\tilde X)$ is nonzero, then $\tilde X$ is bounded from above by $(1+\delta+2\gamma)\mu$, and thus
  \begin{align*}
    \Ex\big[g(\tilde X)\big] & \ge \Pr\big(g(\tilde X) \neq 0\big) \cdot \frac{m_t}{\ff{(1+\delta+2\gamma)}{t}} \\
                             & =\big(1-\Pr( |\tilde X - (1 + \delta+\gamma) \mu | \geq \gamma\mu)\big) \cdot \frac{m_t}{\ff{(1+\delta+2\gamma)}{t}}
                            \stackrel{\eqref{eq:chebyshev}}{\geq} \frac{(1-\eps) \cdot m_t}{\ff{(1+\delta+2\gamma)}{t}}.
  \end{align*}
  Using Lemma~\ref{lemma:ff} with $x = (1+\gamma)\mu$, we obtain
  \[
    \log \ff{(1+\delta+2\gamma)}{t} \le I\big(t / (1+\gamma)\mu\big) \cdot (1+\gamma)\mu + t \log \big((1+\gamma)\mu\big) +  \frac{t+1}{(1+\gamma)\mu}.
  \]
  On the other hand, our assumptions imply that, when $\eta$ is small,
  \[
    \log m_t \ge t \log \mu + \log (1-\eta) \ge t \log \mu - 2\eta.
  \]
  Combining the above bounds, we obtain
  \[
    -\log \Pr\big(X \ge (1+\delta)\mu\big) \le  I\big(t / (1+\gamma)\mu\big) \cdot (1+\gamma)\mu + t \log (1+\gamma) + \frac{t+1}{(1+\gamma)\mu} + 2\eta - \log(1-\eps).
  \]
  Recalling that $t = (\delta + \gamma) \mu$, the continuity of $I$ and the fact that $I(\delta) > 0$ imply that
  \[
    -\log \Pr\big(X \ge (1+\delta)\mu\big) \le I(\delta) \cdot \mu + \eps \Psi_X(\delta) = (1+\eps) \Psi_X(\delta),
  \]
  provided that $\gamma$ is sufficiently small and $\eta$ is sufficiently small.
\end{proof}

\subsection{Cluster analysis}

We will deduce both Propositions~\ref{prop:kappoisson} and~\ref{prop:Hpoisson}
from Proposition~\ref{prop:poisson} and Lemma~\ref{lem:poissonhg}, stated below,
by analysing the component structure of certain random hypergraphs.
Since the proofs turn out to be quite similar, we adopt a general point of view from the start. Suppose that
$\cH$ is a hypergraph and, given some $p\in (0,1)$, denote by $\cH_p$ the
random induced subhypergraph of $\cH$ obtained by keeping every
vertex with probability $p$, independently.
The \emph{dependency graph} $G_\cH$ is the graph on the vertex set
$E(\cH)$ whose edges are all pairs $\{\sigma_1, \sigma_2\}$ such that
$\sigma_1\cap \sigma_2\neq \emptyset$. A \emph{cluster} is a set $E'\subseteq E(\cH)$
that induces a connected subgraph in $G_\cH$.
We write $D_s(\cH_p)$ for the number of clusters of size $s$ whose elements are edges in $\cH_p$.

\begin{lemma}\label{lem:poissonhg}
  For all positive real numbers $c$ and $\eta$, there exists a positive 
  constant $K$ such that the following holds.
  Let $\cH$ be a uniform hypergraph, let $p\in (0,1)$, and define
  $X = e_{\cH_p}$ and $\mu = \e[X]$. Assume that $K\leq \mu \leq
  \sqrt{e_\cH}/K$ and that $\e[D_s(\cH_p)]\leq \exp(-Ks)$
  for every integer $s$ such that $2\leq s\leq c\mu$. 
  Then $|M_t(X) - \mu^t|\leq \eta\mu^t$ for every integer $t$ such that $1 \leq t\leq c\mu$.
\end{lemma}

\begin{proof}
  Let $t \le c\mu$ be a positive integer and let $\cH(t)$ denote the set of all sequences of $t$
  distinct edges of $\cH$. For each sequence $\sigmab = (\sigma_1,\dotsc,\sigma_t)
  \in \cH(t)$, let $X_\sigmab$ be the indicator random variable for the event 
  $\sigma_1 \cup \dotsb \cup \sigma_t \subseteq V(\cH_p)$. Denote the uniformity of $\cH$ by $k$.
  Our definitions readily imply that $\ff{X}{t} = \sum_{\sigmab\in \cH(t)}
  X_{\sigmab}$, that $|\cH(t)| = \ff{e_{\cH}}{t}$, and that $\e[X_\sigmab]\geq p^{kt}$ for all $\sigmab\in \cH(t)$.
  Thus
  \[
    M_t(X) =\e[\ff{X}{t}]= \sum_{\sigmab\in \cH(t)} \e[X_\sigmab] \geq \ff{e_{\cH}}{t} \cdot p^{kt}.
  \]
  Since, for every $x \ge t$,
  \[
    \ff{x}{t} = x^t \prod_{s=0}^{t-1} \left(1-\frac{s}{x}\right) \geq x^t \left(1-\sum_{s=0}^{t-1} \frac{s}{x}\right) \ge x^t \left(1 - \frac{t^2}{x}\right),
  \]
  and $t \le c\mu \le c\sqrt{e_{\cH}}/K \le e_\cH$ for sufficiently large $K$, we have
  \[
    M_t(X) \geq \left( 1- \frac{t^2}{e_{\cH}}\right) \cdot {e_{\cH}}^tp^{kt}
    = \left(1-\frac{t^2}{e_{\cH}}\right) \cdot \mu^t \geq \left(1-\frac{c^2}{K^2}\right) \cdot \mu^t \geq (1-\eta) \cdot \mu^t,
  \]
  provided that $K$ is sufficiently large.
  
  It remains to prove the upper bound. It will be convenient to partition the set $\cH(t)$ of sequences according to the component structure of the subgraph
  of $G_\cH$ induced by the elements of the sequence. More precisely, given a nonnegative integer $\ell$ and
  integers $s_1, \dotsc, s_\ell$ such that $2 \le s_1 \le \dotsb \le s_\ell$, let $\cH(t; s_1, \dotsc, s_\ell)$ be the family of all $\sigmab = (\sigma_1, \dotsc, \sigma_t) \in \cH(t)$
  such that the set $\{\sigma_1, \dotsc, \sigma_t\}$ induces a subgraph in $G_\cH$ whose $\ell$ nontrivial connected components (maximal clusters) have
  sizes $s_1, \dotsc, s_\ell$, so that this graph has $t - (s_1 + \dotsb + s_\ell)$ isolated vertices.\footnote{This includes the case $\ell = 0$ in which $\cH(t; \emptyset)$ corresponds to induced subgraphs of $G_\cH$ all of whose connected components are isolated vertices.} Observe that, for every collection $W_1, \dotsc, W_\ell$
  of connected subsets of vertices $G_\cH$ with sizes $s_1, \dotsc, s_\ell$, respectively, there are at most $t^{s_1 + \dotsb + s_\ell} \cdot {e_\cH}^{t-(s_1 + \dotsb + s_\ell)}$
  sequences $\sigmab = (\sigma_1, \dotsc, \sigma_t) \in \cH(t)$ such that the nontrivial connected components of $\{\sigma_1, \dotsc, \sigma_t\}$ are
  exactly $W_1, \dotsc, W_\ell$; indeed, there are at most $t^{s_1 + \dotsb + s_\ell}$ ways to choose the locations of the vertices in $W_1 \cup \dotsb \cup W_\ell$
  in a sequence of length $t$ and, for each such choice, at most ${e_\cH}^{t-(s_1 + \dotsb + s_\ell)}$ choices for the remaining elements of the sequence.
  We conclude that
  \[
    \sum_{\sigmab \in \cH(t;s_1, \dotsc, s_\ell)} \Ex[X_\sigmab] \le \mu^{t-(s_1+\dotsb+s_\ell)} \cdot \prod_{i=1}^{\ell} \Ex\big[D_{s_i}(\cH_p)\big] \cdot t^{s_i}
  \]
  and, consequently, summing over all $\ell$ and all sequences $s_1, \dotsc, s_\ell$ and using the assumed upper bound on the expectation of $D_s(\cH_p)$, valid for each $s \le t$,
  \[
    \begin{split}
      M_t(X) & \leq \sum_{s=0}^{t}\mu^{t-s} \cdot \sum_{\ell \ge 0} \sum_{\substack{s_1+\dotsb+s_\ell=s\\ 2 \le s_1 \le \dotsb \le s_\ell}} \prod_{i=1}^{\ell} \e\big[D_{s_i}(\cH_p)\big] \cdot t^{s_i} \\
      & \leq \sum_{s=0}^{t}\mu^{t-s} \cdot \sum_{\ell \ge 0} \sum_{\substack{s_1+\dotsb+s_\ell=s\\ 2 \le s_1 \le \dotsb \le s_\ell}} \prod_{i=1}^{\ell} \exp(-K s_i+s_i \log t)\\
      & = \sum_{s=0}^{t}\mu^{t} \cdot \sum_{\ell \ge 0} \sum_{\substack{s_1+\dotsb+s_\ell=s\\ 2 \le s_1 \le \dotsb \le s_\ell}}     \exp\big(-K s + s\log (t/\mu)\big).
    \end{split}
  \]
  Since, for every $s \ge 0$, there are at most $2^s$ sequences $s_1, \dotsc, s_\ell$ of positive integers whose sum is $s$ (this includes the case when $s=0$, when the only such sequence is the empty sequence), we have
  \[
    M_t(X) \le \mu^t \cdot \sum_{s=0}^t \exp\big(-Ks+s\log(t/\mu)+s\big).
  \]
  Finally, since $t/\mu \leq c$, we may choose $K =  K(c,\eta)$ so that
  \[
    M_t(X) \leq \mu^t \cdot \sum_{s=0}^t \left(\frac{\eta}{1+\eta}\right)^s \leq (1+\eta)\mu^t,
  \]
  completing the proof.
\end{proof}

\subsection{Proof of Proposition~\ref{prop:kappoisson}}

Let $\cH$ be the hypergraph on the vertex set $\br N$ whose edges are $k$-term arithmetic progressions in $\br N$,
so that $X= X_{N,p}^\kap = e_{\cH_p}$. Let $\mu = \Ex[X]$, let $\eta = \eta(\eps,\delta)$ be the constant from
the statement of Proposition~\ref{prop:poisson},
and let $K = K(\eps, \delta, \eta)$ be the constant from the statement of Lemma~\ref{lem:poissonhg}.

For any two integers $a, b \in \br{N}$ with $a < b$, there is at most one $k$-term arithmetic progression that starts
with $a$ and ends with $b$ (and exactly one such progression if $b-a$ is divisible by $k-1$). Therefore, $N^2/(2k) \le e_{\cH} \le N^2$
for all large enough $N$. In particular, since we assume that $CN^{-1}\leq p^{k/2} \leq C^{-1}N^{-1}\log N$, we find that
\begin{equation}
  \label{eq:Poisson-kap-mu}
  \frac{C^2}{2k} \leq \mu \leq \left(\frac{\log N}{C}\right)^2
\end{equation}
and thus $\max\{1/\eta, K\} \leq \mu \leq \sqrt{e_{\cH}} / K$ whenever $C$ is large.
The claimed estimate on $- \log \Pr\big(X \ge (1+\delta)\Ex[X] \big)$ will follow from Proposition~\ref{prop:poisson}
and Lemma~\ref{lem:poissonhg} once we verify that $D_s(\cH_p)$, the number of clusters of $s$ arithmetic progressions
of length $k$ in the set $\br{N}_p$, satisfies
\[
  \Ex[D_s(\cH_p)] \le \exp(-Ks)
\]
for every $s$ satisfying $2 \le s \le (\delta +\eps)\mu$.

In order to do so, let $\cD(s,m)$ be the the set of all
clusters $\{\sigma_1,\dotsc, \sigma_s\}$ of $s$ arithmetic progressions of length $k$ in $\br{N}$
such that $|\sigma_1\cup \dotsb \cup \sigma_s| = m$; we also let $D_{s,m}$ be the number of such clusters
whose union is contained in the random set $\br N_p$.
When $s \ge 2$, the union of any $s$ distinct $k$-term arithmetic progressions contains
between $k+1$ and $ks$ numbers, and therefore $D_{s,m} = 0$ unless $k+1 \le m \le ks$. Thus, we can write
\[
  D_s(\cH_p) = \sum_{m=k+1}^{ks} D_{s,m}.
\]

For each integer $m$, let $a_m$ denote the number of $m$-element subsets of $\br{N}$ that are the union
of a single, nonempty (but possibly trivial) cluster of $k$-term arithmetic progressions. Since a progression
is uniquely determined by its first and second element, it follows that, for each $s$,
\begin{equation}
  \label{eq:apcore}
  \Ex[D_{s,m}] \leq a_m p^m \binom{m^2}{s}.
\end{equation}

\begin{claim}
  \label{claim:am-upper}
  For every integer $m \ge 1$,
  \[
    a_m \leq N^2 \cdot (2kmN)^{\frac{m-k}{k-1}}. 
  \]
\end{claim}

\begin{proof}
  We prove the claimed upper bound on $a_m$ by induction on $m$. It is vacuously true when $m < k$, since then $a_m = 0$,
  or when $m = k$, as $a_k = e_{\cH} \le N^2$. Assume now that $m \ge k+1$ and let $A$ be an arbitrary set counted by $a_m$.
  Since $m > k$, the set $A$ must be a union of at least two different progressions. Moreover, there are a proper subset $A' \subsetneq A$
  that is a union of a (nonempty) cluster of $k$-term arithmetic progressions and a $k$-term progression $\sigma$
  that intersects $A'$ such that $A = A' \cup \sigma$; note that the number of
  $\sigma_i$'s whose union is $A'$ may be significantly smaller than the number
  that was used to generate $A$. By
  construction, we have that $|A'| = |A| - k + |A' \cap \sigma| = m-k+|A' \cap
  \sigma|$. Since there are at most $k |A'| N \le kmN$ arithmetic progressions
  of length $k$ that intersect $A'$ in exactly one element and
  at most $k^2 |A'|^2 \le k^2m^2$ progressions that intersect $A'$ in two or more elements,
  \[
    a_m \le kmN \cdot a_{m-k+1} + k^2m^2 \cdot (a_{m-k+2} + \dotsb + a_{m-1}).
  \]
  It follows from our inductive assumption that
  \[
    \begin{split} a_m &\leq 
      km N\cdot N^2 \cdot (2km N)^{\frac{m-2k+1}{k-1}} + k^2m^2 \cdot k \cdot N^2 \cdot (2kmN)^{\frac{m-k-1}{k-1}}\\
      &=
      N^2\cdot \left((2km N)^{\frac{m-k}{k-1}} / 2 + k^3m^2 \cdot (2km N)^{\frac{m-k}{k-1}-\frac{1}{k-1}}\right).
    \end{split}
  \]
  Finally, as~\eqref{eq:Poisson-kap-mu} implies that $m \leq ks\leq k(\delta+\eps)\mu \leq k(\delta +\eps)(C^{-1}\log N)^2$ and $N \ge C$, then
  \[
    k^3m^2 \cdot (2kmN)^{-1/(k-1)} \le 1/2,
  \]
  provided that $C$ is sufficiently large. This implies the claimed upper bound on $a_m$.
\end{proof}

Assume now that $k+1 \le m \le ks$. Since $s \le (\delta+\eps)\mu$, inequality~\eqref{eq:Poisson-kap-mu} implies that $m$ is only
polylogarithmic in $N$; on the other hand, $p\leq (C^{-1}N^{-1}\log N)^{2/k}$. Since $k \ge 3$ and $N \ge C$, there is a positive
constant $\gamma$ that depends only on $k$ such that
\[ 
  (2km N)^{1/(k-1)} p \leq (2km N)^{1/(k-1)} \cdot (C^{-1}N^{-1}\log N)^{2/k} \le N^{-2(k+1)\gamma}  \le N^{-2m\gamma/ (m-k)}.
\]
In particular, Claim~\ref{claim:am-upper} implies that
\[
  a_m p^m \leq N^2p^k \cdot \left((2km N)^{1/(k-1)}p\right)^{m-k}
  \leq N^2p^k \cdot N^{-2m\gamma} \leq N^{-m\gamma},
  \]
  where for the last inequality we use that $N^2p^k$ is at most polylogarithmic in $N$ and $N \ge C$.
  Combining this bound with \eqref{eq:apcore} we conclude that
  \[
    \Ex[D_{s,m}] \le N^{-m\gamma}\cdot \binom{m^2}{s} \leq \exp\left(-m\gamma\log N + s\log\left( \frac{em^2}{s}\right)\right).
  \]
  Let $f \colon (0, \infty) \to (0, \infty)$ be the function defined by
  \[
    f(x) = \exp\left( - x \gamma \log N+ s\log\left(\frac{ex^2}{s}\right)\right) = \exp\big(-x\gamma \log N + 2s\log x + s\log(e/s)\big),
  \]
  so that $\Ex[D_{s,m}] \le f(m)$. Elementary calculus shows that $f$ is maximised at $x = 2s/(\gamma \log N)$. Therefore,
  \[
    \Ex[D_{s,m}] \le f\left(\frac{2s}{\gamma \log N}\right) = \exp\left( - 2s+ s\log\left(\frac{4es}{\gamma^2 (\log N)^2}\right)\right).
  \]
  Since our assumptions imply that, see~\eqref{eq:Poisson-kap-mu},
  \[
    \frac{s}{(\log N)^2} \le \frac{(\delta+\eps)\mu}{(\log N)^2} \le \frac{\delta+\eps}{C^2},
  \]
  we may conclude that $\Ex[D_{s,m}] \le \exp(-(K+k)s)$, provided that $C$ is sufficiently large.
  Therefore, if $C$ is sufficiently large,
  \[
      \e[D_s(\cH_p)]  = \sum_{m=k+1}^{ks}\e[D_{s,m}] \le ks \cdot \exp(-Ks-ks) \le \exp(-Ks).
  \]
  This completes the proof.

\subsection{Proof of Proposition~\ref{prop:Hpoisson}}
Let $H$ be a connected, $\Delta$-regular graph and let $\cH$ be the hypergraph on the vertex set $\binom{\br n}{2}$ 
whose edges are copies of $H$ in $K_n$, so that $X = X_{n,p}^H = e_{\cH_p}$. 
Let $\mu = \Ex[X]$, let $\eta = \eta(\eps,\delta)$ be the constant from
the statement of Proposition~\ref{prop:poisson},
and let $K = K(\eps, \delta, \eta)$ be the constant from the statement of Lemma~\ref{lem:poissonhg}.

Since $(n/v_H)^{v_H} \le \binom{n}{v_H} \le e_\cH \le n^{v_H}$ for all large enough $n$, our assumption
$Cn^{-1} \le p^{\Delta/2} \le C^{-1} n^{-1} (\log n)^{\frac{1}{v_H-2}}$ and the fact that $2e_H = \Delta v_H$ imply that
\begin{equation}
  \label{eq:Poisson-H-mu}
  \left(\frac{C}{v_H}\right)^{v_H} \le \mu \le \frac{(\log n)^{1+\frac{2}{v_H-2}}}{C^{v_H}},
\end{equation}
and thus $\max\{1/\eta, K\} \le \mu \le \sqrt{e_\cH}/K$ whenever $C$ is sufficiently large.
The claimed estimate on $- \log \Pr\big(X \ge (1+\delta)\Ex[X] \big)$ will follow from Proposition~\ref{prop:poisson}
and Lemma~\ref{lem:poissonhg} once we verify that $D_s(\cH_p)$, the number of clusters of $s$ copies of $H$
in the random graph $G_{n,p}$, satisfies
\[
  \Ex[D_s(\cH_p)] \le \exp(-Ks)
\]
for every $s$ satisfying $2 \le s \le (\delta +\eps)\mu$.

To this end, for every $s \ge 1$, every $k \ge 1$, and every $m \ge 1$, let $\cD(s,k,m)$ denote the set of all
clusters $\{\sigma_1,\dotsc, \sigma_s\}$ of $s$ distinct copies of $H$ in $K_n$ such that
the graph $\sigma_1\cup \dotsb \cup \sigma_s$ has $k$ vertices (of nonzero degree) and $m$ edges.
We further let $D_{s,k,m}$ denote the number of such clusters whose union is contained in $G_{n,p}$.
When $s \ge 2$, the union of any $s$ distinct copies of $H$ contains
between $v_H$ and $v_Hs$ vertices and between $e_H+1$ and $e_Hs$ edges,
and thus $D_{s,k,m} = 0$ unless $v_H \le k \le v_Hs$ and $e_H+1 \le m \le e_Hs$. We can therefore write
\[
  D_s(\cH_p) = \sum_{k=v_H}^{v_Hs} \sum_{m=e_H+1}^{e_Hs} D_{s,k,m}.
\]

\begin{claim}
  \label{claim:Dskm-upper}
  There exists a positive constant $\gamma$ such that, for every $s \ge 2$, every $k \ge v_H$, and every $m \ge e_H+1$,
  \[
    \Ex[D_{s,k,m}] \leq n^{-2\gamma m} \binom{k^2}{m} \binom{(2m)^{v_H/2}}{s}.
  \]
\end{claim}
\begin{proof}
  We first show that, for every $s \ge 1$, the set $\cD(s,k,m)$ is empty unless
  \begin{equation}
    \label{eq:edges-vertices-cluster}
    m - e_H \ge \left(\frac{\Delta}{2} + \frac{1}{2v_H}\right) \cdot (k - v_H).
  \end{equation}
  We prove this fact by induction on $s$. The case $s=1$ holds vacuously, as the set $\cD(1,k,m)$ is
  nonempty only when $k = v_H$ and $m = e_H$. Assume now that $s \ge 2$ and let
  $G$ be the union of copies of $H$ that form some cluster in $\cD(s,k,m)$. By
  definition, $G$ has $k$ vertices and $m$ edges. Furthermore,
  for some $s' < s$, $k' \leq k$ and $m' < m$, there exist a subgraph $G'
  \subseteq G$ and a copy $\sigma$ of $H$ in $K_n$ that intersects (the edge
  set of ) $G'$ such that $G'$ is the union of copies of $H$ that form some
  cluster in $\cD(s',k',m')$, and $G = G' \cup \sigma$. We note that $s'$ may
  be strictly smaller than $s -1$. 
 Let $J \subseteq H$
  be the subgraph of $H$ that is isomorphic to $\sigma \cap G'$, so that $m = m' + e_H - e_J$ and $k = k' + v_H - v_J$. It follows
  from the inductive assumption that
  \[
    \begin{split}
      m - e_H & = m' - e_J \ge \left(\frac{\Delta}{2} + \frac{1}{2v_H}\right) \cdot (k' - v_H) +e_H - e_J \\
      & = \left(\frac{\Delta}{2} + \frac{1}{2v_H}\right) \cdot (k - v_H) - \left(\frac{\Delta}{2} + \frac{1}{2v_H}\right) \cdot (v_H - v_J) + e_H - e_J.
    \end{split}
  \]
  We claim that the above inequality implies~\eqref{eq:edges-vertices-cluster}. This is obviously true when $J = H$. If $J$ is a proper
  subgraph of $H$, then $2e_J \le\Delta v_J - 1$, since $H$ is connected and $\Delta$-regular, and therefore
  \[
    e_H - e_J \ge \frac{\Delta v_H}{2} - \frac{\Delta v_J - 1}{2} = \left(\frac{\Delta}{2} + \frac{1}{2(v_H-v_J)}\right) \cdot (v_H - v_J)
    \ge \left(\frac{\Delta}{2} + \frac{1}{2v_H}\right) \cdot (v_H - v_J),
  \]
  which gives~\eqref{eq:edges-vertices-cluster}.
  
  To complete the proof of the claim, note that
  \[ 
    \e[D_{s,k,m}] \leq n^kp^m \cdot \binom{k^2}{m} \cdot \binom{N(H, k, m)}{s},
  \]
  where $N(H,k,m)$ denotes the largest number of copies of $H$ in a graph with $k$ vertices and $m$ edges.
  Inequality~\eqref{eq:edges-vertices-cluster} implies that
  \[
    n^kp^m = n^{v_H}p^{e_H} \cdot n^{k-v_H} p^{m-e_H} \le n^{v_H} p^{e_H} \cdot \left(n p^{\Delta/2+1/(2v_H)}\right)^{2v_H(m- e_H)/(\Delta v_H +1)}.
  \]
  Since $p^{\Delta/2} \le C^{-1}n^{-1}\log n$, the quantity $n^{v_H}p^{e_H}$ is only polylogarithmic in $n$ and
  $n p^{\Delta/2+1/(2v_H)} \le n^{-\gamma'}$ for some positive constant $\gamma'$. As $m \ge e_H+1$ and
  $n$ is large, then there is a positive constant $\gamma$ such that $n^k p^m \le n^{-2\gamma m}$.
  The claimed upper bound on $\Ex[D_{s,k,m}]$ now follows from Theorem~\ref{thm:max-copies}, which implies that
  $N(H,k,m) \le (2m)^{v_H/2}$.
\end{proof}

Claim~\ref{claim:Dskm-upper} and the inequality $\binom{a}{b} \le (ea/b)^b$ impy that
\[
  \begin{split}
    \e[D_{s,k,m}] & \le \exp\left(-2\gamma m \log n + m\log(ek^2/m)+ s\log\big(e(2m)^{v_H/2}/s\big) \right) \\
    & \le \exp\left(-\gamma m\log n + s\log\big(e(2m)^{v_H/2}/s\big) \right),
  \end{split}
\]
where the second inequality as $k\leq v_Hs \leq v_H(\delta+\eps)\mu$ and $\mu$ is at most polylogarithmic
in $n$, see~\eqref{eq:Poisson-H-mu}. Let $f \colon (0, \infty) \to (0, \infty)$ be the function defined by
  \[
    f(x) = \exp\left( -\gamma x \log n+ s\log\big(e(2x)^{v_H/2}/s\big)\right),
  \]
  so that $\Ex[D_{s,k,m}] \le f(m)$. Elementary calculus shows that $f$ is maximised at $x = sv_H/(2 \gamma \log n)$. Therefore,
  \[
    \Ex[D_{s,k,m}] \le f\left(\frac{sv_H}{2\gamma \log n}\right) = \exp\left( -\frac{sv_H}{2} + \frac{sv_H}{2} \log\left(\frac{e^{2/v_H}v_H s^{1-2/v_H}}{\gamma \log n}\right)\right).
  \]
  Since our assumptions imply that, see~\eqref{eq:Poisson-H-mu},
  \[
    \frac{s^{1-2/v_H}}{\log n}  \le \frac{\big((\delta+\eta)\mu\big)^{1-2/v_H}}{\log n} \le \frac{(\delta+\eta)^{1-2/v_H}}{C^{v_H-2}},
  \]
  we may conclude that $\Ex[D_{s,k,m}] \le \exp(-(K+v_He_H)s)$, provided that $C$ is sufficiently large. Therefore, if $C$ is sufficiently large,
  \[
    \Ex[D_s(\cH_p)] = \sum_{k=v_H}^{v_Hs} \sum_{m=e_H+1}^{e_Hs} \Ex[D_{s,k,m}] \le s^2v_He_H \cdot \exp(-Ks-v_He_Hs) \le \exp(-Ks).
  \]
  This completes the proof.

\section{Beyond polynomials with nonnegative coefficients}
\label{sec:nonpolynomial}
 Although
Theorem~\ref{thm:packaged} applies only to the case where $X= X(Y)$ is a
polynomial with nonnegative coefficients, the proof can be adapted to
yield a similar result for all nonnegative functions $X\colon \{0,1\}^N \to
\RR_{\geq 0}$. In this case, the degree assumption in
Theorem~\ref{thm:packaged} has to be replaced by a more general assumption on the
`complexity' of $X$.

Given an $I \subseteq \br{N}$ and a $z \in \{0,1\}^N$, we let
\[
  F(I, z) = \{y\in \{0,1\}^N: \text{$y_i = z_i$ for all $i\in I$}\};
\]
we call sets of this from \emph{subcubes}. If $F$ is a subcube, then there is a unique
set $I$ such that $F = F(I, z)$, for some $z$. We can thus define
the \emph{codimension} of $F$ by $\codim F = |I|$.
Given a nonnegative function $X$ on the
hypercube, we define the \emph{complexity} of $X$ to be the smallest integer
$d$ for which it is possible to represent $X$ as a
linear combination with nonnegative coefficients
of indicator functions of subcubes with codimension at most $d$. The complexity
of $X$ is well defined, and at most $N$, since $X = \sum_{z\in \{0,1\}^N}
X(z) \1_{F(\br N,z)}$ is such a linear combination.
Note also that the complexity of every polynomial with nonnegative
coefficients and degree $d$ is at most $d$.

Assume now that $Y$ is a random variable taking values in $\{0,1\}^N$
and that $X = X(Y)$.
Given a subcube $F\subseteq \{0,1\}^N$, we write $\e_F[X] =
\e[X\mid Y\in F]$ for the expectation of $X$ conditioned on $Y \in F$. We
further define $\Phi_X \colon \RR_{\geq 0} \to \RR_{\geq 0}\cup \{\infty\}$ by
\begin{equation}\label{eq:beyond-phi}
  \Phi_X(\delta) = \min\big\{-\log \Pr(Y\in F): \text{$F\subseteq \{0,1\}^N$
  is a subcube with }\e_F[X]\geq (1+\delta)\e[X]\big\}.
\end{equation}
If $X$ is an increasing function of $Y$, this definition coincides 
with our earlier definition of $\Phi_X(\delta)$, because then the minimum
is achieved at a subcube of the form $F(I,\mathbf 1)$, where $\mathbf 1$ is the $N$-dimensional
all-ones vector. One may adapt the proof of Theorem~\ref{thm:packaged} to show the following. (A precise proof of this theorem can be found in~\cite{cohen}, which was written after this work was completed).

\begin{theorem}\label{thm:genpackaged}
  For every positive integer $d$ and all positive real numbers
  $\eps$ and $\delta$ with $\eps<1/2$, there is a positive
  $K=K(d,\eps,\delta)$ such that the following holds. Let $Y$ be a sequence
  of $N$ independent $\Ber(p)$ random variables for some $p \in (0, 1-\eps]$ and
  assume that $X=X(Y)$ is non-negative and has complexity at most $d$ and satisfies
  $\Phi_X(\delta-\eps)\geq K\log(1/p)$. Denote by $\cF^*$ the collection of all
  subcubes $F\subseteq \{0,1\}^N$ satisfying
  \begin{enumerate}[label=(F\arabic*)]
    \item\label{item:genthmcore-bias}
      $\e_F[X] \geq (1+\delta-\eps)\e[X]$,
    \item\label{item:genthmcore-size}
      $\codim F\leq K \cdot \Phi_X(\delta+\eps)$.
  \end{enumerate}
 Then, 
\[
\Pr\big( X \geq (1+\delta)\e[X] \big) \leq (1 + \eps) \cdot \Pr \left(Y \in F \text{ for some } F \in \cF^* \right).
\] 
\end{theorem}

We note that this theorem does not exactly match Theorem~\ref{thm:packaged}; indeed, we only assert that the upper tail event is dominated by the appearance of a subcube in $\cF^*$. It is possible to further restrict the family $\cF^*$, analogous to the extraction of cores from seeds. We do not pursue this direction here.

Theorem~\ref{thm:genpackaged} can be used to study the upper tail problem for
induced subgraph counts.
Suppose that $H$ is a fixed graph and $X =X_{n,p}^{H\text{-ind}}$ is the
number of induced copies of $H$ in the random graph $G_{n,p}$.
Let $N =
\binom{n}{2}$ and,
 for an arbitrary
bijection $\sigma_n\colon \binom{\br n}{2} \to \br N$,
let $Y_i$ be the indicator random variable of the event that $\sigma_n^{-1}(i)$ is an
edge in $G_{n,p}$. Then
 we can write
\[ X = \sum_{\substack{H'\subseteq K_n\\ H'\cong H}}
\prod_{e\in E(H)} Y_{\sigma_n(e)}
\prod_{e\in \binom{V(H)}{2}\setminus E(H)}(1- Y_{\sigma_n(e)}).
 \]
In particular, the complexity of $X$ is bounded
by $e_H$ and Theorem~\ref{thm:genpackaged} applies.
On the other hand, it is clear that $X$ is generally not monotone, so
one cannot use Theorem~\ref{thm:packaged}. This direction was successfully pursued to by Cohen~\cite{cohen}.

\section{Concluding remarks}
\label{sec:conclusion}

    \subsubsection*{Nonregular graphs}
    It is an open problem to extend Theorems~\ref{thm:H} and \ref{thm:Hpoisson}
    to nonregular graphs. 
    It is straightforward to extend Theorem~\ref{thm:Hpoisson} to the more general
    case of strictly balanced graphs;
    however, note that \v{S}ileikis and Warnke
    \cite{vsileikis2018counterexample} constructed balanced graphs
    for which the conclusion of Theorem~\ref{thm:Hpoisson} does not hold.
    In the localised regime, the results in~\cite{chatterjee2016nonlinear,cook2018large,cook2021hypergraph,eldan2016gaussian} apply to arbitrary (as opposed to only $\Delta$-regular) graphs $H$; however, these works require polynomially-suboptimal assumptions on the density $p$. Recently, 
    \v{S}ileikis and Warnke~\cite{sileikis2019stars}
    determined the order of the
    logarithmic upper tail probability for the number of copies
    of the star graph $K_{1,s}$ in $G_{n,p}$.
    
    \subsubsection*{The phase transition between the Poisson and the localised regimes}
      We believe that the logarithmic upper tail probabilities of the random
      variables considered in this paper are always determined by either the Poisson behaviour,
      the localised behaviour, or the coexistence of the two (in the regime where
      they are commensurate). More precisely, we believe that for both 
      $X= X_{n,p}^{H}$ (for a connected, $\Delta$-regular $H$) and $X=X_{N,p}^\kap$,
      \begin{equation}
        \label{eq:shy-conj}
        -\log \Pr\big(X\geq (1+\delta)\e[X]\big) = \big(1\pm o(1)\big)\cdot \min_{0\leq
          \delta'\leq \delta}\big(\Phi_X(\delta') + \Psi_X(\delta-\delta')\big),
      \end{equation}
      as long as $\e[X]\to\infty$ and $p\to 0$.
      Let $p^* = p^*(\delta,n)$ be such that $\Phi_X(\delta)  = \Psi_X(\delta)$.
      Note that if $p \ll p^*$, then $\Psi_X(\delta) \ll \Phi_X(\delta)$ and we recover
      Theorems~\ref{thm:kappoisson} and \ref{thm:Hpoisson}, whereas if $p\gg p^*$, we have 
      $\Psi_X(\delta) \gg \Phi_X(\delta)$ and~\eqref{eq:shy-conj} implies (in some cases
      a stronger version of) Theorems~\ref{thm:kap}, \ref{thm:H}, and~\ref{thm:krrate}.
      If $p= \Theta(p^*)$, then both terms are of the same order and the conjecture
      allows for the upper tail to be dominated by configurations
      exhibiting features of both the Poisson and localised regimes.

      \subsubsection*{Structural theorems for non-complete graphs}
      In the case where $X = X_{n,p}^H$ for a connected, $\Delta$-regular graph $H$,
      we have neither determined the asymptotics of $\Phi_X(\delta)$ in the range $np^\Delta \to c\in (0,\infty)$
      nor given a structural description of the upper tail event $\{X \ge (1+\delta)\Ex[X]\}$ for any density $p$.
      Doing the former would yield the logarithmic upper tail probability of $X$, via Proposition~\ref{prop:ldp-H};
      it is likely that the value of $\Phi_X(\delta)$ is given by a mixture of 
      the clique construction and a `hub-like' construction in which
      a constant number of vertices have degrees linear in $n$.
      As for the latter, the method used to prove Theorem~\ref{thm:Kr-stability} can be generalised to
      yield an analogous statement in which $K_r$ is replaced with an arbitrary $\Delta$-regular graph $H$.
      Armed with such a `stability' statement, it is relatively straightforward to show that, when $np^\Delta \to 0$,
      the random graph $G_{n,p}$ conditioned on the upper tail event $\{X \ge (1+\delta)\Ex[X]\}$ contains
      an `almost-clique' of the `right' size, as was the case when $H = K_r$. We were not able
      to prove such a structural statement in the complementary range $np^\Delta = \Omega(1)$.

      \subsubsection*{Stability results for arithmetic progressions}
    An interesting problem is to characterise the near-minimisers of the optimisation problem
    for $\Phi_X(\delta)$ when $X = X_{N,p}^\kap$. More precisely, we ask for a description
    of all subsets $I\subseteq \br N$ that satisfy $\e_I[X]\geq (1+\delta)\e[X]$ and $|I|\leq (1+\eps)\Phi_X(\delta+\eps)$.
    As a consequence of Theorem~\ref{thm:packaged} and the entropic stability of $X_{N,p}^\kap$, which
    we established in the proof of Proposition~\ref{prop:kap-rate}, such a result would imply
    a structural characterisation of the upper tail event. Since the dominant contribution to
    the difference $\Ex_I[X] - \Ex[X]$ comes from $k$-term arithmetic progressions contained
    in $I$, this problem is equivalent to understanding the structure of sets $I\subseteq \ZZ$ 
    that are near-maximisers of the number of $k$-term arithmetic progressions (among
    subsets of a given size). The structure of true maximisers was described by Green and
    Sisask~\cite{green2008maximal} in the case for $k=3$.
    
    \subsubsection*{Decomposing the upper tail measure}
    Let $\bar{Y}$ be the random variable obtained by conditioning $Y$ on the
    upper tail event $\{X(Y)\geq (1+\delta)\e[X]\}$ and let $\tilde{Y}$ be the random
    variable obtained by first choosing a uniformly random solution $I$ of the
    optimisation problem for $\Phi_X(\delta)$ and then conditioning $Y$ on
    $\prod_{i\in I}Y_i=1$. It would be very interesting to determine necessary
    and sufficient conditions so that $\bar{Y}$ and $\tilde{Y}$ are close in some metric.
    In particular, are the assumptions of Theorem~\ref{thm:packaged} sufficiently
    strong to imply this?
    This question is closely related to the more general problem of decomposing a Gibbs measure into a mixture of product measures. 
    The work of Eldan and Gross~\cite{eldan2018decomposition},
    and, more recently, of Austin~\cite{austin2018structure}, gives
    general conditions for the existence of such a decomposition.

    \subsubsection*{Moderate deviations}
    Throughout this paper, we have assumed that $\delta$ is a fixed, positive constant. It is interesting and natural
    to study the probability of $\{X \ge (1+\delta) \Ex[X]\}$ when $\delta$ is allowed to depend on $N$ and~$p$.
    In the case where $\delta \Ex[X]$ is of the same order as $\sqrt{\Var(X)}$, one can often prove a Central Limit
    Theorem, see~\cite{BarKarRuc89, BarKocLiu19, Ruc88}. The regime in which $\delta \Ex[X] \gg \sqrt{\Var(X)}$ but $\delta \to 0$
    is often referred to as the moderate deviation regime. One expects that, under reasonable assumptions,
    the logarithmic upper tail probability $-\log \Pr(X \ge (1+\delta)\Ex[X])$ is of order $\min\{ (\delta\Ex[X])^2/\Var(X), \, \Phi_X(\delta) \}$;
    this has been verified in certain cases, see~\cite{bhattacharya2016upper,griffiths2019regular,GolGriSco19, griffiths2020irregular, sileikis2019stars,warnke2017upper}.
    Our methods can be adapted to the moderate deviation regime. In an upcoming work~\cite{HarMouSam-kap}, we calculate
    the logarithmic upper tail probability for $X = X_{N,p}^\kap$ for nearly all pairs $(p, \delta)$ for which
    localisation is believed to occur---that is, when $\Phi_X(\delta) \ll (\delta \Ex[X])^2/\Var(X)$.

    \subsubsection*{Other random graph models}
    Upper tails for subgraph counts have been considered in random graph
    models other than $G_{n,p}$, such as exponential random graphs~\cite{chatterjee2013estimating},
    random geometric graphs~\cite{chatterjee2014localization}, random regular graphs~\cite{gunby2020regular, HooLipMos19},
    and (dense) uniform random graphs~\cite{dembo2018large}.
    The framework developed here can be generalised to other (non-product) measures
    on the hypercube, providing a possible approach to such questions.  
    It is likely that this requires adapting the notions of cores and entropic stability
    to the model.

    \subsection*{Acknowledgements}
    We thank Bhaswar Bhattacharya, Asaf Cohen, Nicholas Cook, Ronen Eldan, Michael Krivelevich, Eyal Lubetzky, Matas {\v{S}}ileikis, Lutz Warnke, and Yufei Zhao for helpful comments and discussions. We are indebted to the three anonymous referees whose comments greatly improved the presentation of this paper.

\bibliographystyle{abbrv}
\bibliography{refs}

\begin{thebibliography}{10}

\bibitem{alon1981number}
N.~Alon.
\newblock On the number of subgraphs of prescribed type of graphs with a given
  number of edges.
\newblock {\em Israel J. Math.}, 38(1-2):116--130, 1981.

\bibitem{augeri2019transportation}
F.~Augeri.
\newblock A transportation approach to the mean-field approximation.
\newblock \texttt{arXiv:1903.08021}.

\bibitem{augeri2018nonlinear}
F.~Augeri.
\newblock Nonlinear large deviation bounds with applications to {W}igner
  matrices and sparse {E}rd{\H{o}}s-{R}\'{e}nyi graphs.
\newblock {\em Ann. Probab.}, 48(5):2404--2448, 2020.

\bibitem{austin2018structure}
T.~Austin.
\newblock The structure of low-complexity {G}ibbs measures on product spaces.
\newblock {\em Ann. Probab.}, 47(6):4002--4023, 2019.

\bibitem{BalMorSamWar2016}
J.~Balogh, R.~Morris, W.~Samotij, and L.~Warnke.
\newblock The typical structure of sparse {$K_{r+1}$}-free graphs.
\newblock {\em Trans. Amer. Math. Soc.}, 368(9):6439--6485, 2016.

\bibitem{BarKarRuc89}
A.~D. Barbour, M.~Karo\'{n}ski, and A.~Ruci\'{n}ski.
\newblock A central limit theorem for decomposable random variables with
  applications to random graphs.
\newblock {\em J. Combin. Theory Ser. B}, 47(2):125--145, 1989.

\bibitem{BarKocLiu19}
Y.~Barhoumi-Andr\'{e}ani, C.~Koch, and H.~Liu.
\newblock Bivariate fluctuations for the number of arithmetic progressions in
  random sets.
\newblock {\em Electron. J. Probab.}, 24:Paper No. 145, 32, 2019.

\bibitem{basak2019upper}
A.~Basak and R.~Basu.
\newblock Upper tail large deviations of regular subgraph counts in
  {E}rd{\H{o}}s--{R}{\'e}nyi graphs in the full localized regime.
\newblock \texttt{arXiv:1912.11410}.

\bibitem{bhattacharya2017upper}
B.~B. Bhattacharya, S.~Ganguly, E.~Lubetzky, and Y.~Zhao.
\newblock Upper tails and independence polynomials in random graphs.
\newblock {\em Adv. Math.}, 319:313--347, 2017.

\bibitem{bhattacharya2016upper}
B.~B. Bhattacharya, S.~Ganguly, X.~Shao, and Y.~Zhao.
\newblock Upper tail large deviations for arithmetic progressions in a random
  set.
\newblock {\em Int. Math. Res. Not. IMRN}, (1):167--213, 2020.

\bibitem{bollobas1981threshold}
B.~Bollob{\'a}s.
\newblock Threshold functions for small subgraphs.
\newblock {\em Math. Proc. Camb. Philos. Soc.}, 90(2):197--206, 1981.

\bibitem{bollobas1998random}
B.~Bollob{\'a}s.
\newblock {\em Random graphs}.
\newblock Cambridge University Press, Cambridge, 2nd edition, 2001.

\bibitem{BonMur08}
J.~A. Bondy and U.~S.~R. Murty.
\newblock {\em Graph theory}, volume 244 of {\em Graduate Texts in
  Mathematics}.
\newblock Springer, New York, 2008.

\bibitem{boucheron2013book}
S.~Boucheron, G.~Lugosi, and P.~Massart.
\newblock {\em Concentration inequalities}.
\newblock Oxford University Press, Oxford, 2013.

\bibitem{briet2017gaussian}
J.~Bri\"{e}t and S.~Gopi.
\newblock Gaussian width bounds with applications to arithmetic progressions in
  random settings.
\newblock {\em Int. Math. Res. Not. IMRN}, (22):8673--8696, 2020.

\bibitem{chatterjee2012missing}
S.~Chatterjee.
\newblock The missing log in large deviations for triangle counts.
\newblock {\em Random Structures Algorithms}, 40(4):437--451, 2012.

\bibitem{Cha17}
S.~Chatterjee.
\newblock {\em Large deviations for random graphs}, volume 2197 of {\em Lecture
  Notes in Mathematics}.
\newblock Springer, Cham, 2017.
\newblock Lecture notes from the 45th Probability Summer School held in
  Saint-Flour (\'{E}cole d'\'{E}t\'{e} de Probabilit\'{e}s de Saint-Flour),
  June 2015.

\bibitem{chatterjee2016nonlinear}
S.~Chatterjee and A.~Dembo.
\newblock Nonlinear large deviations.
\newblock {\em Adv. Math.}, 299:396--450, 2016.

\bibitem{chatterjee2013estimating}
S.~Chatterjee and P.~Diaconis.
\newblock Estimating and understanding exponential random graph models.
\newblock {\em Ann. Statist.}, 41(5):2428--2461, 2013.

\bibitem{chatterjee2014localization}
S.~Chatterjee and M.~Harel.
\newblock Localization in random geometric graphs with too many edges.
\newblock {\em Ann. Probab.}, 48(2):574--621, 2020.

\bibitem{chatterjee2011large}
S.~Chatterjee and S.~R.~S. Varadhan.
\newblock The large deviation principle for the {E}rd{\H{o}}s--{R}{\'e}nyi
  random graph.
\newblock {\em European J. Combin.}, 32(7):1000--1017, 2011.

\bibitem{ChuGraFraShe86}
F.~R.~K. Chung, R.~L. Graham, P.~Frankl, and J.~B. Shearer.
\newblock Some intersection theorems for ordered sets and graphs.
\newblock {\em J. Combin. Theory Ser. A}, 43:23--37, 1986.

\bibitem{cohen}
A.~Cohen.
\newblock The upper tail problem for induced $4$-cycles in sparse random
  graphs.
\newblock M.Sc. thesis, 2021.

\bibitem{cook2018large}
N.~Cook and A.~Dembo.
\newblock Large deviations of subgraph counts for sparse
  {E}rd{\H{o}}s-{R}\'{e}nyi graphs.
\newblock {\em Adv. Math.}, 373:107289, 53, 2020.

\bibitem{cook2021hypergraph}
N.~A. Cook, A.~Dembo, and H.~T. Pham.
\newblock Regularity method and large deviation principles for the
  {E}rd{\H{o}}s--{R}\'enyi hypergraph.
\newblock \texttt{arXiv:2102.09100}.

\bibitem{demarco2012tight}
R.~DeMarco and J.~Kahn.
\newblock Tight upper tail bounds for cliques.
\newblock {\em Random Structures Algorithms}, 41(4):469--487, 2012.

\bibitem{dembo2018large}
A.~Dembo and E.~Lubetzky.
\newblock A large deviation principle for the {E}rd{\H{o}}s--{R}{\'e}nyi
  uniform random graph.
\newblock {\em Electron. Commun. Probab.}, 23, 2018.

\bibitem{dembo1998book}
A.~Dembo and O.~Zeitouni.
\newblock {\em Large deviations techniques and applications}, volume~38 of {\em
  Applications of Mathematics}.
\newblock Springer, New York, 2nd edition, 1998.

\bibitem{Die17}
R.~Diestel.
\newblock {\em Graph theory}, volume 173 of {\em Graduate Texts in
  Mathematics}.
\newblock Springer, Berlin, 5th edition, 2017.

\bibitem{eldan2016gaussian}
R.~Eldan.
\newblock Gaussian-width gradient complexity, reverse log-{S}obolev
  inequalities and nonlinear large deviations.
\newblock {\em Geom. Funct. Anal.}, 28(6):1548--1596, 2018.

\bibitem{eldan2018decomposition}
R.~Eldan and R.~Gross.
\newblock Decomposition of mean-field {G}ibbs distributions into product
  measures.
\newblock {\em Electron. J. Probab.}, 23, 2018.

\bibitem{Erd62}
P.~Erd\H{o}s.
\newblock On the number of complete subgraphs contained in certain graphs.
\newblock {\em Magyar Tud. Akad. Mat. Kutat\'{o} Int. K\"{o}zl.}, 7:459--464,
  1962.

\bibitem{griffiths2019regular}
G.~Fiz~Pontiveros, S.~Griffiths, M.~Secco, and O.~Serra.
\newblock Deviation probabilities for arithmetic progressions and other regular
  discrete structures.
\newblock \texttt{arXiv:1910.12835}.

\bibitem{FriKah98}
E.~Friedgut and J.~Kahn.
\newblock On the number of copies of one hypergraph in another.
\newblock {\em Israel J. Math.}, 105:251--256, 1998.

\bibitem{Gal}
D.~Galvin.
\newblock Three tutorial lectures on entropy and counting.
\newblock \texttt{arXiv:1406.7872}.

\bibitem{GolGriSco19}
C.~Goldschmidt, S.~Griffiths, and A.~Scott.
\newblock Moderate deviations of subgraph counts in the
  {E}rd{\H{o}}s-{R}\'{e}nyi random graphs {$G(n,m)$} and {$G(n,p)$}.
\newblock {\em Trans. Amer. Math. Soc.}, 373(8):5517--5585, 2020.

\bibitem{green2008maximal}
B.~Green and O.~Sisask.
\newblock On the maximal number of 3-term arithmetic progressions in subsets of
  {$\mathbb Z/p\mathbb Z$}.
\newblock {\em Bulletin London Math. Soc.}, 40(6):945--955, 2008.

\bibitem{griffiths2020irregular}
S.~Griffiths, C.~Koch, and M.~Secco.
\newblock Deviation probabilities for arithmetic progressions and irregular
  discrete structures.
\newblock \texttt{arXiv:2012.09280}.

\bibitem{gunby2020regular}
B.~Gunby.
\newblock Upper tails of subgraph counts in sparse regular graphs.
\newblock \texttt{arXiv:2010.00658}.

\bibitem{HarMouSam-kap}
M.~Harel, F.~Mousset, and W.~Samotij.
\newblock Upper tails for arithmetic progressions in the moderate deviation
  regime.
\newblock Manuscript.

\bibitem{Har60}
T.~E. Harris.
\newblock A lower bound for the critical probability in a certain percolation
  process.
\newblock {\em Proc. Cambridge Philos. Soc.}, 56:13--20, 1960.

\bibitem{hoeffding1963probability}
W.~Hoeffding.
\newblock Probability inequalities for sums of bounded random variables.
\newblock {\em J. Amer. Statist. Assoc.}, 58(301):13--30, 1963.

\bibitem{janson1990poisson}
S.~Janson.
\newblock Poisson approximation for large deviations.
\newblock {\em Random Structures Algorithms}, 1(2):221--229, 1990.

\bibitem{janson2004upper}
S.~Janson, K.~Oleszkiewicz, and A.~Ruci\'nski.
\newblock Upper tails for subgraph counts in random graphs.
\newblock {\em Israel J. Math.}, 142:61--92, 2004.

\bibitem{janson2002infamous}
S.~Janson and A.~Ruci{\'n}ski.
\newblock The infamous upper tail.
\newblock {\em Random Structures Algorithms}, 20(3):317--342, 2002.

\bibitem{janson2004deletion}
S.~Janson and A.~Ruci{\'n}ski.
\newblock The deletion method for upper tail estimates.
\newblock {\em Combinatorica}, 24(4):615--640, 2004.

\bibitem{janson2011upper}
S.~Janson and A.~Ruci{\'n}ski.
\newblock Upper tails for counting objects in randomly induced subhypergraphs
  and rooted random graphs.
\newblock {\em Ark. Mat.}, 49(1):79--96, 2011.

\bibitem{janson2016lower}
S.~Janson and L.~Warnke.
\newblock The lower tail: Poisson approximation revisited.
\newblock {\em Random Structures Algorithms}, 48(2):219--246, 2016.

\bibitem{karonski1983number}
M.~Karo{\'n}ski and A.~Ruci{\'n}ski.
\newblock On the number of strictly balanced subgraphs of a random graph.
\newblock In {\em Graph Theory, Lágow 1981}, volume 1018 of {\em Lecture Notes
  in Mathematics}, pages 79--83. Springer, 1983.

\bibitem{katona1966theorem}
G.~Katona.
\newblock A theorem of finite sets.
\newblock In {\em Proc. Colloq. Theory of Graphs, Tihany 1966}, pages 187--207.
  Academic Press, New York, 1966.

\bibitem{keevash2008shadows}
P.~Keevash.
\newblock Shadows and intersections: Stability and new proofs.
\newblock {\em Adv. Math.}, 218:1685--1703, 2008.

\bibitem{kim2000concentration}
J.~H. Kim and V.~H. Vu.
\newblock Concentration of multivariate polynomials and its applications.
\newblock {\em Combinatorica}, 20(3):417--434, 2000.

\bibitem{kim2004divide}
J.~H. Kim and V.~H. Vu.
\newblock Divide and conquer martingales and the number of triangles in a
  random graph.
\newblock {\em Random Structures Algorithms}, 24(2):166--174, 2004.

\bibitem{KozSam19}
G.~Kozma and W.~Samotij.
\newblock Lower tails via relative entropy.
\newblock \texttt{arXiv:2104.04850}.

\bibitem{kruskal1963number}
J.~B. Kruskal.
\newblock The number of simplices in a complex.
\newblock In {\em Mathematical Optimization Techniques}, volume~10, pages
  251--278. Univ. of California Press, Berkeley, CA, 1963.

\bibitem{lubetzky2015replica}
E.~Lubetzky and Y.~Zhao.
\newblock On replica symmetry of large deviations in random graphs.
\newblock {\em Random Structures Algorithms}, 47(1):109--146, 2015.

\bibitem{lubetzky2017variational}
E.~Lubetzky and Y.~Zhao.
\newblock On the variational problem for upper tails in sparse random graphs.
\newblock {\em Random Structures Algorithms}, 50(3):420--436, 2017.

\bibitem{mousset2017probability}
F.~Mousset, A.~Noever, K.~Panagiotou, and W.~Samotij.
\newblock On the probability of nonexistence in binomial subsets.
\newblock {\em Ann. Probab.}, 48(1):493--525, 2020.

\bibitem{MukBha18}
S.~Mukherjee and B.~B. Bhattacharya.
\newblock Replica symmetry in upper tails of mean-field hypergraphs.
\newblock {\em Adv. in Appl. Math.}, 119:102047, 25, 2020.

\bibitem{rivin2002counting}
I.~Rivin.
\newblock Counting cycles and finite dimensional {$L^p$} norms.
\newblock {\em Adv. in Appl. Math.}, 29(4):647--662, 2002.

\bibitem{Ruc88}
A.~Ruci\'{n}ski.
\newblock When are small subgraphs of a random graph normally distributed?
\newblock {\em Probab. Theory Related Fields}, 78(1):1--10, 1988.

\bibitem{vsileikis2012upper}
M.~{\v{S}}ileikis.
\newblock On the upper tail of counts of strictly balanced subgraphs.
\newblock {\em Electron. J. Combin.}, 19(1):4, 2012.

\bibitem{sileikis2019stars}
M.~{\v{S}}ileikis and L.~Warnke.
\newblock Upper tail bounds for stars.
\newblock {\em Electron. J. Combin.}, 27(1):Paper No. 1.67, 23, 2020.

\bibitem{talagrand1995concentration}
M.~Talagrand.
\newblock Concentration of measure and isoperimetric inequalities in product
  spaces.
\newblock {\em Publ. Math. Inst. Hautes {\'E}tudes Sci.}, 81(1):73--205, 1995.

\bibitem{HooLipMos19}
P.~van~der Hoorn, G.~Lippner, and E.~Mossel.
\newblock Regular graphs with linearly many triangles.
\newblock \texttt{arXiv:1904.02212}.

\bibitem{vsileikis2018counterexample}
M.~\v{S}ileikis and L.~Warnke.
\newblock A counterexample to the {D}e{M}arco-{K}ahn upper tail conjecture.
\newblock {\em Random Structures Algorithms}, 55(4):775--794, 2019.

\bibitem{Vu00}
V.~H. Vu.
\newblock On the concentration of multivariate polynomials with small
  expectation.
\newblock {\em Random Structures Algorithms}, 16(4):344--363, 2000.

\bibitem{vu2002concentration}
V.~H. Vu.
\newblock Concentration of non-{L}ipschitz functions and applications.
\newblock {\em Random Structures Algorithms}, 20(3):262--316, 2002.

\bibitem{warnke2017upper}
L.~Warnke.
\newblock Upper tails for arithmetic progressions in random subsets.
\newblock {\em Israel J. Math.}, 221(1):317--365, 2017.

\bibitem{Warnke-miss-log}
L.~Warnke.
\newblock On the missing log in upper tail estimates.
\newblock {\em J. Combin. Theory Ser. B}, 140:98--146, 2020.

\bibitem{Zhao17}
Y.~Zhao.
\newblock On the lower tail variational problem for random graphs.
\newblock {\em Combin. Probab. Comput.}, 26(2):301--320, 2017.

\end{thebibliography}

\end{document}